\documentclass[12pt,twoside,a4paper]{amsart}
\usepackage{amssymb}
\usepackage{latexsym}
\usepackage{amsmath}
\usepackage{euscript}
\newcommand{\RR}{\mathbb{R}}
\newcommand{\CC}{\mathbb{C}}

\newcommand{\DD}{\mathbb{D}}

\newcommand{\ZZ}{\mathbb{Z}}
\newcommand{\TT}{\mathbb{T}}
\font\gothic=eufm10

\def\gR{\mbox{\gothic\char'122}}

\def\gb{\mbox{\gothic\char'142}}

\def\filbcirc#1{\hbox{\rlap{\hbox to 25pt {$\bigcirc$\hfill}}
\hbox to 2pt{\hfill \small #1\hfill}}}

\def\skthr{\noalign{\vskip 3pt}}
\def\sksix{\noalign{\vskip 6pt}}

\def\sD{{\mathfrak D}}      \def\sF{{\mathfrak F}}

      \def\sR{{\mathfrak R}}

\def\gh{{\mathfrak h}}
\def\gb{{\mathfrak b}}    \def\gs{{\mathfrak s}}

      \def\dC{{\mathbb C}}
\def\dD{{\mathbb D}}

   \def\dN{{\mathbb N}}

\def\NN{{\mathbb N}}
      
      \def\cF{{\mathcal F}}
   \def\cH{{\mathcal H}}   
   \def\cK{{\mathcal K}}   \def\cL{{\mathcal L}}

\def\cS{{\mathcal S}}      \def\cU{{\mathcal U}}

\newcommand{\IPi}{\hbox{$\,$I\vbox{\moveleft 2pt\hbox{$\Pi$}}}}

\newcommand{{\bx}}{\bf x}

\newcommand{{{\bff}}}{\bf f}

\newcommand{{\bw}}{\bf w}
\newcommand{{\bg}}{\bf g}
\newcommand{{\bb}}{\bf b}
\newcommand{{\ba}}{\bf a}
\newcommand{{\bq}}{\bf q}

\newcommand{\lam}{\lambda}
\newcommand{\om}{\omega}
\newcommand{\ptp}{p\times p}
\newcommand{\ptq}{p\times q}
\newcommand{\qtq}{q\times q}

\newcommand{\qtp}{q\times p}
\newcommand{\mtm}{m\times m}

\newcommand{\dsp}{\displaystyle}
\newcommand{\wtilde}{\widetilde}

\def\h#1{{{\hat #1} }}
\def\wt#1{{{\widetilde #1} }}
\def\wh#1{{{\widehat #1} }}

\def\bm\chi{\mbox{\boldmath$\chi$}}

\def\ran{{\rm rng\,}}

\def\dim{{\rm dim\,}}

\def\rank{{\rm rank\,}}

\let\xker=\ker \def\ker{{\xker\,}}

\newtheorem{thm}{Theorem}[section]
\newtheorem{proposition}[thm]{Proposition}
\newtheorem{corollary}[thm]{Corollary}
\newtheorem{lem}[thm]{Lemma}
\newtheorem{definition}[thm]{Definition}

\theoremstyle{definition}
\newtheorem{example}{Example}
\newtheorem{remark}[thm]{Remark}

\numberwithin{equation}{section}

\begin{document}

\title[Linear fractional transformations]
{On  linear fractional transformations associated with
generalized $J$-inner matrix functions}
\author{Vladimir Derkach}
\address{Department of Mathematics \\
Donetsk State University \\
%Universitetskaya str. 24 \\
%83055
Donetsk \\
Ukraine} \email{derkach.v@gmail.com}
\author{Harry Dym}
\address{Department of Mathematics,  \\
Weizmann Institute, Israel} \email{harry.dym@weizmann.ac.il}
\dedicatory{}
%\date{\today}
\thanks{V. Derkach, Weston Visiting Scholar, wishes to thank the Weizmann Institute of Science
for hospitality and support}
%\translator{}

\begin{abstract}
A class $\cU_{\kappa_1}(J)$ of generalized $J$-inner mvf's (matrix valued
functions) $W(\lam)$ which appear as resolvent matrices for bitangential
interpolation problems in the generalized Schur class of $\ptq$ mvf's
$\cS_\kappa^{\ptq}$ and some associated reproducing kernel Pontryagin spaces
are studied. These spaces are used to describe the range of the linear
fractional transformation $T_W$ based on $W$  and applied to
${\mathcal S}_{\kappa_2}^{\ptq}$. Factorization formulas for mvf's $W$ in a subclass
$\cU_{\kappa_1}^\circ(J)$ of $\cU_{\kappa_1}(J)$ found and then used to
parametrize the set
$\cS_{\kappa_1+\kappa_2}^{\ptq}\cap T_W[\cS_{\kappa_2}^{\ptq}]$.
Applications to bitangential interpolation problems in the class
$\cS_{\kappa_1+\kappa_2}^{\ptq}$ will be presented elsewhere.
\end{abstract}

\maketitle
%%%%%%%%%%%%%%%%%%%%%%%%%%%%%%%%%%%%%%%%%%%%%%%%

\section{Introduction }
Let $J$ be an  $m\times m$ signature matrix (i.e., $J=J^*$ and
$JJ^*=I_m$) and let $\Omega_+$ be equal to either
$\dD=\{\lam\in\dC:\,|\lam|<1\}$ or
$\Pi_+=\{\lam\in\dC:\,\lam+\bar\lam>0\}$. An $m\times m$ mvf (matrix
valued function) $W(\lambda)$ that is meromorphic in $\Omega_+$
belongs to the class $\cU_\kappa(J)$ of {\it generalized $J$-inner}
mvf's if:
\begin{enumerate}
\item[(i)]
the kernel
\begin{equation}\label{kerK}
{\mathsf K}_\omega^W(\lambda)=
\frac{J-W(\lambda)JW(\omega)^*}{\rho_\omega(\lambda)}
\end{equation}
has $\kappa$ negative squares in ${\mathfrak h}_W^+\times{\mathfrak
h}_W^+$ (see definition in Subsection~2.1) and
\item[(ii)]
$J-W(\mu)JW(\mu)^*=0$ a.e. on the boundary $\Omega_0$ of $\Omega_+$;
\end{enumerate}
where ${\mathfrak h}_W^+$ denotes the
domain of holomorphy of $W$ in $\Omega_+$   and
\[
\rho_{\om}(\lam)=\left\{\begin{array}{ll}
                          1-\lam\overline{\om},
& \mbox{ if }\Omega_+=\dD; \\ \sksix
                          2\pi (\lam+\overline{\om}), & \mbox{ if }\Omega_+=\Pi_+. \\
                        \end{array}\right.
\]
Thus, in both cases
\[
\Omega_+=\{\omega\in\dC:\rho_\omega(\omega)>0\}
\] and
$\Omega_0=\{\omega\in\dC:\rho_\omega(\omega)=0\}$ is the boundary of
$\Omega_+$. Correspondingly we set
\begin{equation}
\label{eq:omegaminus} \Omega_- =\dC \setminus (\Omega_+
\cup\Omega_0) = \{\om\in\dC: \rho_{\om}(\om) < 0\}.
\end{equation}

Most of the other notation that we use will be fairly standard:
\medskip

\noindent
mvf for matrix valued function, vvf for vector valued function,
$\ker A$ for the kernel of a matrix $A$, $\ran A$ for
its range, $\sigma (A)$
for its spectrum if $A$ is square, and, if $A=A^*$,
$\nu_-(A)$  for the number of negative eigenvalues (counting multiplicities).
If $f(\lam)$ is a mvf, then
$$
f^{\#}(\lam)=f(\lam^o)^*, \mbox{ where}\quad \lam^o=
\left\{\begin{array}{ll} 1/\overline{\lam}
& \,\mbox{if}\quad\Omega_+=\dD \quad\text{and}\quad
\lam\ne 0; \\ \sksix
-\overline{\lam}  & \,\mbox{ if}\quad \Omega_+=\Pi_+ \\
\end{array}\right.
$$
and
$$
\gh_f=\{\lambda\in\CC\ \textrm{at which $f(\lambda)$ is
holomorphic}\}\quad\textrm{and}\quad \gh_f^\pm=\gh_f\cap\Omega_\pm.
$$
We also set
$$
\langle f, g\rangle_{st}=\left\{\begin{array}{ll}
{\displaystyle\frac{1}{2\pi}
\int_0^{2\pi}g(\mu)^*f(\mu)d\mu}&\quad\text{if}\quad
\Omega_+=\DD;\\ \\
\int_{-\infty}^{\infty}g(i\mu)^*f(i\mu)d\mu&\quad\text{if}\quad
\Omega_+=\Pi_+
\end{array}\right.
$$
and
$$
\widetilde{L}_1^{\ptq}=\left\{\begin{array}{ll} L_1^{\ptq}(\Omega_0)
&\quad\text{if}\quad
\Omega_+=\DD;\\ \\
\{f:\,(1+\vert\mu\vert^2)^{-1}f\in L_1^{\ptq}(\Omega_0)\}&\quad\text{if}\quad
\Omega_+=\Pi_+.
\end{array}\right.
$$
The symbol $H_2^{\ptq}$ (resp., $H_\infty^{\ptq}$)  stands for the class of
$\ptq$ mvf's with entries in the Hardy space $H_2$ (resp., $H_\infty$);
$H_2^p$  is short for $H_2^{p\times 1}$ and $(H_2^p)^\perp$ is the
orthogonal complement of $H_2^p$ in $L_2^p$ with respect to the standard
inner product on $\Omega_0$.

 The   class $\cU_\kappa(J)$ and reproducing kernel Pontryagin spaces
$\cK(W)$ with the reproducing kernel ${\mathsf K}_\omega^W(\lambda)$ based on
$W\in\cU_\kappa(J)$ were studied in ~\cite{AD86} and \cite{ADRS}.
In~\cite{Nud81}, \cite{BH83}, \cite{BGR} and \cite{Der03} mvf's
$W\in\cU_\kappa(J)$ appear as resolvent matrices of  interpolation
problems; in~\cite{Kuzh}, \cite{AI}, \cite{DrR}, \cite{DLS} and
\cite{Der01a} mvf's $W\in\cU_\kappa(J)$ were considered as characteristic
functions of linear operators in indefinite inner product spaces.

Let ${\mathcal S}_{\kappa}^{p\times q}(\Omega_+)$ denote the {\it generalized
Schur class} of mvf's $s$
that are meromorphic in $\Omega_+$ and for which the kernel
\begin{equation}\label{kerLambda}
{\mathsf \Lambda}_\omega^s(\lambda)=
\frac{I_{p}-s(\lambda)s(\omega)^*}{\rho_\omega(\lambda)}
\end{equation}
has $\kappa$ negative squares on ${\mathfrak h}_s^+\times{\mathfrak
H}_s^+$ (see~\cite{KL}).
%%%%%%%%%%%%%%%%%%%%%%%%%%%%%%%%

A fundamental result of  Kre\u{\i}n and Langer~\cite{KL} guarantees that  every
generalized Schur function $s\in
{\mathcal S}_{\kappa}^{\ptq}:={\mathcal S}_{\kappa}^{\ptq}(\Omega_+)$ admits a pair of
coprime factorizations
\begin{equation}\label{KL}
s(\lam)=b_{\ell}(\lam)^{-1}s_{\ell}(\lam)=s_r(\lam)b_r(\lam)^{-1}
\quad\text{for}\, \lam\in{\mathfrak h}_s^+,
\end{equation}
where $b_{\ell}$ and $b_r$ are Blaschke--Potapov products of degree
$\kappa$ and sizes $\ptp$ and $\qtq$, respectively,  and the mvf's
$s_{\ell}$ and $s_r$ both belong to the {\it Schur class} ${\mathcal
S}^{\ptq}:={\mathcal S}_0^{\ptq}(\Omega_+)$. The classes of inner
and outer mvf's in ${\mathcal S}^{\ptp}$ will be denoted  $
{\mathcal S}_{in}^{\ptp}$ and ${\mathcal S}_{out}^{\ptp}$,
respectively.

%%%%%%%%%%%%%%%%%%%%%%%%%%%%%%%%%%%%%%%%%%%%%%%%%%%%%%%%%%
In this paper we fix
$$
J=j_{pq}=\begin{bmatrix}I_p&0\\0&-I_q\end{bmatrix},\quad\text{where}\ p+q=m,
$$
and if $W\in\cU_\kappa(j_{pq})$ is written in block form $W=[w_{ij}]_{i,j=1}^2$
conformally with
$j_{pq}$, then the linear fractional transformation
\begin{equation}\label{eq:0.9}
    T_W[\varepsilon]=(w_{11}\varepsilon+w_{12})(w_{21}\varepsilon+w_{22})^{-1}
\end{equation}
is well defined on
${\mathcal S}_{\kappa_2}^{\ptq}$ for all $\lambda\in\Omega_+$ except for at
most a finite set of
points and $s=T_W[\varepsilon]$ belongs to ${\mathcal S}_{\kappa'}^{\ptq}$ with
$\kappa'\le \kappa_1+\kappa_2$. The main results of this paper are:
\begin{enumerate}
\item[\rm(1)]   A
characterization of the set
\[
  T_W[{\mathcal S}_{\kappa_2}^{\ptq}]=\left\{T_W[\varepsilon]:\,\varepsilon\in
{\mathcal S}_{\kappa_2}^{\ptq}\right\}
\]
that is formulated in terms of the mvf
\[
\Delta_s(\mu):=\left[\begin{array}{cc}
I_p & -s(\mu)\\
-s(\mu)^* & I_q   \end{array} \right]  \quad \text{a.e. on}\ \Omega_0
\]
and a $\kappa$-dimensional operator $\Gamma_r$ that is defined below.
\vskip 6pt
\item[\rm(2)] A parametrization of the intersection of this set with
${\mathcal S}_{\kappa_1+\kappa_2}^{\ptq}$ when $W$ belongs to a
subclass $\cU_{\kappa_1}^\circ(j_{pq})$ of $\cU_{\kappa_1}(j_{pq})$,
introduced below in~\eqref{eq:11.10}.
\end{enumerate}

Let the mvf $s\in {\mathcal S}_{\kappa}^{p\times q}$ admit the
Kre\u{\i}n-Langer factorizations
\eqref{KL} and let
$$
\cH_*(b_\ell):=(H_2^q)^\perp\ominus b_\ell^*(H_2^p)^\perp\quad\textrm{and}
\quad
\cH(b_r)=H_2^q\ominus b_rH_2^q.
$$
Since  $\textup{dim } \cH(b_r)=\kappa$, the operator
\[
    X_r:h\mapsto P_-sh\quad (h\in \cH(b_r)),
\]
is a finite-dimensional operator of rank at most $\kappa$. In fact, as will
be shown below, $X_r$ is a
1-1-isomorphism from $\cH(b_r)$ onto $\cH_*(b_{\ell})$. Let
    \[
     \Gamma_r:L_2^p\ni g\mapsto X_r^{-1}P_{\cH_*(b_{\ell})}g\in\cH(b_r).
    \]

\begin{thm}
\label{thm:0.1} Let $\kappa_1,\kappa_2\in\dN\cup\{0\}$, let $W\in
\cU_{\kappa_1}(j_{pq})$ and let
$s\in {\mathcal S}^{\ptq}_{\kappa_1+\kappa_2}$ admit the
Kre\u{\i}n-Langer factorizations \eqref{KL}. Then
$s\in T_W[{\mathcal S}^{\ptq}_{\kappa_2}]$ if and
only if the following conditions hold:
\begin{enumerate}
    \item [\rm(1)]
    $\begin{bmatrix}
      b_{\ell} & -s_{\ell}
    \end{bmatrix}f\in H_2^p$ for every $f\in
    {\cK}(W)$; \vskip 6pt
    \item [\rm(2)]
    $\begin{bmatrix}
      -s_r^* & b_r^*
    \end{bmatrix} f\in (H_2^q)^\perp$ for every $f\in
    {\cK}\rm(W)$; \vskip 6pt
    \item[\rm(3)]
    $
   {\displaystyle \left\langle \{\Delta_s
     +\Delta_s
     \left[\begin{array}{cc}
       0 & \Gamma_r^* \\
       \Gamma_r & 0 \\
    \end{array}\right]\Delta_s\}f,f\right
    \rangle_{st}\le\langle f,f\rangle_{{\cK}(W)}}
    $ for every $f\in
    {\cK}(W)$.
    \end{enumerate}
\end{thm}
The description of the set $T_W[{\mathcal S}_{\kappa_2}^{\ptq}]$ given in
Theorem~\ref{thm:0.1} is
a generalization to the indefinite setting of a result from~\cite{D8}. The
proof is
based on the theory of the reproducing kernel Pontryagin  spaces $\cK(s)$ and
$\cK(W)$
associated with the kernels ${\mathsf \Lambda}_\omega^s(\lambda)$ and ${\mathsf
K}_\omega^W(\lambda)$ developed in~\cite{ADRS} and~\cite{AD86}, respectively.
For the
convenience of the reader we review and partially extend  the parts of this
theory that
are needed in this paper in Section 2. In particular, we furnish an
indefinite analog of
the de Branges-Rovnayk description (\cite{dBR}) of the space $\cK(s)$ and a
boundary
characterization of an indefinite analog $\sD(s)$ of the de Branges-Rovnyak
reproducing
kernel Hilbert space. In the definite case the left hand side of (3) in Theorem
\ref{thm:0.1} coincides with $\|f\|_{\sD(s)}^2$, which clarifies the
connection of this
result with the setting of the abstract interpolation problem in~\cite{KKhYu}.

%%%%%%%%%%%%%%%%%%%%%%%%%%%%%%%%%%%%%%%%%%%%%%%%%%%%%%%%%%%%%%%%%%%%%%%
For every mvf $W\in\cU_\kappa(j_{pq})$ %is written in block form
$W=[w_{ij}]_{i,j=1}^2$
%conformally with $j_{pq}$, then
the lower right hand $q\times q$ corner
$w_{22}(\lambda)$ of $W(\lambda)$
is invertible for all $\lambda\in\gh_W^+$ except
for at most $\kappa$ points. Thus, the Potapov-Ginzburg transform
$S=PG(W)$ of $W$ is  defined on $\gh_W^+$ by the formula
\begin{equation}\label{PGtrans}
    S(\lambda)=\left[\begin{array}{cc}
      s_{11}(\lambda) & s_{12}(\lambda) \\
      s_{21}(\lambda) & s_{22}(\lambda)
    \end{array}      \right] %PG(W)
    :=\left[\begin{array}{cc}
      w_{11}(\lambda) & w_{12}(\lambda) \\
      0 & I_q
    \end{array}      \right]
    \left[\begin{array}{cc}
            I_p &       0\\
      w_{21}(\lambda) & w_{22}(\lambda)
    \end{array}      \right]^{-1}
\end{equation}
and it belongs to the class ${\mathcal S}_\kappa^{\mtm}$ (as may be
verified by writing ${\mathsf \Lambda}^{S}_\omega(\lam)$ in terms of
${\mathsf K}^W_\omega(\lam)$). Moreover, since
\begin{equation}
\label{eq:jul4a7}
S=PG(W)\Longrightarrow W=PG(S)\,,
\end{equation}
the mvf $W$ is of bounded type. Thus, the nontangential limits $W(\mu)$ exist
a.e. on $\Omega_0$ and  assumption (ii) in the definition of $\cU_\kappa(J)$
makes sense.

Let
\begin{equation}
\label{eq:11.10}
\cU_\kappa^\circ(j_{pq})=\{W\in\cU_\kappa(j_{pq}):\, s_{21}:=-w_{22}^{-1}w_{21}\quad
\text{belongs to}\  {\mathcal S}_\kappa^{\qtp}\}.
\end{equation}
The Kre\u{\i}n-Langer factorizations of  $s_{21}$ will be written as
\begin{equation}\label{eq:0.4}
    s_{21}(\lambda):={\gb}_{\ell}(\lambda)^{-1}{\gs}_{\ell}(\lambda)
={\gs}_r(\lambda){\gb}_r(\lambda)^{-1}\quad\text{for}\ \lambda\in\gh_s^+,
\end{equation}
where ${\gb}_{\ell}$, ${\gb}_{r}$ are Blashke-Potapov products of degree
$\kappa$ and ${\gs}_{\ell},{\gs}_r\in {\mathcal S}^{\qtp}$;
german fonts are used to emphasize
that the factorization is now for a mvf of size $\qtp$.

In Theorem~\ref{thm:11.2} we shall show that the mvf's ${\gb}_{\ell}s_{22}$
and $s_{11}{\gb}_r$ belong to the classes ${\mathcal S}^{\ptp}$ and
${\mathcal S}^{\qtq}$,
respectively. Therefore, they admit inner-outer and outer-inner factorizations
\begin{equation}\label{eq:0.5}
    s_{11}{\gb}_r=b_1\varphi_1,\quad{\gb}_{\ell}s_{22}=\varphi_2b_2,
\end{equation}
where $b_1\in {\mathcal S}_{in}^{\ptp}$, $b_2\in {\mathcal S}_{in}^{\qtq}$, $\varphi_1\in
{\mathcal S}_{out}^{\ptp}$, $\varphi_2\in {\mathcal S}_{out}^{\qtq}$.
In keeping with the usage in \cite{ArovD97} and \cite{ArovD08},
the pair $\{b_1,b_2\}$ is called an
{\it associated pair} of the mvf
$W\in\cU_\kappa^\circ(j_{pq})$ and denoted $\{b_1,b_2\}\in\mbox{ap}(W)$.
If $\kappa=0$, the
formulas in~\eqref{eq:0.5} reduce to the inner-outer
and outer-inner factorizations of $s_{11}$ and $s_{22}$
(see~\cite{Arov93},~\cite{ArovD97}).

In Theorem~\ref{thm:11.8} it will be shown that if
$W\in\cU_\kappa^\circ(j_{pq})$ and   $\{b_1,b_2\}\in\mbox{ap}(W)$, then
there are $H_\infty$ mvf's $K$, $c_\ell$, $d_\ell$, $c_r$, $d_r$,
such that the factorizations
\begin{equation}
\label{eq:0.6}
W=\Theta\,\Phi\quad\text{and}\quad W=\wt{\Theta}\,\wt{\Phi},
\end{equation}
hold over $\Omega_+$ and $\Omega_-$, respectively, with
\begin{equation}
\label{eq:0.7}
\Theta=\begin{bmatrix}b_1&Kb_2^{-1}\\0&b_2^{-1}\end{bmatrix},\quad
\wt{\Theta}=\begin{bmatrix}b_1&0\\K^\# b_1&b_2^{-1}\end{bmatrix}
=j_{pq}\Theta^{-\#}j_{pq},
\end{equation}
\begin{equation}
\label{eq:0.8} \Phi=
\begin{bmatrix}\varphi_1&0\\0&\varphi_2^{-1}\end{bmatrix}
\begin{bmatrix}c_r&d_r\\-\gs_\ell&\gb_\ell\end{bmatrix}\quad\textrm{and}\quad
\quad \wt{\Phi}=
\begin{bmatrix}\varphi_1^{-\#}&0\\0&\varphi_2^\#\end{bmatrix}
\begin{bmatrix}\gb_r^\#&-\gs_r^\#\\ d_\ell^\#&c_\ell^\#\end{bmatrix}.
\end{equation}

If $W\in \cU^\circ_{\kappa_1}(j_{pq})$, then  Theorem~\ref{thm:0.1}
is supplemented
by the following parametrization of the set
$T_W[{\mathcal S}^{\ptq}_{\kappa_2}]
\cap {\mathcal S}^{\ptq}_{\kappa_1+\kappa_2}$ in terms of the parameter $\varepsilon$.
\begin{thm}\label{thm:0.2}
Let the mvf $W\in \cU_{\kappa_1}^\circ(j_{pq})$, $S=PG(W)$, let
$s_{21}$ have the Kre\u{\i}n-Langer factorizations~\eqref{eq:0.4},
$\{b_1, b_2\}\in ap(W)$, %let $\varphi_{21}$, $\varphi_{22}$,
%$\wt\varphi_{11}$, $\wt\varphi_{12}$ be defined
%by~\eqref{eq:0.6}, \eqref{eq:0.8}
and let  $\varepsilon\in
{\mathcal S}^{\ptq}_{\kappa_2}$
satisfy the assumption
\begin{equation}\label{eq:0.10}
(I_q-s_{21}\varepsilon)^{-1}|_{\Omega_0} \in
\wtilde{L}_1^{q\times q}
\end{equation}
and admit the Kre\u{\i}n-Langer factorizations
\[
\varepsilon=\theta_\ell^{-1}\varepsilon_\ell=\varepsilon_r\theta_r^{-1},
\]
where $\theta_\ell$, $\theta_{r}$ are Blashke-Potapov products of degree $\kappa_2$ and
$\varepsilon_{\ell},\varepsilon_r\in {\mathcal S}^{\ptq}$. Then the mvf $s=T_W[\varepsilon]$
belongs to ${\mathcal S}_{\kappa_1+\kappa_2}^{\ptq}$ if and only if  the factorizations
\begin{equation}\label{Reg1}
    \theta_\ell w^\#_{11}+\varepsilon_\ell w^\#_{12}
    =(\theta_\ell \gb_r-\varepsilon_\ell\gs_r)(b_{1}\varphi_1)^{-1},
\end{equation}
\begin{equation}\label{Reg2}
    w_{21}\varepsilon_r+ w_{22}\theta_r=(\varphi_2 b_{2})^{-1}(-\gs_\ell
\varepsilon_r+\gb_\ell\theta_r)
\end{equation}
are coprime over $\Omega_+$.
\end{thm}
The proof of this result is based on the
factorizations~\eqref{eq:0.6}-\eqref{eq:0.8} and an application of
the Kre\u{\i}n-Langer generalization of Rouche's theorem
from~\cite{KL81}. The cases when the assumption~\eqref{eq:0.10} can
be dropped are discussed.

 The paper is organized as follows. In Section
\ref{preli} the basic notions of left and right coprime
factorizations are introduced. Their connection with the
Kre\u{\i}n-Langer factorizations of generalized Schur functions is
discussed. The theory of reproducing kernel Pontryagin spaces,
associated with a generalized Schur function $s$ from~\cite{ADRS} is
reviewed and extended.  In Section~\ref{Sec3} we prove the first main
result of the paper: Theorem~\ref{thm:0.1}, which characterizes the
range of the linear fractional transformation $T_W$ associated with
the mvf $W\in \cU_{\kappa_1}(j_{pq})$.  In Section~\ref{Sec4} we
 obtain factorization
formulas for mvf's $W\in \cU_\kappa^\circ(j_{pq})$ and use them to
characterize the corresponding
reproducing kernel Pontryagin spaces $\cK(W)$.   A parametrization of the set
$T_W[{\mathcal
S}^{\ptq}_{\kappa_2}]\cap {\mathcal S}^{\ptq}_{\kappa_1+\kappa_2}$
is given in Theorem~\ref{thm:0.2} for $W\in\cU_{\kappa_1}^\circ(j_{pq})$.

\section{Preliminaries \label{preli}}

\subsection{The generalized Schur class}
Recall that a Hermitian kernel  ${\mathsf
K}_\omega(\lambda):\Omega\times\Omega\to\dC^{m\times m}$ is said to
have $\kappa $ negative squares
if for every positive integer $n$ and every choice of $\omega_j\in\Omega$
and $u_j\in\dC^m$
$(j=1,\dots,n)$ the matrix
\[
\left(\left<{\mathsf
K}_{\omega_j}(\omega_k)u_j,u_k\right>\right)_{j,k=1}^n
\]
has at most $\kappa$ negative eigenvalues and for some choice of
$\omega_1,\ldots,\omega_n\in\Omega$ and $u_1,\ldots,u_n\in\dC^m$ exactly
$\kappa$ negative
eigenvalues. In this case we write
\[
\mbox{sq}_-{\mathsf K}=\kappa.
\]

The class ${\mathcal S}^{\ptq}:={\mathcal S}_{0}^{\ptq}(\Omega_+)$ is the
usual
Schur class.
Recall that a mvf $s\in {\mathcal S}^{\ptq}$ is called inner
(resp., $*$-inner), if
$s(\mu)$ is an isometry (resp., a co-isometry) for a.e.
$\mu\in\Omega_0$, that is
\[
I_q-s(\mu)^*s(\mu)=0 \quad (\mbox{resp., } I_p-s(\mu)s(\mu)^*=0), \quad \mu\in\Omega_0\,
(a.e.).
\]
Let ${\mathcal S}_{in}^{p\times q}$ (resp., ${\mathcal S}_{*in}^{p\times q}$)
denote the set of all
inner (resp., $*$-inner) mvf's $s\in {\mathcal S}^{\ptq}$. An example of
an inner square mvf is provided by the finite {\it Blaschke--Potapov product},
that in the case of the unit
disc ($\Omega_+=\dD$) is given by
\begin{equation}\label{BPprod}
b(\lam)=\prod_{j=1}^{\kappa}b_j(\lam),\quad
b_j(\lam)=I-P_j+\frac{\lam-\alpha_j}{1-\bar\alpha_j \lam}P_j,
\end{equation}
where $\alpha_j\in\dD$ and  $P_j$ is an orthogonal projection in $\dC^p$ for
$j=1,\dots ,n$. The factor $b_j$ is called {\it simple} if $P_j$
has rank one. The representation of $b(\lam)$ as a
product of simple Blaschke-Potapov factors is not unique. However, the number
$\kappa$ of such simple factors is
the same for every representation~\eqref{BPprod}. It is called the
{\it degree} of the Blaschke--Potapov product
$b(\lam)$~\cite{Pot}.

A theorem of  Kre\u{\i}n and Langer~\cite{KL} guarantees that
every generalized Schur
function $s\in {\mathcal S}_{\kappa}^{\ptq}(\Omega_+)$ admits a factorization
of the form
\begin{equation}\label{KLleft}
s(\lam)=b_{\ell}(\lam)^{-1}s_{\ell}(\lam) \quad \text{for}\
\lam\in\mathfrak{h}_s^+,
\end{equation}
where $b_{\ell}$ is a Blaschke--Potapov product of degree $\kappa$,
$s_{\ell}$ is in the Schur class ${\mathcal S}^{\ptq}(\Omega_+)$ and
\begin{equation}\label{KLcanon}
\ker s_{\ell}(\lam)^*\cap \ker
b_{\ell}(\lam)^*=\{0\}\quad\text{for}\ \lam\in\Omega_+.
\end{equation}
The representation~\eqref{KLleft} is called a {\it left
Kre\u{\i}n--Langer factorization}. The constraint~\eqref{KLcanon}
can be expressed in the equivalent form
\begin{equation}\label{KLcanon2}
\rank \left[\begin{array}{cc}
 b_{\ell}(\lam) & s_{\ell}(\lam)
\end{array}\right]
=p\quad (\lam\in\Omega_+).
\end{equation}
If $\alpha_j\in\dD$ $(j=1,\dots ,n)$ are all the zeros of $b_{\ell}$
in $\Omega_+$, then the noncancellation condition~\eqref{KLcanon}
ensures that ${\mathfrak
h}_s^+=\Omega_+\setminus\{\alpha_1,\dots,\alpha_n\}$. The left
Kre\u{\i}n--Langer factorization~\eqref{KLleft} is essentially
unique in a sense that $b_{\ell}$ is defined uniquely up to a left
unitary factor $V\in\dC^{\ptp}$.

Similarly, every generalized Schur
function $s\in {\mathcal S}_{\kappa}^{\ptq}(\Omega_+)$ admits a {\it right
Kre\u{\i}n-Langer factorization}
\begin{equation}\label{KLright}
s(\lam)=s_r(\lam)b_r(\lam)^{-1}\quad \text{for}\ \lam\in{\mathfrak h}_s^+,
\end{equation}
where $b_r$ is a Blaschke--Potapov product of degree $\kappa$,
$s_r\in {\mathcal S}^{\ptq}(\Omega_+)$ and
\begin{equation}\label{KRcanon}
\ker s_r(\lam)\cap \ker b_r(\lam)=\{0\}\quad\text{for}\  \lam\in\Omega_+.
\end{equation}
This condition can be rewritten in the equivalent form
\begin{equation}\label{KRcanon2}
\rank \left[\begin{array}{cc}
 b_r(\lam)^* & s_r(\lam)^*
\end{array}\right]  =q\quad (\lam\in\Omega_+).
\end{equation}
Under assumption~\eqref{KRcanon} the mvf $b_r$ is uniquely
defined up to a right unitary factor $V^\prime\in\CC^{\qtq}$.

%%%%%%%%%%%%%%%%%%%%%%%%%%%%%%%%%%%%%%%%%%%%%%%%%%%%%%%%%%%%%%%%%%%
\begin{lem}\label{Corona}
A mvf $s_\ell\in\cS^{p\times q}$ and a finite Blaschke-Potapov
product $b_\ell\in {\mathcal S}_{in}^{p\times p}$ meet the rank condition (\ref{KLcanon2}),
if and only if there exists a pair of mvf's $c_\ell\in
H^{p\times p}_\infty$ and $d_\ell\in H^{q\times p}_\infty$ such that
\begin{equation}\label{CorFormula}
b_\ell(\lam)c_\ell(\lam)+s_\ell(\lam)d_\ell(\lam)=I_p\quad \text{for}\ \lam\in
\Omega_+.
\end{equation}
\end{lem}
\begin{proof} This is a matrix version of the Carleson Corona theorem. The
proof is adapted from Fuhrmann \cite{Fuhr68} who treated the
square block case. We sketch the details for the convenience of
the reader.

Let
$$
a(\lam)=[a_1(\lam)\ \cdots \ a_m(\lam)]=[b_\ell(\lam)\ s_\ell(\lam)]
$$
be the $p\times m$ matrix with columns $a_i(\lam),\ i=1,\ldots,m$, $m=p+q$,
and let
$$
\alpha_{i_1\ldots i_p}(\lam)=\det[a_{i_1}(\lam)\ \cdots \
a_{i_p}(\lam)]
$$
for every choice of positive integers $i_1,\ldots,i_p$ that meet
the constraint $1\leq i_1<i_2<\cdots <i_p\leq m$. Then, since
$b_\ell(\lam)$ is a finite Blaschke product, there exists a
positive number $r<1$ such that
$$
|\det b_\ell(\lam)| = |\det[a_1(\lam)\ \cdots \ a_p(\lam)]|\geq
\frac{1}{2}
$$
for $r\leq |\lam|\leq 1$. Thus, it is readily checked that there
exists a $\delta>0$ such that
$$
\sum_{i_1\ldots i_p}|\alpha_{i_1\ldots i_p}(\lam)|\geqslant
\delta>0
$$
for every point $\lam\in \Omega_+$. Therefore, by the scalar Corona
theorem \cite{Duren}, there exists a  set of functions
$\beta_{i_1\ldots i_p}(\lam)$ in $H_\infty$ such that
$$
\sum_{i_1\ldots i_p}\alpha_{i_1\ldots i_p}(\lam) \beta_{i_1\ldots
i_p}(\lam)=1.
$$
Now expand each of the determinants $ \alpha_{i_1\ldots i_p}$
along its $i^{th}$ row and express the resulting equality as
$$
a_{i1}(\lam)b_{1i}(\lam)+\dots  +a_{im}(\lam)b_{mi}(\lam)=1
$$
and observe that
$$
a_{k1}(\lam)b_{1i}(\lam)+\dots + a_{km}(\lam)b_{mi}(\lam)=0
$$
if $k\not= i$, because this expression may be obtained by
replacing the $i^{th}$ row in each of the determinants
$\alpha_{i_1\ldots i_p}$ by its $k^{th}$ row. The proof is
completed by setting
$$
\begin{bmatrix}
c_\ell(\lam)\\ \skthr d_\ell(\lam)\end{bmatrix}= \left[
\begin{array}{lcl}
  b_{11}(\lam) & \cdots & b_{1p}(\lam) \\
 \vdots &  & \vdots \\
 b_{m1}(\lam) &\cdots  &b_{mp}(\lam)
\end{array}
\right].
$$
\end{proof}

A dual statement for Lemma~\ref{Corona} is obtained by applying
Lemma~\ref{Corona} to transposed vvf's.
\begin{lem}\label{Corona1}
A mvf $s_r\in\cS^{p\times q}$ and a finite Blaschke-Potapov
product $b_r\in {\mathcal S}_{in}^{q\times q}$ meet the rank condition (\ref{KRcanon2}),
if and only if there exists a pair of mvf's $c_r\in
H^{q\times q}_\infty$ and $d_r\in H^{q\times p}_\infty$ such that
\begin{equation}\label{CorFormula1}
c_r(\lam)b_r(\lam)+d_r(\lam)s_r(\lam)=I_q\quad \text{for}\ \lam\in
\Omega_+.
\end{equation}
\end{lem}
The factorization~\eqref{KLleft} is called a {\it left coprime} factorization of $s$ if $s_\ell$ and
$b_\ell$ satisfy~\eqref{CorFormula}.
Similarly,  the factorization~\eqref{KLright} is called a {\it right coprime} factorization of $s$
if $s_r$ and $b_r$ satisfy~\eqref{CorFormula1}.

Every vvf $h(\lambda)$ from $H_2^p$ $({H_2^p}^\perp)$  has
nontangential limits a.e. on the boundary $\Omega_0$.  These nontangential
limits identify the vvf $h$
uniquely. In what follows we often identify a vvf $h\in H_2^p$
$({H_2^p}^\perp)$ with its boundary values $h(\mu)$.

Let
$P_+$ and $P_-$ denote the orthogonal projections from $L_2^k$ onto $H_2^k$
and ${H_2^k}^\perp$, respectively, where $k$ is a positive integer that will
be understood from the context. The Hilbert spaces
\begin{equation}
\label{eq:feb10a8}
\cH(b_r)=H_2^q\ominus b_rH_2^q, \quad \cH_*(b_\ell):=(H_2^p)^\perp\ominus
b_\ell^*(H_2^p)^\perp
\end{equation}
and the operators
\begin{equation}
\label{eq:feb10b8}
X_r: h\in\cH(b_r)\mapsto P_-sh\quad\text{and}\quad
X_\ell: h\in\cH_*(b_\ell)
\mapsto P_+s^*h
\end{equation}
based on $s\in {\mathcal S}_{\kappa}^{\ptq}$ will play an important role.

\begin{lem}\label{Ker*}{\rm(cf.~\cite{Der01})}
If $s\in {\mathcal S}_{\kappa}^{\ptq}$ and its two factorizations are given by
(\ref{KLleft}) and (\ref{KLright}), then:
\begin{enumerate}
\item[(1)]  The operator $X_\ell$ maps
$\cH_*(b_\ell)$ injectively onto $\cH(b_r)$. \vskip 6pt
\item[(2)]
 The operator $X_r$ maps $\cH(b_r)$ injectively onto $\cH_*(b_{\ell})$.
\vskip 6pt
\item[(3)] $X_\ell=X_r^*$.
\end{enumerate}
\end{lem}
\begin{proof} If $h\in \cH_*(b_{\ell})$, $h_+\in H_2^q$  and
$f=b_rh_+$,
it  is readily checked that
\[
\langle P_+s^*h,f\rangle_{st} =\langle
s^*h,b_rh_+\rangle_{st}
=\langle h,s_rh_+\rangle_{st}=0\quad \forall\  h_+ \in
H_2^q,
\]
i.e., $X_\ell$ maps $\cH_*(b_\ell)$ into $\cH(b_r)$.
Therefore, since $\cH_*(b_\ell)$ and $\cH(b_r)$ are finite
dimensional spaces of the same dimension, and
\[
\dim\cH_*(b_\ell)=\dim\ker X_\ell+\dim\mbox{range } X_\ell,
\]
it suffices to show that $\textup{ker }X_\ell=\{0\}$. But, if $h\in
\cH_*(b_\ell)$ and  $P_+s^* h=0$, then $b_\ell h\in H_2^p$ and, in
view of Lemma~\ref{Corona},
\begin{eqnarray*}
  \xi^* (b_\ell h)(\omega) &=&\langle b_\ell h,\frac{\xi}{\rho_\omega}\rangle_{st}
   = \langle b_\ell h,(b_\ell c+s_\ell d)\frac{\xi}{\rho_\omega}\rangle _{st}\\
  & =& \langle b_\ell h,s_\ell d\frac{\xi}{\rho_\omega}\rangle _{st}
  =\langle P_+s^* h, d\frac{\xi}{\rho_\omega}\rangle_{st}=0
\end{eqnarray*}
for every choice of $\omega\in \Omega_+$, and $\xi\in \dC^p$. Since
$b_\ell(\omega)\not\equiv 0$, this implies that $h(\omega)\equiv 0$.

 Statement (ii) can be obtained by similar calculations, and (iii) is easy.
 \end{proof}

%%%%%%%%%%%%%%%%%%%%%%%%%%%%%%%%%%%%%%%%%%%%%%%%%%%%%%%%%%%%%%%%%%%%%%%%%%%%%%%%%%%%%%%%%%%%%%%%%%

\begin{definition}\label{GammaS}  Let
    \begin{equation}\label{GammaS2}
    \Gamma_{\ell} : f\in L_2^q \to  X_{\ell}^{-1}P_{\cH(b_r)}f\in\cH_*(b_{\ell})\quad
\text{and}\quad
    \Gamma_r: g\in L_2^p\to X_r^{-1}P_{\cH_*(b_{\ell})}g\in \cH(b_r),
\end{equation}
where $X_{\ell}$ and $X_r$ are defined in formula
(\ref{eq:feb10b8}).
\end{definition}
It is readily checked that
\begin{equation}\label{GfGg}
 P_+s^*\Gamma_{\ell} f=P_{\cH(b_r)}f\quad\textrm{and}\quad
P_-s\Gamma_rg=P_{\cH_*(b_{\ell})}g,
 \quad f\in L_2^q,\,g\in L_2^p.
\end{equation}
\begin{remark}\label{rem:Ker*2}
If $f_1\in\cH_*(b_{\ell})$ and $f_2\in\cH(b_r)$, then, in view of~ Lemma~\ref{Ker*},
\[
P_+s^*f_1=P_{\cH(b_r)}f\Longleftrightarrow f_1=\Gamma_{\ell} f,
\]
and
\[
P_-sf_2=P_{\cH_*(b_{\ell})}f \Longleftrightarrow  f_2=\Gamma_rf.
\]
\end{remark}

    \begin{lem}\label{Ker*2}
The operators  $\Gamma_{\ell}$ and $\Gamma_r$ satisfy the
equalities:
\begin{enumerate}
\item[\rm(1)]
$\Gamma_{\ell}^*=\Gamma_r$; \vskip 6pt
\item[\rm(2)]
  $\langle
(\overline{\psi}\Gamma_{\ell}-\Gamma_{\ell}\overline{\psi})f,
g\rangle_{st}=0$ for $f\in\cH(b_r)$, $g\in\cH_*(b_{\ell})$ and
$\psi$ a scalar inner function.
\end{enumerate}
\end{lem}
\begin{proof} (1)
Let $f\in L_2^q$ and $g\in L_2^p$. Then, since $X_{\ell}$ maps
$\cH_*(b_\ell)$ onto $\cH(b_r)$ and  $X_{\ell}^*=X_r$,
$$
\left<\Gamma_{\ell}f,g\right>_{st}=
\left<X_{\ell}^{-1}P_{\cH(b_r)}f,P_{\cH_*(b_\ell)}g\right>_{st}
    =\left<P_{\cH(b_r)}f, X_r^{-1}P_{\cH_*(b_{\ell})}g\right>_{st}
    =\left<f, \Gamma_r g\right>_{st}.
$$
(2) If $f\in\cH(b_r)$ and $g\in\cH_*(b_{\ell})$, then
\[
\begin{split}
\langle
(\overline{\psi}\Gamma_{\ell}-\Gamma_{\ell}\overline{\psi})f,
g\rangle_{st}&= \langle
\overline{\psi}\Gamma_{\ell}f,P_-s\Gamma_rg\rangle_{st} -\langle
P_+s^*\Gamma_{\ell}f,
\psi\Gamma_rg\rangle_{st}\\ \sksix
&= \langle \overline{\psi}\Gamma_{\ell}f,s\Gamma_rg \rangle_{st}
-\langle s^*\Gamma_{\ell}f, \psi\Gamma_rg\rangle_{st}=0.
\end{split}
\]
\end{proof}
\subsection{Reproducing kernel Pontryagin spaces}
In this subsection we review some facts and notation from \cite{AI,Bo} on
the theory of
indefinite inner product spaces for the convenience of the reader. A linear
space $\cK$
equipped with a sesquilinear form $\left<\cdot,\cdot\right>_\cK$ on
$\cK\times\cK$ is
called an indefinite inner product space. A subspace $\sF$ of $\cK$ is
called positive
(negative) if $\left<f,f\right>_\cK>0\,(<0)$ for all $f\in\cF$, $f\ne 0$. If
the full
space $\cK$ is positive and complete with respect to the norm
$\|f\|=\left<f,f\right>_\cK^{1/2}$ then it is a Hilbert space.

An indefinite inner product space $(\cK,\left<\cdot,\cdot\right>_\cK)$ is
called a Pontryagin space, if it can be decomposed as the
orthogonal sum
\begin{equation}\label{Pspace}
    \cK=\cK_+\oplus\cK_-
\end{equation}
of a positive subspace $\cK_+$ which is a Hilbert space and a
negative subspace $\cK_-$ of finite dimension. The number
$\mbox{ind}_-\cK :=\dim\cK_-$ is referred to as the negative index
of $\cK$. The convergence in a Pontryagin space
$(\cK,\left<\cdot,\cdot\right>_\cK)$ is meant with respect to the
Hilbert space norm
\begin{equation}\label{Pnorm}
    \|h\|^2=\left<h_+,h_+\right>_\cK-\left<h_-,h_-\right>_\cK,\quad
    h=h_++h_-,\quad h_\pm\in\cK_\pm.
\end{equation}
It is easily seen that the convergence does not depend on a choice of
the decomposition~\eqref{Pspace}.

A Pontryagin space $(\cK,\left<\cdot,\cdot\right>_\cK)$ of
$\dC^m$-valued functions defined on a subset $\Omega$ of $\dC$ is called
a reproducing kernel Pontryagin space if there exists a
Hermitian kernel ${\mathsf
K}_\omega(\lambda):\Omega\times\Omega\to\dC^{m\times m}$ such that
\begin{enumerate}
\item[\rm(1)] for every $\omega\in\Omega$ and every $u\in\dC^m$ the vvf
${\mathsf K}_\omega
(\lambda)u$ belongs to $\cK$; \vskip 6pt
\item[\rm(2)] for every $h\in\cK$, $\omega\in\Omega$ and $u\in\dC^m$ the
following identity holds
\begin{equation}\label{RKprop}
    \left<h,{\mathsf K}_\omega u\right>_\cK=u^*f(\omega).
\end{equation}
\end{enumerate}

It is known (see~\cite{Sch}) that for every Hermitian kernel
${\mathsf K}_\omega
(\lambda):\Omega\times\Omega\to\dC^{m\times m}$ with a finite number of
negative squares
on $\Omega\times\Omega$ there is a unique Pontryagin space $\cK$ with
reproducing kernel
${\mathsf K}_\omega (\lambda)$, and that $\mbox{ind}_-\cK =\mbox{sq}_-{\mathsf
K}=\kappa$. In the case $\kappa=0$ this fact is due to Aronszajn~\cite{Ar50}.

Let $\cK$, $\cK_1$ be Pontryagin spaces and let $A:\cK_1\to\cK$ be
a continuous linear operator. The adjoint $A^*:\cK\to\cK_1$ is
defined by the equality
\[
\left<A^*h,g\right>_{\cK_1}=\left<h,Ag\right>_{\cK},\quad \mbox{for
all }h\in\cK,g\in\cK_1.
\]
The operator $A$ is called a contraction if
\begin{equation}\label{Contr}
\left<Ag,Ag\right>_\cK\le\left<g,g\right>_{\cK_1} \quad \mbox{for
all }g\in\cK_1;
\end{equation}
$A$ is called an isometry if equality prevails
 in~\eqref{Contr}; it is called a coisometry  if
 $A^*:\cK\to\cK_1$ is an isometry.

A  subspace $\cK_1$ of a Pontryagin space $\cK$ is said to be contained
contractively in $\cK$, if the inclusion mapping $\imath:\cK_1\to\cK$ is a
contraction, i.e.
\[
\langle g,g\rangle_\cK\le \langle g,g\rangle_{\cK_1} \mbox{ for all }g\in\cK_1.
\]
$\cK_1$
is said to be contained isometrically in $\cK$ if the inclusion
$\imath:\cK_1\to\cK$  is an isometry.

A subspace $\cK_2$ of a Pontryagin space $\cK$ is said to be {\it
complementary} to the subspace $\cK_1$ in the sense of de~Branges
if:
\begin{enumerate}
\item[\rm(1)]
Every $h\in \cK$ can be decomposed as
\begin{equation}\label{dBcompl}
    h=h_1+h_2,\quad h_1\in\cK_1,\,\,h_2\in\cK_2.
\end{equation}
\item[\rm(2)] The inequality
\begin{equation}\label{dBcompl2}
    \left<h,h\right>_\cK\le\left<h_1,h_1\right>_{\cK_1}+\left<h_2,h_2
\right>_{\cK_2}
\end{equation}
holds for every decomposition $h=h_1+h_2$, with $ h_1\in\cK_1$ and
$h_2\in\cK_2$. \vskip 6pt
\item[\rm(3)]
There is a unique decomposition~\eqref{dBcompl} such that
equality prevails in~\eqref{dBcompl2}.
\end{enumerate}
If the subspaces $\cK_1$ and $\cK_2$ are complementary in $\cK$ in
the sense of de~Branges, then
$\mbox{ind}_-\cK=\mbox{ind}_-{\cK_1}+\mbox{ind}_-{\cK_2}$ and {\it
the overlapping space} $\sR=\cK_1\cap\cK_2$ of $\cK_1$ and $\cK_2$
is a Hilbert space with respect to the inner product
\begin{equation}\label{Overlap}
  \left<h,g\right>_\sR=\left<h,g\right>_{\cK_1}+\left<h,g\right>_{\cK_2},\quad
     h,g\in\sR.
\end{equation}

The following theorem is due to L.~de~Branges~\cite{dB88}.
\begin{thm}\label{deBrCompl}
Let $\cK$, $\cK_1$ be Pontryagin spaces such that $\cK_1$ is
contained contractively in $\cK$. Then there is exactly one subspace
$\cK_2$ of $\cK$ that is complementary to $\cK_1$ in the sense of
de~Branges. Moreover, the following statements are equivalent:
\begin{enumerate}
\item[\rm(1)]
$\cK_2$ is the orthogonal complement of $\cK_1$ in $\cK$; \vskip 6pt
\item[\rm(2)] $\cK_1\cap\cK_2=\{0\}$; \vskip 6pt
\item[\rm(3)]
$\cK_1$ is contained isometrically in $\cK$.
\end{enumerate}
\end{thm}

These notions come into play when the reproducing kernel ${\mathsf
K}_\omega (\lambda)$ of a Pontryagin space is decomposed as a sum
of two such kernels
\begin{equation}\label{SumK12}
    {\mathsf K}_\omega(\lambda)={\mathsf K}_\omega^{(1)}(\lambda)+{\mathsf
    K}_\omega^{(2)}(\lambda).
\end{equation}
The following theorem is a paraphrase of Theorem~1.5.5 from~\cite{ADRS}.
\begin{thm}\label{SUMK12}{\rm(\cite[Theorem~1.5.5]{ADRS}.)}
Let ${\mathsf K}_\omega(\lambda)$, ${\mathsf
K}_\omega^{(1)}(\lambda)$, ${\mathsf K}_\omega^{(2)}(\lambda)$  be
$\dC^{m\times m}$- valued Hermitian kernels on $\Omega\times \Omega$
with finite negative squares such that \eqref{SumK12} holds. If
$\cK$, $\cK_1$, $\cK_2$ are the corresponding reproducing kernel
Pontryagin spaces, then
\begin{equation}\label{sqK12}
\textup{sq}_-{\mathsf K}\le\textup{sq}_-{\mathsf
K^{(1)}}+\textup{sq}_-{\mathsf K}^{(2)}.
\end{equation}
Moreover the following assertions are equivalent:
\begin{enumerate}
\item[\rm(1)]
Equality prevails in~\eqref{sqK12};\vskip 6pt
\item[\rm(2)]
 $\cK_1$ and $\cK_2$ are contained contractively in
$\cK$ as complementary spaces in the sense of de~Branges; \vskip 6pt
\item[\rm(3)]
the overlapping space $\sR=\cK_1\cap\cK_2$ is a Hilbert space with
respect to the inner product \eqref{Overlap}.
\end{enumerate}
\end{thm}

A multiplicative version of this statement is formulated below for
kernels of the form
\begin{equation}\label{FactK}
    {\mathsf K}_\omega(\lambda)=R(\lambda){\mathsf
    K}_\omega^{(1)}(\lambda)R(\omega)^*.
\end{equation}
\begin{thm}\label{FACTK} {\rm(see~\cite[Theorem~1.5.7]{ADRS}.)}
Let $R(\lambda)$ be a $\dC^{m\times n}$-valued function on $\Omega$,
let ${\mathsf K}_\omega(\lambda)$, ${\mathsf
K}_\omega^{(1)}(\lambda)$ be $\dC^{m\times m}$ and $\dC^{n\times n}$
valued Hermitian kernels on $\Omega\times \Omega$, respectively,
with finite negative squares and let $\cK$, $\cK_1$ be the
corresponding reproducing kernel Pontryagin spaces. Then
\begin{equation}\label{sqKK1}
\textup{sq}_-{\mathsf K}\le\textup{sq}_-{\mathsf K^{(1)}}.
\end{equation}
Equality prevails in~\eqref{sqKK1} if and only if the
multiplication by $R(\lambda)$ is a coisometry from $\cK_1$ to
$\cK$, whose kernel is a Hilbert space.
\end{thm}
%%%%%%%%%%%%%%%%%%%%%%%%%%%%%%%%%%%%%%%%%%%%%%%%%%%%%%%%%%%%

Next we consider some examples of reproducing kernel Hilbert spaces
associated with a mvf  $s\in
{\mathcal S}^{\ptq}(\Omega_+)$.

\begin{example}\label{dBspace}
Let $\cH(s)$ be the reproducing kernel Hilbert space associated
with the kernel ${\mathsf \Lambda}_\omega^s(\lambda)$ on
$\mathfrak{H}_s^+\times\mathfrak{H}_s^+$.
The following description of $\cH(s)$ is due to de Branges and Rovnyak,~\cite{dBR}: A vvf $f\in H_2^p$ belongs to $\cH(s)$ if and only if
\begin{equation}\label{dBR2}
    \alpha(f):=\sup\{ \|f+s\varphi\|_{st}^2-\|\varphi\|_{st}^2:\,\varphi\in{
    H}_2^q\}<\infty.
\end{equation}
Moreover, if $f\in \cH(s)$,  then
\[
\|f\|_{\cH(s)}^2=\alpha(f).
\]
If  $s\in
{\mathcal S}_{in}^{\ptq}(\Omega_+)$, then
\[
\cH(s)=H_2^p\ominus sH_2^p\quad\textrm{and}\quad \Vert f\Vert_{\cH(s)}=
\Vert f\Vert_{st}.
\]
\end{example}
\begin{example}\label{dBspace1}
Let $\cH_*(s)$ be the reproducing kernel Hilbert space associated with the
kernel
\begin{equation}\label{KernelL}
    {\mathsf  L}_\om^s(\lambda)=\frac{I_q-s^{\#} (\lambda) s^{\#} (\om)^*}
{-\rho_\om
    (\lambda)} \quad\text{on}\ \Omega_{s^{\#}}\times \Omega_{s^{\#}}.
\end{equation}
The space $\cH_*(s)$ admits the following description: A vvf $f\in
H_2^q(\Omega_+)^\perp$ belongs to $\cH_*(s)$ if and only if
\begin{equation}
    \beta(f):=\sup\{ \|f+s^\#\varphi\|_{st}^2-\|\varphi\|_{st}^2:\,
\varphi\in H_2^p(\Omega_+)^\perp\}<\infty.
\end{equation}
If $f\in \cH_*(s)$, then
\[
\|f\|_{\cH_*(s)}^2=\beta(f).
\]
If $s\in
{\mathcal S}_{*in}^{\ptq}(\Omega_+)$, then
\[
\cH_*(s)=H_2^q(\Omega_+)^\perp\ominus s^\#H_2^p(\Omega_+)^\perp.
\]
\end{example}
In the following example the reproducing kernel space $\cK(s)$ is negative.
\begin{example}
If  $b$ is a Blaschke-Potapov product of degree $\kappa$, then
$s=b^{-1}\in {\mathcal S}_{\kappa}^{p\times p}$ and it follows from
Theorem~\ref{FACTK} and the identity
\[
{\mathsf
    \Lambda}_\omega^{b^{-1}}(\lambda)=b(\lam)^{-1}({-\mathsf
    \Lambda}_\omega^{b}(\lambda))b(\om)^{-*}
\]
that the space $\cK(b^{-1})$ corresponding to the kernel ${\mathsf
    \Lambda}_\omega^{b^{-1}}(\lambda)$
coincides with $b^{-1}\cH(b)$ as
a set. %, i.e. $f\in\cK(b^{-1})$ if and only if $bf\in\cH(b)$.
Since the inner product $\langle f,f\rangle_{\cK(b)^{-1}}$ is
negative in $\cK(b^{-1})$,
\[
\langle f,f\rangle_{\cK(b^{-1})}=-\langle
bf,bf\rangle_{\cH(b)}=-\langle f,f\rangle_{st},\quad
f\in{\cK(b^{-1})},
\]
$\cK(b^{-1})$ is called the antispace of the Hilbert space
$b^{-1}\cH(b)$.
\end{example}

%%%%%%%%%%%%%%%%%%%%%%%%%%%%%%%%%%%%%%%%%%%%%%%%
\subsection{The spaces $\cK(s)$ and $\cK_*(s)$}
If $s\in{\mathcal S}_{\kappa}^{\ptq}(\Omega_+)$, then the reproducing
kernel $\Lambda_\omega^s(\lambda)$ of the reproducing kernel Pontryagin space $\cK(s)$ can   be expressed in terms of the
right Kre\u{i}n-Langer factorization (\ref{KLright})  of $s\in {\mathcal
S}^{\ptq}_\kappa$ as
\begin{equation}\label{eq:1.1}
{\mathsf \Lambda}^s_\omega(\lambda)={\mathsf
\Lambda}^{s_r}_\omega(\lambda)+ s_r(\lambda){\mathsf
\Lambda}_\omega^{b_r^{-1}}(\lambda)s_r(\omega)^* ={\mathsf
\Lambda}^{s_r}_\omega(\lambda)- s(\lambda){\mathsf
\Lambda}_\omega^{b_r}(\lambda)s(\omega)^*,
\end{equation}
which leads to fundamental decomposition of the Pontryagin space
$\cK(s)$. The following theorem extends Theorem 4.2.3 of \cite{ADRS}.

\begin{thm}\label{thm:1.10}%(cf.~\cite[Theorem 4.2.3]{ADRS})
Let $s\in {\mathcal S}^{\ptq}_\kappa$ have Kre\u{i}n-Langer factorization
(\ref{KLright}) $s=s_rb_r^{-1}$. Then:
\begin{enumerate}
    \item [\rm(1)] The space $\cK(s)$ admits the orthogonal
    decomposition
    \begin{equation}\label{eq:1.2}
    \cK(s)=\cH(s_r)\oplus s_r  \cK(b^{-1}_r).
    \end{equation}
    \item [\rm(2)] If $f\in \cK(s)$, then
\begin{equation}\label{eq:1.3a}
    f-s\Gamma_rf\in\cH(s_r),\quad s\Gamma_r f\in s_r\cK(b^{-1}_r),
\end{equation}
    \begin{equation}\label{eq:1.3}
   f=(f-s\Gamma_r f)+s\Gamma_rf,
\end{equation}
and
\begin{equation}\label{eq:1.4}
   \langle
   f,f\rangle_{\cK(s)}=\|f-s\Gamma_rf\|^2_{\cH(s_r)}-\|\Gamma_rf\|^2_{st}.
   \end{equation}
\end{enumerate}
\end{thm}
\begin{proof}
If $f\in\cH(s_r)\cap s_r\cK(b^{-1}_r)$ and $f=s_r h$ for some
$h\in \cK(b^{-1}_r)$, then,  since the factorization (\ref{KLright}) is
coprime, it follows from
Lemma~\ref{Corona1} that there exists a pair of mvf's
$c\in H^{q\times q}_\infty,\ d\in H^{q\times p}_\infty$ such that
\begin{equation}\label{eq:1.5A}
   cb_r+ds_r=I_q.
\end{equation}
Hence
$$
h=(cb_r) h+d(s_rh)=c(b_rh)+df\in H^q_2.
$$
Since $h\in\cK(b_r^{-1})\subset (H^q_2)^\perp$ this implies that
$h\equiv 0$, and hence $f\equiv 0$, i.e.,
    \begin{equation}\label{eq:1.5}
  \cH(s_r)\cap s_r\cK(b^{-1}_r)=\{0\}.
\end{equation}
The space $\cK(s)$ can be identified with the set $\cK$ of vvf's
\begin{equation}\label{eq:Ks}
f(\lam)=
\left[%
\begin{array}{cc}
  I_p & -s(\lam) \\
\end{array}%
\right]
 \begin{bmatrix}f_1(\lam)\\f_2(\lam)\end{bmatrix}
\quad\text{with}\quad f_1\in\cH(s_r)\quad\text{and}\quad
f_2\in\cH(b_r)
\end{equation}
endowed with the indefinite inner product
\begin{equation}
\label{eq:ipk} \langle f, f\rangle_{\cK}=\langle f_1,
f_1\rangle_{\cH(s_r)}-\langle f_2, f_2\rangle_{st},
\end{equation}
which is correctly defined in view of~\eqref{eq:1.5}. In particular,
\begin{equation}
\label{eq:Lu} {\mathsf \Lambda}^s_\omega u=\left[%
\begin{array}{cc}
  I_p & -s \\
\end{array}%
\right]
\begin{bmatrix}{\mathsf \Lambda}^{s_r}_\omega u \\
{\mathsf \Lambda}^{b_r}_\omega s(\omega)^*u\end{bmatrix} \quad
\text{for every $\omega\in\gh_s^+$ and $u\in\CC^p$}
\end{equation}
and
\begin{eqnarray*}
\langle f,{\mathsf \Lambda}^s_\omega u\rangle_{\cK}&=& \langle f_1, {\mathsf
\Lambda}^{s_r}_\omega u
\rangle_{\cH(s_r)}-\langle f_2, {\mathsf \Lambda}^{b_r}_\omega s(\omega)^*u\rangle_{st}\\
&=&u^*f_1(\omega)-u^*s(\omega)f_2(\omega)=u^*f(\omega)
\end{eqnarray*}
for every $f\in\cK$. Since $\cK$ is a Pontryagin space with respect
to the indefinite inner product (\ref{eq:ipk}), it coincides with
$\cK(s)$.

Next, the decomposition (\ref{eq:Ks}) implies that $P_-f=P_{\cH_*(b_\ell)}f$
for every $f\in\cK(s)$. Therefore, by another application of  (\ref{eq:Ks}),
\[
P_{\cH_*(b_\ell)}f=P_-f=-P_-sf_2=-X_rf_2,
\]
i.e.,
$$
f_2=-X_r^{-1}P_{\cH_*(b_\ell)}f=-\Gamma_r f\quad
\textrm{belongs to $\cH(b_r)$  for every}\ f\in\cK(s)
$$
and
\[
    f-s\Gamma_r f=f+sf_2=f_1\quad \textrm{belongs to}\  \cH(s_r).
\]
This proves the second statement.
\end{proof}

\begin{corollary}
The space $\cK(s)$ consists of vvf's $f\in H^p_2\oplus\cH_*(b_\ell)$
such that
\begin{equation}\label{eq:1.11}
\sup\{\|f-s\Gamma_r
f-s_r\varphi\|^2_{st}-\|\varphi\|^2_{st}:\,\varphi\in
H_2^q\}<\infty.
\end{equation}
Moreover,  if $f\in\cK(s)$, then
\begin{equation}
\label{eq:jul3a7} \langle f,f\rangle_{\cK(s)}=\sup\{\|f-s\Gamma_r
f-s_r\varphi\|^2_{st}-\|\varphi\|^2_{st}:\,\varphi\in
H_2^q\}-\|\Gamma_rf\|^2_{st}\,.
\end{equation}
\end{corollary}

Let $\cK_*(s)$ designate  the reproducing kernel Pontryagin space with kernel
$$
{\mathsf  L}^s_\omega(\lambda)
=\frac{I_q-s^\#(\lambda)s^\#(\omega)^*}{-\rho_\omega(\lambda)}\,.
$$
The left Kre\u{\i}n-Langer
factorization (\ref{KLleft}) yields the representation
\begin{equation}\label{eq:1.15}
{\mathsf L}^s_\omega(\lambda)={\mathsf
L}^{s_\ell}_\omega(\lambda)+s^\#_\ell(\lambda){\mathsf
L}^{b_\ell}_\omega(\lambda)s^\#_\ell(\omega)^*
   \quad \text{on} \ \gh^-_{s^\#}\times \gh^-_{s^\#},
\end{equation}
which leads to a dual version of Theorem \ref{thm:1.10}:

\begin{thm}\label{thm:1.12}
If $s\in {\mathcal S}^{\ptq}_\kappa$ and its left Kre\u{i}n-Langer
factorization (\ref{KLleft}) is $s=b_\ell^{-1}s_\ell$, then:
\begin{enumerate}
    \item [\rm(1)] the space $\cK_*(s)$ admits the fundamental
    decomposition
    \begin{equation}\label{eq:1.16}
    \cK_*(s)=\cH_*(s_\ell)\oplus s_\ell^*  \cK(b^{-1}_\ell);
    \end{equation}
    \item [\rm(2)]for every $f\in \cK_*(s)$ its orthogonal
    decomposition corresponding to (\ref{eq:1.16}) takes the form
    \begin{equation}\label{eq:1.17}
   f=(f-s^*\Gamma_\ell f)+s^*\Gamma_\ell f,
\end{equation}
where $f-s^*\Gamma_\ell f\in\cH_*(s_\ell),\ s^*\Gamma_\ell f\in
s_\ell^*\cK(b^{-1}_\ell)$ and
\begin{equation}\label{eq:1.18}
   \langle
   f,f\rangle_{\cK_*(s)}=\|f-s^*\Gamma_\ell f\|^2_{\cH_*(s_\ell)}-\|\Gamma_\ell f\|^2_{st}.
   \end{equation}
\end{enumerate}
\end{thm}

\begin{corollary}
The space $\cK_*(s)$ consists of vvf's $f\in
(H^q_2)^\perp\oplus\cH(b_r)$ such that
\begin{equation}\label{eq:1.11a}
\sup\{\|f-s^*\Gamma_\ell
f-s^*_\ell\varphi\|^2_{st}-\|\varphi\|^2_{st}:\,\varphi\in
   (H^p_2)^\perp\}-\|\Gamma_\ell f\|^2_{st}<\infty.
\end{equation}
For every $f\in\cK_*(s)$ the inner product $\langle
f,f\rangle_{\cK_*(s)}$ coincides with the left hand side
of~\eqref{eq:1.11a}.
\end{corollary}

%%%%%%%%%%%%%%%%%%%%%%%%%%%%%%%%%%%%%%%%%%%%%%%%%%%%%%%%%%%%%%%%%%%%%%%%%%
Another decomposition of the reproducing kernel Pontryagin space
$\cK(s)$ can
be based on the
representation
\begin{equation}\label{Lambdecom}
    {\mathsf { \Lambda}}_\omega^s(\lambda)=b_{\ell}^{-1}(\lam){\mathsf
    \Lambda}_\omega^{s_{\ell}}(\lambda)b_{\ell}(\om)^{-*}+{\mathsf
    \Lambda}_\omega^{b_{\ell}^{-1}}(\lambda)
    \quad\text{on}\ \mathfrak{h}_s^+\times\mathfrak{h}_s^+.
\end{equation}
In view of Theorems~\ref{SUMK12}, \ref{FACTK} this leads to the
following representation of $\cK(s)$ %(see~\cite{AD86})
\begin{equation}\label{Hsdecom}
    \cK(s)=b_{\ell}^{-1}\cH(s_{\ell})\dotplus\cK(b_{\ell}^{-1}),
\end{equation}
where $\cH(s_{\ell})$ is a Hilbert space and $\cK(b_{\ell}^{-1})$ is
the antispace of  the Hilbert space $b_{\ell}^{-1}\cH(b_{\ell})$.
The decomposition (\ref{Hsdecom}) is not necessarily an orthogonal
decomposition in $\cK(s)$ (see e.g., \cite[p.148]{ADRS} for an
example). However, if $s_{\ell}\in\cS_{in}^{\ptq}$, then
$$
\Vert f\Vert_{\cH(s_\ell)}=\Vert f\Vert_{st}
$$
and the decomposition (\ref{Hsdecom}) is orthogonal.

\begin{thm}\label{Hsinner}
If $s \in {\mathcal S}_{\kappa}^{\ptq}$ and its left
Kre\u{\i}n-Langer factorization (\ref{KLleft}) is
$b_\ell^{-1}s_\ell$ and $s_{\ell} \in {\mathcal S}_{in}^{\ptq}$,
then:
\begin{enumerate}
\item[\rm(1)]
the decomposition~\eqref{Hsdecom} of $\cK(s)$ is orthogonal; \vskip 6pt
\item[\rm(2)]
$f\in\cK(s)$ if and only if
\begin{equation}\label{eq:1.35a}
    b_{\ell}f\in H_2^p\quad \textrm{and}\quad b_r^*s^*f\in(H_2^q)^\perp.
\end{equation}
\end{enumerate}
\end{thm}
\begin{proof}
By Lemma \ref{Corona},
there exists a pair of mvf's $c_\ell$ and $d_\ell$ with entries in
$H_\infty$ such that
$b_{\ell}c_\ell+s_{\ell}d_\ell=I_p$. Thus, if
$f\in\cH(s_{\ell})\cap\cH(b_{\ell})$,
then
\[
\langle f,f\rangle_{st}=\langle f,b_{\ell}c_{\ell}f\rangle_{st}+\langle
f,s_{\ell}d_{\ell}f\rangle_{st}=0,
\]
and, hence, $f=0$, i.e., $\cH(s_{\ell})\cap\cH(b_{\ell})=\{0\}$. Therefore,
the decomposition~\eqref{Hsdecom} is orthogonal, by
Theorem~\ref{deBrCompl}.

If $f\in\cK(s)$ then it follows from~\eqref{Hsdecom} that $f$
admits the decomposition
\begin{equation}\label{eq:1.36}
    f=b_{\ell}^{-1}h+b_{\ell}^{-1}x,
\end{equation}
where $h\in\cH(s_{\ell})$ and $x\in\cH(b_{\ell})$. Therefore
\[
b_r^*s^*f=b_r^*(s_{\ell}^*h)+s_r^*(b_{\ell}^*x).
\]
Since $s_{\ell}\in {\mathcal S}_{in}^{p\times q}$ and $b_{\ell}\in
{\mathcal S}_{in}^{p\times p}$ one obtains
$s_{\ell}^*h\in(H_2^q)^\perp$, %$b_{\ell}^*h\in(H_2^q)^\perp$. In view of Remark~\ref{Hin*in}
$b_{\ell}^*x\in\cH_*(b_{\ell})\subset(H_2^q)^\perp$. Thus
$b_r^*s^*f\in(H_2^q)^\perp$.

Conversely, assume that~\eqref{eq:1.35a} holds. Then
\[
s^*f\in b_r(H_2^q)^\perp=(H_2^q)^\perp\oplus\cH(b_r).
\]
In view of~\eqref{GfGg} $\Gamma_\ell(s^*f)$ satisfies the equality
\[
P_+(s^*f-s^*\Gamma_\ell(s^*f))=0
\]
and hence
\[
h=b_\ell(f-\Gamma_\ell(s^*f))\in\cH(s_{\ell}),\quad
x=b_\ell\Gamma_\ell(s^*f)\in\cH(b_{\ell}).
\]
Therefore, $f=b_{\ell}^{-1}h+b_{\ell}^{-1}x\in\cK(s)$.
\end{proof}
Similarly to the Example~\ref{dBspace} the right Kre\u{\i}n--Langer
factorization~\eqref{KLright} leads to the following  decomposition
of the reproducing kernel Pontryagin space $\cK_*(s)$
(see~\cite[Theorem~4.2.3]{ADRS})
\begin{equation}\label{Hsdecom*}
    \cK_*(s)=b_r^{-\#}\cH_*(s_r)\dotplus \cK_*(b_r^{-1}),
\end{equation}
where $\cH_*(s_r)$ is a Hilbert space  and $\cK_*(b_r^{-1})$ is
the antispace of a Hilbert space $b_r^{-\#}\cH_*(b_r)$.

\begin{corollary}\label{Hs*inner}
Let a mvf $s \in {\mathcal S}_{\kappa}^{\ptq}$ admit the Kre\u{\i}n-Langer
factorization (\ref{KLright}) with $s_r \in {\mathcal S}_{*in}^{\ptq}$. Then:
\begin{enumerate}
\item[(1)]
the decomposition~\eqref{Hsdecom*} of $\cK_*(s)$ is orthogonal; \vskip 6pt
\item[(2)]
$f\in\cK_*(s)$ if and only if
\begin{equation}\label{eq:1.350}
    b_r^*f\in (H_2^q)^\perp,\quad b_{\ell}sf\in H_2^p.
\end{equation}
\end{enumerate}
\end{corollary}
Combining Theorem~\ref{Hsinner}  and Corollary~\ref{Hs*inner} one
obtains
\begin{corollary}\label{Hsin*in}
Let a mvf $s \in {\mathcal S}_{\kappa}^{m\times m}$ admit the
Kre\u{\i}n-Langer factorization (\ref{KLleft}) with $s_{\ell} \in {\mathcal
S}_{in}^{m\times m}$. Then:
\[
\cK_*(s)=s^*\cK(s).
\]
\end{corollary}

\subsection{The de Branges-Rovnyak space $\sD(s)$.}
The symbol  $A^{[-1]}$ stands for the Moore-Penrose pseudoinverse of the matrix $A$,
$$
dm(\mu)=\left\{\begin{array}{l}\frac{1}{2\pi}d\mu\quad\text{if}\ \Omega_0=\TT
\\ \\
d\mu\quad\text{if}\ \Omega_0=i\RR\end{array}\right.
$$
and
\begin{equation}\label{DeltaS}
\Delta_s(\mu):=\left[\begin{array}{cc}
I_p & -s(\mu)\\
-s(\mu)^* & I_q   \end{array} \right],  \quad \text{a.e. on}\  \Omega_0\quad
\text{for}\ s\in {\mathcal S}_{\kappa}^{\ptq}.
\end{equation}
\begin{definition}\label{WHDS}
    Let a mvf $s\in {\mathcal S}_{\kappa}^{\ptq}$ admit left and right Kre\u{\i}n-Langer
factorizations~\eqref{KLleft} and~\eqref{KLright}. Define $\sD(s)$
as the set of vvf's $ f(t)=\textup{col}\, (f^+(t), f^-(t))$, such
that  the following conditions hold:
\begin{enumerate}
\def\labelenumi{\rm (\roman{enumi})}
\item[(${\sD}1$)]
$b_{\ell} f^+\in H_2^p$;
\item[(${\sD}2$)]
$b_r^* f^-\in {H_2^q}^\perp$;
\item[(${\sD}3$)]
$ f(t)$ belongs to the range of the matrix $\Delta_s(\mu)$ $\Omega_0$-a.e. and the
following integral
\[
\int_{\Omega_0}f(\mu)^*\Delta_s(\mu)^{[-1]} f(\mu)dm(\mu)%<\infty.
\]
converges.
\end{enumerate}
 The inner product in $\sD(s)$ is defined by
\begin{equation}\label{whDsInner}
  \left< f, g\right>_{\sD(s)}=\int_{\Omega_0}g(\mu)^*\left(\Delta_s(\mu)^{[-1]}+\left[\begin{array}{cc}
     0 & \Gamma_r^* \\
     \Gamma_r & 0
    \end{array}
    \right]\right) f(\mu)dm(\mu),
\end{equation}
where $\Gamma_r$ is the operator from ${L_2^p}$ onto
$\cH(b_r)$ that is defined in~\eqref{GammaS2}.
\end{definition}

\begin{remark}\label{rem:1.16}
If $\kappa=0$, then $b_\ell=I_p$, $b_r=I_q$, $\Gamma_r=0$ and  $\sD(s)$
coincides with the Hilbert space introduced in \cite{dBR}.
\end{remark}%\end{enumerate}

\begin{remark}\label{rem:1.17}
If $\kappa>0$, then  $\sD(s)$ is an indefinite inner
product space, and, since $\Gamma_\ell= \Gamma_\ell P_{{\cH}(b_r)}$
maps the $\kappa$ dimensional space ${\cH}(b_r)$ bijectively onto
${\cH}_*(b_\ell)$ and
$$
\wt\Gamma\,:=\begin{bmatrix}0 &\Gamma_r^*\\ \Gamma_r &
0\end{bmatrix}=\begin{bmatrix}0 &\Gamma_{\ell}\\ \Gamma_{\ell}^* &
0\end{bmatrix}=
\begin{bmatrix}\Gamma_{\ell} & 0\\ 0 & P_{\cH(b_r)}\end{bmatrix}
\begin{bmatrix}0 & I_{\cH(b_r)}\\ I_{\cH(b_r)} & 0\end{bmatrix}
\begin{bmatrix}\Gamma_{\ell}^* & 0\\ 0 & P_{\cH(b_r)}\end{bmatrix}\,,
$$
$\nu_-(\wt\Gamma)=\kappa$. Therefore,
\begin{equation}\label{IndD}
    \textup{ind}_-\sD(s)\le\kappa,
\end{equation}
since the operator $\Delta_s^{[-1]}$ is nonnegative. In fact
equality prevails in~\eqref{IndD}, as will be shown in Lemma
\ref{lem:1.15}.
\end{remark}
\begin{remark}
\label{rem:jul3a7}
Definitions~\ref{GammaS} and \eqref{GfGg} imply that if
$f\in\sD(s)$, then
\begin{equation}\label{eq:1.28}
\left(I_m+\Delta_s\begin{bmatrix}
    0    & \Gamma_{\ell}  \\
  \Gamma_r & 0
\end{bmatrix}\right)f=\begin{bmatrix}
    I_p-s \Gamma_r  \\
  \Gamma_r
\end{bmatrix}f^++\begin{bmatrix}
     \Gamma_{\ell}  \\
  I_q-s^*\Gamma_{\ell}
\end{bmatrix}f^-,
\end{equation}
where
\begin{equation}\label{eq:1.29}
  \begin{bmatrix}
    I_p-s \Gamma_r  \\
  \Gamma_r
\end{bmatrix}f^+\in H_2^m,\quad
\begin{bmatrix}
     \Gamma_{\ell}  \\
  I_q-s^*\Gamma_{\ell}
\end{bmatrix}f^-\in(H_2^m)^\perp.
\end{equation}
\end{remark}
\bigskip

%%%%%%%%%%%%%%%%%%%%%%%%%%%%%%%%%%%%%%%%%%%%%%%%%%%%%%
Now let $s\in {\mathcal S}_{\kappa}^{p\times q}$, and let the mvf $\Phi_\omega$
be defined on $\Omega_0$ by the formula
\begin{equation}\label{eq:1.30}
    \Phi_\omega=\left\{
    \begin{array}{lll}
      {\displaystyle\frac{1}{\rho_\omega}
      \Delta_s\begin{bmatrix}
      I_p\\ \skthr
      s(\omega)^*
      \end{bmatrix}}&\textrm{if}& \omega\in\gh^+_s; \\
{\displaystyle \frac{-1}{\rho_\omega}
      \Delta_s\begin{bmatrix}
      s^\#(\omega)^*\\
    I_q
      \end{bmatrix}}&\textrm{if} & \omega\in\gh^-_{s^\#},
        \end{array}\right.
\end{equation}
and let
\begin{equation}\label{eq:1.31}
  \sD^+=\text{span}\ \{\Phi_\alpha u:\alpha\in\gh^+_s,\ u\in\CC^p\},
\quad  \sD^-=\text{span}\ \{\Phi_\beta v:\beta\in\gh^-_{s^\#},\ v\in\CC^q\},
\end{equation}
where $\text{span}$ stands for the set of finite linear
combinations. Then $\sD^\pm\subset\sD(s)$.

Indeed, if $\alpha\in \gh^+_s$ and $u\in \dC^p$, then the vvf
$f=\Phi_\alpha u=\textup{col}\,(f^+(\mu), f^-(\mu))$, where
\begin{equation}\label{eq:1.33}
   f^+(\mu)=\frac{I_p-s(\mu)s(\alpha)^*}{\rho_{\alpha}(\mu) }
u\quad\textrm{and}\quad
      f^-(\mu)=-\frac{s^*(\mu)-s(\alpha)^*}{{\rho_{\alpha}(\mu)}}u,
\end{equation}
have meromorphic continuations to $\gh^+_s$ and $\gh^-_s$,
respectively:
$$
f^+{(\lambda)}=\frac{I_p-s(\lambda)s(\alpha)^*}{\rho_{\alpha}(\lambda)}u,
\quad\textrm{and}\quad
f^-{(\lambda)}=-\frac{s^\#(\lambda)-s(\alpha)^*}{\rho_{\alpha}(\lambda)}u.
$$
It follows from (\ref{KLleft}), (\ref{KLright}) that
$$
b_\ell f^+\in H^p_2,\quad b^\#_r f^-\in H^{q\perp}_2.
$$
Since $f(\mu)$ belongs to the range of $\Delta_s(\mu)$, it is readily
checked that
$({\sD}3$) also holds and hence that $f=\Phi_\alpha u\in\sD(s)$, if
$\alpha\in \gh^+_s$
and $u\in \CC^p$.

If $\beta\in \gh^-_{s^\#}$ and $v\in\CC^q$, then the vvf $f=\Phi_\beta v=
\textup{col}\,(f^+ , f^-)$, where
\begin{equation}\label{eq:1.34}
f^+(\mu)=\frac{s(\mu)-s^\#(\beta)^*}{\rho_\beta(\mu)} v,\quad
f^-(\mu)=-\frac{I_q-s(\mu)^*s^\#(\beta)^*}{\rho_\beta(\mu)} v.
\end{equation}
Similar observations show that $f=\Phi_\beta v\in \sD(s)$ for
$\beta\in \gh^-_{s^\#}$ and $v\in\CC^q$.

\begin{lem}\label{lem:1.15}
Let $s\in {\mathcal S}^{\ptq}_\kappa(\kappa\in \NN)$, and let $\sD^+$
and $\sD^-$ be the subspaces of $\sD(s)$ defined by (\ref{eq:1.31}). Then:
\begin{enumerate}
    \item [\rm(1)] $\textup{span }\{\sD^+,\sD^-\}$ is dense in $\sD(s)$;
    \item [\rm(2)]for every choice of
\begin{equation}\label{eq:1.35}
\begin{split}
     \alpha_1,\ldots,\alpha_n\in \gh^+_s,&\quad  u_1,\ldots,u_n\in \CC^p,\\
     \beta_1,\ldots,\beta_m\in \gh^-_{s^\#}, &\quad v_1,\ldots,v_m\in \CC^q,
   \end{split}
\end{equation}
the Gram matrix of the system of vvf's
$
\{\Phi_{\alpha_j}u_j, \Phi_{\beta_k}v_k,\ 1\leqslant j\leqslant
n,\ 1\leqslant k\leqslant m\}
$
\begin{equation}\label{eq:1.36a}
G=\begin{bmatrix}
{\dsp \langle\Phi_{\alpha_i}u_i,\Phi_{\alpha_j}u_j\rangle_{\sD(s)}} &
{\dsp \langle\Phi_{\alpha_i}u_i,\Phi_{\beta_l}v_l\rangle_{\sD(s)}}\\
&\\
{\dsp \langle\Phi_{\beta_k}v_k,\Phi_{\alpha_j}u_j\rangle_{\sD(s)}} &
{\dsp \langle\Phi_{\beta_k}v_k,\Phi_{\beta_l}v_l\rangle_{\sD(s)}}
\end{bmatrix}
\end{equation}
takes the form
\begin{equation}\label{GramLL}
G=\begin{bmatrix} {\dsp u_j^*{\mathsf
    \Lambda}_{\alpha_i}(\alpha_j)u_i }&
{\dsp v_l^*\frac{s(\alpha_i)^*-s^\#(\beta_l)}{\rho_{\alpha_i}(\beta_l)}u_i}\\
{\dsp
u_j^*\frac{s(\alpha_j)-s^\#(\beta_k)^*}{\rho_{\beta_k}(\alpha_j)}v_k}
& {\dsp v_l^*{\mathsf
    L}_{\beta_k}(\beta_l)v_k}
\end{bmatrix};
\end{equation}
  \item [\rm(3)]the negative index of $\sD(s)$ is equal $\kappa$
    \begin{equation}\label{eq:1.37}
   ind_-\sD(s)=\kappa.
   \end{equation}
\end{enumerate}
\end{lem}
\begin{proof}
If $f\in \sD(s)$, $\alpha\in \gh^+_s$ and $u\in\CC^p$, then, in view of
Remark \ref{rem:jul3a7},
\begin{equation}\label{eq:1.42}
\begin{split}
\langle f,\Phi_\alpha u\rangle_{\sD(s)}&=\langle (I_m+\Delta_s\left[
    \begin{array}{cc}
    0&\Gamma_\ell\\
    \Gamma_r & 0
    \end{array}
    \right])
    f,\frac{1}{\rho_\alpha}\left[{{u}\atop{s(\alpha)^* u}}\right])
\rangle_{st}\\
&=\langle(I_p-s\Gamma_r)f^+,\frac{u}{\rho_\alpha}\rangle_{st}+\langle
  \Gamma_rf^+,\frac{s(\alpha)^*u}{\rho_\alpha}\rangle_{st}\\
  &=u^*(f^+-s\Gamma_r f^+)(\alpha)+u^*s(\alpha)(\Gamma_r
  f^+)(\alpha)\\
  &=u^* f^+(\alpha).
    \end{split}
\end{equation}

Similarly, if $\beta\in \gh^-_{s^\#}$ and $v\in \CC^q$, then, in view of
(\ref{eq:1.28})--(\ref{eq:1.30}),
\begin{equation}\label{eq:1.43}
\begin{split}
  \langle f,\Phi_\beta v
  \rangle_{\sD(s)}&=\left\langle(I_m+\Delta_s\left[{{0\quad
  \Gamma_\ell}\atop{\Gamma_r\quad 0}}\right])f,\
  -\frac{1}{\rho_\beta}\left[{{s^\#(\beta)^*v}\atop{v}}\right]
\right\rangle_{st}\\
  &=\langle f^--s^*\Gamma_\ell
  f^-,-\frac{v}{\rho_\beta}\rangle_{st}+\langle\Gamma_\ell
  f^-,\frac{-s^\#(\beta)^*v}{\rho_\beta}\rangle_{st}\\
  &=v^*(f^--s^\#\Gamma_\ell f^-)(\beta)+v^*s^\#(\beta)(\Gamma_\ell f^-)(\beta)\\
  &=v^* f^-(\beta).
    \end{split}
\end{equation}
Thus, if  $f\in{\sD}(s)$ is orthogonal to span $\{\sD^+,\sD^-\}$, then
$$
f(\mu)=0 \quad \text{a.e. on} \ \Omega_0,
$$
thanks to (\ref{eq:1.42}) and (\ref{eq:1.43}).

The entries in the matrix (\ref{eq:1.36a}) are now easily calculated
from the entries in the matrix (\ref{GramLL}) with the help of the
evaluations (\ref{eq:1.42}) and  (\ref{eq:1.43}) and formula
(\ref{eq:1.30}).

Finally,  since $s \in {\mathcal S}_{\kappa}^{\ptq}(\Omega_+)$ the
kernel ${\mathsf \Lambda}_\om^s(\lambda)$ has $\kappa$ negative
squares on $\gh_s^+$, and hence there is a choice of
$\alpha_j\in\gh_s^+$, $ u_j\in\dC^p$ $(1\le j\le n)$, such that the
Gram matrix
 \[
    \begin{bmatrix}
    \langle\Phi_{\alpha_i}u_i,\Phi_{\alpha_j}u_j\rangle_{\sD(s)}
    \end{bmatrix}_{i,j=1}^{n}
 \]
has exactly $\kappa$ negative eigenvalues.
Thus, $\mbox{ind}_-\sD(s)\ge \kappa$. On the other hand,
$\mbox{ind}_-\sD(s)\le \kappa$, by Remark \ref{rem:1.17}.
This proves~\eqref{eq:1.37}.
\end{proof}

\begin{lem}\label{HandD}
  Let $s\in {\mathcal S}_{\kappa}^{p\times q}(\Omega_+)$ and let
$f=\textup{col}\,(f^+, f^-)\in\sD(s)$. Then:
\begin{enumerate}
  \item[\rm(1)] $f^+\in\cK(s)$ and
\begin{equation}\label{eq:1.33A}
  \langle f^+,f^+\rangle_{\cK(s)}\le\langle f,f\rangle_{\sD(s)};
\end{equation}
    \item[\rm(2)]
    $f^-\in\cK_*(s)$ and
\begin{equation}\label{eq:1.33B}
  \langle f^-,f^-\rangle_{\cK_*(s)}\le\langle f,f\rangle_{\sD(s)}.
\end{equation}
\end{enumerate}
 \end{lem}
\begin{proof}
Since $f\in\sD(s)$, there exists a measurable vvf $h=\textup{col}(h_1,h_2)$ 
with components $h_1$ of height $p$ and $h_2$ of height $q$ such that
\[
  f(\mu)=\begin{bmatrix}f^+(\mu)\\ f^-(\mu)\end{bmatrix}
=\Delta_s(\mu)h(\mu)=\begin{bmatrix}
    h_1(\mu) -s(\mu)h_2(\mu) \\
    -s(\mu)^*h_1(\mu) +h_2(\mu)\\
  \end{bmatrix}\quad (\Omega_0 -a.e.).
\]
The decomposition
\[
\Delta_s(\mu)=\begin{bmatrix}
  I_p \\
  -s(\mu)^*\\
\end{bmatrix}[I_p,-s(\mu)]+\begin{bmatrix}
  0 & 0 \\
  0 & I_q-s(\mu)^*s(\mu)\\
\end{bmatrix}
\]
implies that if $f\in\sD(s)$, then
\[
(I-s^*s)^{1/2}h_2\in L_2^q
\]
and
\begin{equation}\label{eq:1.38A}
\begin{split}
\langle f,f\rangle_{\sD(s)}
&=\langle\Delta_s h,h\rangle_{st}+\langle\begin{bmatrix}
  0 & \Gamma_{\ell} \\
  \Gamma_r & 0 \\
\end{bmatrix}\Delta_s h,\Delta_s h\rangle_{st}\\
&=\|f^+\|_{st}^2+\langle(I_q-s^*s)h_2,h_2\rangle_{st}+2{\gR}
\langle\Gamma_r f^+, f^-\rangle_{st}.
\end{split}
\end{equation}
Let
$$
\alpha_\varphi(f^+)= \|f^+ -s\Gamma_r f^+ +s_r \varphi\|_{st}^2
-\|\Gamma_r f^+\|_{st}^2-\|\varphi\|_{st}^2
$$
for $\varphi\in H_2^q$.  Then
\begin{equation}\label{eq:1.39A}
\begin{split}
  \|f^+ -s\Gamma_r f^+ +s_r \varphi\|_{st}^2
  &=\|f^+ -s\Gamma_r f^+\|_{st}^2 +\|s_r \varphi\|_{st}^2\\
  & +2{\gR}\langle f^+ -s\Gamma_r f^+ ,s_r \varphi\rangle_{st},
\end{split}
\end{equation}
and, as $f^-\in b_r(H_2^q)^{\bot}$, $\Gamma_r f^+\in{\cH}(b_r)$
and $b_r \varphi\in b_r H_2^q$,
\begin{equation}\label{eq:1.40A}
\begin{split}
\langle f^+ -s\Gamma_r f^+ , s_r \varphi\rangle_{st}
&=\langle h_1 -sh_2 , s_r \varphi\rangle_{st}-\langle s\Gamma_r f^+ ,
s_r \varphi\rangle_{st}\\
&=\langle s^*h_1 -h_2 , b_r \varphi\rangle_{st}+\langle(I_q -s^*s)h_2 ,
b_r \varphi\rangle_{st}-\langle s^*s\Gamma_r f^+ , b_r \varphi\rangle_{st}\\
&=\langle(I_q -s^*s)(h_2 +\Gamma_r f^+), b_r \varphi\rangle_{st}.
\end{split}
\end{equation}
Therefore,
\begin{equation}\label{eq:1.41A}
\begin{split}
\alpha_\varphi(f^+) &=\|f^+ -s\Gamma_r f^+\|_{st}^2-\|\Gamma_r f^+\|_{st}^2
+\| s_r\varphi\|_{st}^2
-\|\varphi\|_{st}^2\\
&+2{\gR}\langle(I_q -s^*s)(h_2 +\Gamma_r f^+), b_r \varphi\rangle_{st}\\
&=\|f^+\|_{st}^2-\langle(I_q -s^*s)\Gamma_r f^+, \Gamma_r
f^+\rangle_{st}-
\langle(I_q -s^*s)b_r \varphi, b_r \varphi\rangle_{st}\\
&+2{\gR}\langle(I_q -s^*s)(h_2 +\Gamma_r f^+),b_r
\varphi\rangle_{st}-2{\gR}\langle s\Gamma_r f^+ , f^+\rangle_{st},
\end{split}
\end{equation}
and hence,
\begin{equation}\label{eq:1.42A}
\begin{split}
\langle f,f\rangle_{\sD(s)}-\alpha_\varphi(f^+)
&=\langle(I_q -s^*s)h_2 , h_2\rangle+\langle(I_q -s^*s)b_r \varphi, b_r
\varphi\rangle_{st}\\
&+\langle(I_q -s^*s)\Gamma_r f^+ , \Gamma_r f^+\rangle_{st}\\
&-2{\gR}\langle(I_q -s^*s)(\Gamma_r f^+ +h_2), b_r \varphi\rangle_{st}\\
&+2{\gR}\left\{\langle\Gamma_r f^+ , h_2-s^*h_1 \rangle_{st}
+\langle\Gamma_r f^+ , s^*(h_1-sh_2) \rangle_{st}\right\}\\
&=\langle(I_q -s^*s)(h_2+\Gamma_r f^+ -b_r\varphi),
(h_2+\Gamma_r f^+ -b_r\varphi)\rangle_{st}\\
&\ge 0.
\end{split}
\end{equation}
Therefore, by Theorem~\ref{thm:1.10}, $f^+\in{\cK}(s)$ and
\eqref{eq:1.33A} holds.

The proof of the second set of assertions is similar.
\end{proof}

%%%%%%%%%%%%%%%%%%%%%%%%%%%%%%%%%%%%%%%%%%%%%%%%%%%%%
\subsection{The space $\wh\sD(s)$}
%%%%%%%%%%%%%%%%%%%%%%%%%%%%%%%%%%%%%%%%%%%%%%%%%%%%%
 Let $s \in {\mathcal S}_{\kappa}^{\ptq}(\Omega_+)$ and let
the kernel ${\mathsf D}_\om^s(\lambda)$ be defined on
$({\mathfrak h}_s^+\cup {\mathfrak h}_{s^{\#}}^-)\times
({\mathfrak h}_s^+\cup {\mathfrak h}_{s^{\#}}^-)$
by the formulas
\begin{equation}\label{DOL+}
    {\mathsf D}_\om^s(\lambda)=\left\{\begin{array}{cc}
      {\mathsf \Lambda}_\om^s(\lambda) &\quad \text{if}\ (\lam, \omega)
\in{\mathfrak h}_{s}^+\times{\mathfrak h}_{s}^+,\\
     {\dsp\frac{s^{\#}(\lambda)- s(\om)^*}{-\rho_\om
    (\lambda)}} & \quad \text{if}\ (\lam, \omega)\in{\mathfrak h}_{s^{\#}}^-
\times
{\mathfrak h}_{s}^+,
    \end{array}
    \right.
\end{equation}
\begin{equation}\label{DOL-}
    {\mathsf D}_\om^s(\lambda)=\left\{\begin{array}{cc}
      {\dsp\frac{s(\lambda)- s^{\#}(\om)^*}{\rho_\om
    (\lambda)}}  & \quad \text{if}\ (\lam, \omega)\in{\mathfrak h}_{s}^+\times
{\mathfrak h}_{s^{\#}}^-,\\
    {\mathsf L}_\om^s(\lambda) & \quad \text{if}\ (\lam, \omega)
\in{\mathfrak h}_{s^{\#}}^-\times{\mathfrak h}_{s^{\#}}^- .
    \end{array}
    \right.
\end{equation}
It will be shown below that the kernel ${\mathsf D}_\om^s(\lambda)$ has a
finite number of negative squares and that the corresponding
reproducing kernel Pontryagin space $\wh\sD(s)$ is unitarily equivalent to
${\sD(s)}$.

\begin{thm}\label{DswhDs}
   Let $s\in {\mathcal S}_{\kappa}^{\ptq}$. Then:
\begin{enumerate}
\item[\rm(1)]
The kernel ${\mathsf D}_\om^s(\lambda)$ has a finite number of negative
squares on $\gh_{s}^+\times\gh_{s^\#}^-$; \vskip 6pt
\item[\rm(2)]
The de Branges-Rovnyak space $\sD(s)$ is unitarily equivalent to the
reproducing kernel space $\wh\sD(s)$ via the mapping
\begin{equation}\label{eq:1.300}
  U:f=\begin{bmatrix}
    f^+ \\
    f^- \
  \end{bmatrix}\mapsto \wh f\in\wh\sD(s),
\end{equation}
where $\wh f|{\Omega_+}$ is the meromorphic continuation of
$f^+$ to $\gh_s^+$, and $\wh f|{\Omega_-}$ is the meromorphic continuation
of $f^-$ to $\gh_{s^\#}^-$ such that $f^\pm$ are nontangential limits
of $\wh f|{\Omega_\pm}$ from $\Omega_\pm$.
\end{enumerate}
\end{thm}
\begin{proof}
For every choice of $\alpha_i,\,\beta_j,\,u_i,\,v_j$ as in~\eqref{eq:1.35}
the Gram matrix in~\eqref{eq:1.36a} coincides with the matrix
\[
G=\begin{bmatrix}{\dsp u_j^*{\mathsf D}^s_{\alpha_i}(\alpha_j)u_i}
&{\dsp
v_l^*{\mathsf D}^s_{\alpha_i}(\beta_l)u_i}\\
& \\
 {\dsp u_j^*{\mathsf D}^s_{\beta_k}(\alpha_j)v_k} & {\dsp
v_l^*{\mathsf D}^s_{\beta_k}(\beta_l)v_k}
\end{bmatrix}.
\]
Therefore, the first statement of the theorem~ is implied by
Lemma~\ref{lem:1.15}.

Next, it follows from (\ref{DOL+}) and \eqref{eq:1.30} that the restriction of
$U$ to the subspace $\sD^+$ is given by
\begin{equation}\label{eq:1.31a}
  U:\Phi_\alpha u\mapsto{\mathsf D}_\alpha^su
\end{equation}
and maps $\sD^+$ onto
\[
\wh\sD^+:=\left\{{\mathsf D}_\alpha^su:\,\alpha\in\gh_s^+,\,u\in\dC^p\,
\right\}.
\]
Similarly, the formulas (\ref{DOL+}) and \eqref{eq:1.16} show that  $U$
maps the subspace $\sD^-$  onto
\[
\wh\sD^-:=\left\{{\mathsf D}_\beta^sv:\,\beta\in\gh_{s^\#}^-,\,v\in\dC^q\, \right\}.
\]
Moreover, the restriction of $U$ to $\sD^++\sD^-$ is isometric by the definition of the kernel ${\mathsf D}_\om^s$, since
\begin{equation}\label{eq:1.32}
  \langle{\mathsf D}_{\om_j}^su_j,{\mathsf D}_{\om_k}^su_k\rangle_{\wh\sD(s)}=u_k^*{\mathsf D}_{\om_j}^s(\om_k)u_j
=\langle\Phi_{\om_j}u_j,\Phi_{\om_k}u_k\rangle_{\sD(s)}
\end{equation}
for every choice of $\omega_j$, $u_j$ such as in \eqref{eq:1.35}. Since the sets $\sD^++\sD^-$ and $\wh\sD^++\wh\sD^-$ are dense in $\sD(s)$ and $\wh\sD(s)$, respectively, this proves the second statement.
\end{proof}

%%%%%%%%%%%%%%%%%%%%%%%%%%%%%%%%%%%%%%%%%%%%%%%%%%%%%%%%%%%%%%%%%%%%%%%%%%%%%%
\section{The class $\cU_\kappa(j_{pq})$  and the basic theorem}\label{Sec3}
\subsection{The class $\cU_\kappa(J)$ and the space ${\cK}(W)$.}

If $W\in\cU_\kappa(J)$, then assumption (ii) in the definition of
$\cU_\kappa(J)$  guarantees
that $W(\lambda)$ is
invertible in $\Omega_+$ except for an isolated set of points.
Define  $W$ in $\Omega_-$
by the formula
\begin{equation}\label{eq:1.78}
    W(\lambda)=JW^{\#} (\lambda)^{-1} J= JW(\lambda^\circ)^{-*}
    J\quad \text{if}\
    \lambda^\circ \in \gh_W^+\ \text{and}\ \det W(\lam^\circ)\ne 0.
\end{equation}
Recall ~\cite{DSS}, \cite{Fuhr68}, \cite{ArovD08}, that a mvf $\wt
f$ of bounded type in
$\Omega_-$ %meromorphic in $\Omega_+$
is said to be a pseudocontinuation of %bounded type in $\Omega_-$, if there is
a mvf $f$ of bounded type in $\Omega_+$, if $f(\zeta)=\wt f(\zeta)$
a.e. on $\Omega_0$. Since $W$ is of bounded type both  in $\Omega_+$
and  in $\Omega_-$, the nontangential limits
\[
W_\pm(\mu)=\angle\lim_{\lam\to \mu}\{W(\lam):
  \lam\in\Omega_\pm  \}
\]
exist a.e. on $\Omega_0$; and
%, where the nontangential limits
%$W_\pm(\mu)$ are understood in the sense that $\lambda$ tends to
%$\mu\in\Omega_0$ inside the sector $|\arg
%\{\pm\rho_\mu(\lambda)\}|<\pi/2-\varepsilon$  ($\varepsilon>0$).
assumption (ii) in the definition of $\cU_\kappa(J)$ implies that
the nontangential limits $W_\pm(\mu)$ coincide a.e. in $\Omega_0$,
that is $W|_{\Omega_-}$ is a pseudocontinuation of $W|_{\Omega_+}$.
If $W(\lambda)$ is rational this extension is meromorphic on
${\dC}$.  Formula~\eqref{eq:1.78} implies that
$W(\lam)$ is holomorphic and invertible in
\begin{equation}\label{eq:OW}
    \Omega_W:=\gh_W\cap\gh_{W^\#}.
\end{equation}

Let $W\in  \cU_\kappa(J)$ and let ${\cK}(W)$ be the reproducing kernel
Pontryagin
space associated with the kernel ${\mathsf K}^W_\omega(\lambda)$. The
kernel ${\mathsf
K}^W_\omega (\lambda)$ extended to $\Omega_W$ by the
equality~\eqref{eq:1.78} has the same number $\kappa$ of negative squares.
This fact is
due to a generalization of the Ginzburg inequality \cite[Theorem~2.5.2]{ADRS}.

\subsection{Admissibility and linear fractional transformations}

\begin{definition}\label{Adm}\cite{AD86} A mvf
$X(\lambda)$ that is meromorphic in $\Omega_+$ is said to be
$(\Omega_+,J)_\kappa$-admissible if the kernel
\[
\frac{X(\lambda)JX(\omega)^*}{\rho_\omega(\lambda)}
\]
has $\kappa$ negative squares on $\gh_X^+$, the domain of holomorphy
of $X$ in $\Omega_+$.
\end{definition}
\begin{lem}\label{Inv}
Let $\varphi(\lambda)$ and $\psi(\lambda)$ be $p\times p$ and
$p\times q$ meromorphic mvf's in $\Omega_+$ and assume that:
\begin{enumerate}
\item[(i)]
the mvf
$X(\lambda)=\begin{bmatrix}\varphi(\lambda)&\psi(\lambda)\end{bmatrix}$
is $(\Omega_+,j_{pq})_\kappa$-admissible; \vskip 6pt
\item[(ii)]
$\ker X(\lam)^*=\{0\}$ for all $\lambda\in\gh_X^+$.
\end{enumerate}
 Then:
\begin{enumerate}
\item[\rm(1)] $\varphi(\lambda)$ is invertible for all
$\lambda\in\gh_X^+$ except for at most $\kappa$ points; \vskip 6pt
\item[\rm(2)]
The mvf $\varepsilon(\lambda)=-\varphi(\lambda)^{-1}\psi(\lambda)$
belongs to ${\mathcal S}_{\kappa}^{p\times q}$.
\end{enumerate}
\end{lem}
\begin{proof}
Let $\omega_1,\ldots,\omega_n$ be $n$ distinct points in $\gh_X^+$
such that $\varphi(\omega_j)^* u_j=0$ for some nonzero vectors
$u_j\in\dC^p$, $j=1,\ldots,n$. Since $\ker X(\lam)^*=\{0\}$ this
implies that
\begin{equation}\label{eq:2.91}
    v_j:=\psi(\omega_j)^*{u_j}\neq 0\quad \text{for}\ j=1,\dots,n\,,
\end{equation}
and hence that the matrix
\begin{equation}\label{eq:2.92}
\begin{split}
G&=\left[u_k^*\frac{X(\omega_k)j_{pq}X(\omega_j)^*}{\rho_{\omega_j}(\omega_k)}u_j\right]_{j,k=1}^n\\
&=
\left[u_k^*\frac{\varphi(\omega_k)\varphi(\omega_j)^*-\psi(\omega_k)\psi(\omega_j)^*}{\rho_{\omega_j}(\omega_k)}u_j\right]_{j,k=1}^n\\
&=
-\left[\frac{v_k^*v_j}{\rho_{\omega_j}(\omega_k)}\right]_{j,k=1}^n\\
&=-\left[\left\langle%\frac
{v_j}/{\rho_{\omega_j}},
%\frac
{v_k}/{\rho_{\omega_k}}\right\rangle_{st}\right]_{j,k=1}^n
\end{split}
\end{equation}
differs in sign from the Gram matrix of a system
$\{{v_j}/{\rho_{\omega_j}}\}_{j=1}^n$ of linearly independent
vectors in $H_2^p$. Therefore, ${\nu}_-(G)=n\le \kappa$, since $G$
has at most $\kappa$ negative eigenvalues, by assumption.

Suppose next that $\omega_1,\ldots,\omega_n$ and $u_1,\ldots,u_n$
are chosen so that $\nu_-(G)=\kappa$. Then, by perturbing these
points slightly, if necessary, one can insure that the matrices
$\varphi(\omega_1),\ldots, \varphi(\omega_n)$ are all invertible,
$\nu_-(G)=\kappa$ and
\[
G=\left[u_k^*\varphi(\omega_k)
\frac{I_q-\varepsilon(\omega_k)\varepsilon(\omega_j)^*}{\rho_{\omega_j}
(\omega_k)}\varphi(\omega_j)^* u _j \right]_{j,k=1}^n\,.
\]
The $(\Omega_+ ,j_{pq})_\kappa$-admissibility of $X$ implies that
$\varepsilon\in {\mathcal S}_{\kappa}^{p\times q}$.
\end{proof}
\begin{remark}
    The assumption (ii) in Lemma~\ref{Inv} can be relaxed by invoking
a   version of Leech's theorem that is valid in Pontryagin spaces; see e.g., \cite{ADRS1} for
the latter.
\end{remark}

Let
\begin{equation}\label{eq:2.2}
  T_W[\varepsilon]:=(w_{11}(\lam)\varepsilon(\lam)+w_{12}(\lam))
  (w_{21}(\lam)\varepsilon(\lam)+w_{22}(\lam))^{-1}
\end{equation}
denote the linear fractional transformation of a mvf $\varepsilon\in
{\mathcal S}^{\ptq}_{\kappa_2}$ $(\kappa_2\in \ZZ_+)$ based on the
block decomposition
\begin{equation}\label{eq:2.1}
W(\lam)=\left[\begin{array}{ll}
 w_{11}(\lam) & w_{12}(\lam) \\
  w_{21}(\lam) & w_{22}(\lam)
\end{array}\right]
\end{equation}
of a mvf $W\in\cU_\kappa(j_{pq})$ with blocks $w_{11}(\lam)$ and
$w_{22}(\lam)$ of sizes $p\times p$ and $q\times q$, respectively.
Let  $\Omega_W$ be defined by~\eqref{eq:OW} and let
\begin{equation}\label{eq:Lambda}
    \Lambda=\{\lambda\in \Omega_W\cap \gh_\varepsilon^+:\,
{\det}\left(w_{21}(\lambda)\varepsilon(\lambda)+w_{22}(\lambda)\right)=0\}.
\end{equation}
 The transformation $T_W[\varepsilon]$ is well defined for
$\lam\in(\Omega_W\cap \gh_\varepsilon^+)\setminus\Lambda.$

%%%%%%%%%%%%%%%%%%%%%%%%%%%%%%%%%

\begin{lem}\label{pr:2}
Let $\textup W\in \cU_{\kappa_1}(j_{pq})$, $\varepsilon\in {\mathcal
S}_{\kappa_2}^{p\times q}$ and let $\Lambda$ be defined
by~\eqref{eq:Lambda}. Then
\begin{equation}
\label{eq:may11a9}
\Lambda=\{\lambda\in \Omega_W\cap \gh_\varepsilon^+:\,
\det(w_{11}^\#(\lambda)+
\varepsilon(\lambda)w_{12}^\#(\lambda))=0\}
\end{equation}
and:
\begin{enumerate}
\item[\rm(1)]
 $T_W[\varepsilon]$ admits the supplementary representation
\begin{equation}\label{eq:2.7}
    T_W[\varepsilon]=(w_{11}^\#(\lam)+\varepsilon(\lam)w^\#_{12}(\lam))^{-1}
    (w_{21}^\#(\lam)+\varepsilon(\lam)w^\#_{22}(\lam))\quad
    \lam\in(\Omega_W\cap \gh_\varepsilon^+)\setminus\Lambda;
\end{equation}
\item[\rm(2)]
$\Lambda$ consists of at most $\kappa_1+\kappa_2$ points and
$T_W[\varepsilon]\in {\mathcal S}_{\kappa'}^{p\times q}$ with
$\kappa'\leq\kappa_2+\kappa_1$; \vskip 6pt
\item[\rm(3)]
The following equalities hold:
\begin{equation}\label{eq:Hol}
\Omega_W\cap \gh_\varepsilon^+\cap \gh_s^+=(\Omega_W\cap
\gh_\varepsilon^+)\setminus\Lambda,\quad \Omega_W\cap
\gh_{\varepsilon^\#}^-\cap \gh_{s^\#}^-=(\Omega_W\cap
\gh_{\varepsilon^\#}^-)\setminus\Lambda^\circ,
\end{equation}
\[
\begin{split}
\Lambda^\circ&=\{\lambda\in \Omega_W\cap \gh_{\varepsilon^\#}^-:\, \det (w_{11}(\lambda)\varepsilon(\lambda)+w_{12}(\lambda))=0\}\\
&=\{\lambda\in \Omega_W\cap \gh_{\varepsilon^\#}^-:\,\det
(w_{21}^\#(\lambda)+\varepsilon(\lambda) w_{22}^\#(\lambda))=0\}.
\end{split}
\]
\end{enumerate}
\end{lem}

\begin{proof}
(1) The identity
\[
[I_p\quad \varepsilon(\lambda)]W^\#(\lambda)j_{pq}W(\lambda)
\begin{bmatrix} \varepsilon(\lambda)\\ I_q\end{bmatrix}=
[I_p\quad \varepsilon(\lambda)]j_{pq}
\begin{bmatrix} \varepsilon(\lambda)\\ I_q\end{bmatrix}=0
\]
implies that for all $\lambda\in\Omega_W\cap \gh_\varepsilon^+$
\begin{equation}\label{eq:2.95}
\begin{split}
(w_{21}^\#(\lambda)&+\varepsilon(\lambda)
w_{22}^\#(\lambda))(w_{21}(\lambda)\varepsilon(\lambda)+w_{22}(\lambda))\\
&=(w_{11}^\#(\lambda)+\varepsilon(\lambda)
w_{12}^\#(\lambda))(w_{11}(\lambda)\varepsilon(\lambda)+w_{12}(\lambda)).
\end{split}
\end{equation}
Thus, if
$
(w_{21}(\lambda)\varepsilon(\lambda)+w_{22}(\lambda))\eta=0 $ for
some vector $\eta\neq 0$, then
\begin{equation}\label{eq:2.95a}
\xi=(w_{11}(\lambda)\varepsilon(\lambda)+w_{12}(\lambda))\eta\neq 0
\end{equation}
since $W(\lambda)$ is invertible for all $\lambda\in\Omega_W$. It
follows from~\eqref{eq:2.95} that
\begin{equation}\label{eq:2.95b}
(w_{11}^\#(\lambda)+\varepsilon(\lambda) w_{12}^\#(\lambda))\xi=0.
\end{equation}
Therefore,
\[
\Lambda\subseteq\{\lambda\in \Omega_W\cap \gh_\varepsilon^+:\,
\det(w_{11}^\#(\lambda)+
\varepsilon(\lambda)w_{12}^\#(\lambda))=0\}.
\]
A similar argument shows that the opposite inclusion is also
valid. Therefore, the two sets are equal.   The
identity~\eqref{eq:2.7} is immediate from~\eqref{eq:2.95}.

(2) Next, let
\[
X(\lambda)=\begin{bmatrix}\varphi(\lambda) &
\psi(\lambda)\end{bmatrix}=\begin{bmatrix}I_p
&\varepsilon(\lambda)\end{bmatrix}W^\#(\lambda).
\]
Then the formula
\[
\begin{split}
 \frac{X(\lambda)j_{pq}X(\omega)^*}{\rho_\omega(\lambda)}&=
\frac{1}{\rho_\omega(\lambda)}\begin{bmatrix}I_p
&\varepsilon(\lambda)\end{bmatrix}W^\#(\lambda)
j_{pq}W^\#(\omega)^*\left[%
\begin{array}{c}
   I_p\\
   \varepsilon(\omega)^*
\end{array}%
\right]\\
&=
\begin{bmatrix}I_p
&-\varepsilon(\lambda)\end{bmatrix}W(\lambda)^{-1}
{\mathsf K}_\omega^W(\lambda)W(\omega)^{-*}\left[%
\begin{array}{c}
   I_p\\
   -\varepsilon(\omega)^*
\end{array}%
\right]+{\mathsf \Lambda}_\omega^\varepsilon (\lambda)
\end{split}
\]
implies that $X(\lambda)$ is $(\Omega_+, j_{pq})_{\kappa'}$
admissible with $\kappa'\leq\kappa_1+\kappa_2$, thanks to Theorem
\ref{SUMK12}. Clearly, condition (ii) in Lemma~\ref{Inv} is also
satisfied, and, by Lemma~\ref{Inv},
\[
 \varphi(\lambda)=w_{11}^\#(\lambda)+\varepsilon(\lambda)w_{12}^\#(\lambda)
 \]
is invertible for all $\lambda\in\Omega_W\cap \gh_\varepsilon^+$
except for at most $\kappa'$ points, and the mvf
$s(\lambda):=T_W[\varepsilon](\lambda)$ given by~\eqref{eq:2.7}
belongs to ${\mathcal S}_{\kappa'}^{p\times q}$ with
$\kappa'\leq\kappa_1+\kappa_2$.

(3) If $\lam\in\Lambda$ then it follows from \eqref{eq:2.95a} and
\eqref{eq:2.95b} that $\lam$ is a pole of $s$. This proves the
inclusion
\[
\Omega_W\cap \gh_\varepsilon^+\cap \gh_s^+\subseteq (\Omega_W\cap
\gh_\varepsilon^+)\setminus\Lambda,
\]
and hence the first of the equalities in~\eqref{eq:Hol}, since the
converse is obvious. The proof of the second one is similar. The last equality is implied by~\eqref{eq:2.95}.
\end{proof}

The next theorem characterizes the range of the linear fractional
transform $T_W$ in terms of admissibility; it extends a result of de
Branges and Rovnyak to an indefinite setting.
\begin{thm}\label{thm:3.3}
Let $s\in {\mathcal S}_{\kappa}^{p\times q}(\Omega_+)$ and let $W\in
\cU_{\kappa_1}(j_{pq})$ with $\kappa_1\leq\kappa$. Then $s\in
T_W[{\mathcal S}_{\kappa-\kappa_1}^{p\times q}]$ if and only if
$\begin{bmatrix}I_p &-s\end{bmatrix}W$ is $(\Omega_+
,j_{pq})_{\kappa-\kappa_1}$- admissible.
\end{thm}
\begin{proof}
If $s=T_W[\varepsilon]$ for some $\varepsilon\in {\mathcal
S}_{\kappa-\kappa_1}^{p\times q}$ and  $\lam \in \Omega_W\cap
\gh_\varepsilon^+$, then, in view of \eqref{eq:2.7},
\begin{equation}\label{eq:2.100}
\begin{split}
\begin{bmatrix}I_{p} & -s(\lambda)\end{bmatrix}W(\lambda) &=(w_{11}^\# +\varepsilon
w_{12}^\#)^{-1}
\begin{bmatrix}w_{11}^\#+\varepsilon w^\#_{12} & w_{21}^\#+\varepsilon w^\#_{22}\end{bmatrix}
j_{pq}W(\lambda)\\
&=(w_{11}^\# +\varepsilon w_{12}^\#)^{-1}
\begin{bmatrix}I_{p} & \varepsilon(\lambda)\end{bmatrix}
W^\# (\lambda)j_{pq}W(\lambda)\\
&=(w_{11}^\# +\varepsilon w_{12}^\#)^{-1}
\begin{bmatrix}I_{p} & -\varepsilon(\lambda)\end{bmatrix}.
\end{split}
\end{equation}
Therefore, the mvf $\begin{bmatrix}I_{p} & -s(\lambda)\end{bmatrix}
W(\lambda)$ is $(\Omega_+,j_{pq})_{\kappa-\kappa_1}$- admissible.

Conversely, if the mvf
$$
\begin{bmatrix}I_{p} &
-s(\lambda)\end{bmatrix}W(\lambda)=\begin{bmatrix}
w_{11}(\lambda)-s(\lambda)w_{21}(\lambda)&
w_{12}(\lambda)-s(\lambda)w_{22}(\lambda)\end{bmatrix}
$$
is $(\Omega_+, j_{pq})_{\kappa-\kappa_1}$- admissible, then, by
Lemma~\ref{Inv},  there is a mvf $\varepsilon\in {\mathcal
S}_{\kappa-\kappa_1}^{p\times q}$ such that
\[
w_{12}(\lambda)-s(\lambda)w_{22}(\lambda)
=-[w_{11}(\lambda)-s(\lambda)w_{21}(\lambda)]\varepsilon(\lambda).
\]
Thus, $s=T_W[\varepsilon]$, i.e., $s\in T_W[{\mathcal
S}_{\kappa-\kappa_1}^{p\times q}]$.
\end{proof}
%%%%%%%%%%%%%%%%%%%%%%%%%%%%%%%

\subsection{Proof of Theorem \ref{thm:0.1}}
Since the boundary values $s(\mu)$ of a mvf $s\in {\mathcal
S}^{\ptq}_\kappa$ exist a.e. on $\Omega_0$, the mvf $\Delta_s(\mu)$
may be defined a.e. on $\Omega_0$ by formula \eqref{DeltaS} for such
$s$.
\begin{proof}\quad {\it Necessity}. Let $s=T_W[\varepsilon]$
for some mvf $\varepsilon\in {\mathcal S}^{\ptq}_{\kappa_2}$,
($\kappa=\kappa_1+\kappa_2$) and let
\begin{equation}
\label{eq:2.11} \Delta (\lambda)=\left\{
  \begin{array}{ll}
     \Delta_+(\lam) = \begin{bmatrix}I_{p} & -s(\lambda)\end{bmatrix} \quad
&\text{if}\quad \lam\in \mathfrak{h}_s^+,\\ \\
     \Delta_-(\lam)=\begin{bmatrix}-s^\#(\lam)&\ I_q\end{bmatrix}\quad&
\text{if}\quad \lam \in \mathfrak{h}_{s^{\#}}^-
  \end{array}\right.
\end{equation}
and
\begin{equation}\label{eq:2.13}
   R(\lam)=\left\{\begin{array}{ll}(w^\#_{11}(\lam)
   +\varepsilon(\lam)w^\#_{12}(\lam))^{-1}\quad&\text{if}\quad
\lam\in\Omega_W\cap\gh_\varepsilon^+\setminus\Lambda,\\ \\
(\varepsilon^\#(\lam) w^\#_{21}(\lam)
    +w^\#_{22}(\lam))^{-1}\quad &\text{if}\quad \lam \in
 \Omega_W\cap \gh_{\varepsilon^\#}^-\setminus\Lambda^\circ.
\end{array}\right.
\end{equation}
Then, since
\begin{equation}\label{eq:2.12}
\Delta_+(\lam)W(\lam)=  R(\lam)\begin{bmatrix}I_{p} &
-\varepsilon(\lambda)\end{bmatrix}\quad\text{for}\ \lam \in
\Omega_W\cap\gh_\varepsilon^+\cap{\mathfrak h}_s^+
\end{equation}
by (\ref{eq:2.100}), (\ref{eq:2.12}) implies that
\begin{equation}\label{eq:2.15}
    \begin{split}
      {\mathsf \Lambda} ^s_\omega(\lam)&= \begin{bmatrix}I_{p} &
-s(\lam)\end{bmatrix} \frac{j_{pq}}{\rho_\omega(\lam)}
    \begin{bmatrix}I_p\\
      -s(\omega)^*
    \end{bmatrix}\\
       & = \Delta_+(\lam)\frac{j_{pq}-W(\lam)j_{pq}W(\omega)^*}
       {\rho_\omega(\lam)} \Delta_+(\omega)^*+
       \Delta_+(\lam)\frac{W(\lam)j_{pq}W(\omega)^*}
       {\rho_\omega(\lam)}\Delta_+(\omega)^*\\
       &=\Delta_+(\lam){\mathsf K}^W_\omega(\lam)\Delta_+(\omega)^*+R(\lam){\mathsf \Lambda}^\varepsilon_\omega(\lam)R(\om)^*\quad\text{for}\
\lam \in \Omega_W\cap\gh_\varepsilon^+\cap{\mathfrak h}_s^+.
    \end{split}
\end{equation}
Similarly, \eqref{eq:2.2} implies that
\begin{equation}\label{eq:2.17}
    \begin{array}{ll}
      \Delta_-(\lam)\ W(\lam) &= -\begin{bmatrix}s^\#(\lam)&\ I_q\end{bmatrix}
      j_{pq}\ W(\lam) \\
      \\
       & =-(\varepsilon^\# w^\#_{21}+w^\#_{22})^{-1}
    \begin{bmatrix}\varepsilon^\# w^\#_{11}+w^\#_{12} &\ \varepsilon^\# w^\#_{21}+w^\#_{22}\end{bmatrix}
       j_{pq}\ W(\lam)\\
       \\
      & = -(\varepsilon^\# w^\#_{21}+w^\#_{22})^{-1}
      \begin{bmatrix}\varepsilon^\# &\ I_q\end{bmatrix}
      W^\#(\lam)
     j_{pq}\ W(\lam)\\
     \\
      & =R(\lam)\begin{bmatrix}-\varepsilon^\#(\lam)&\ I_q
\end{bmatrix}\quad\text{for}\ \lam\in \Omega_W\cap \gh_{\varepsilon^\#}^-\cap
{\mathfrak h}_{s^{\#}}^-.
         \end{array}
\end{equation}
Therefore,
\begin{equation}\label{eq:2.19}
   \begin{split}
   \frac{s^\#(\lam)-s(\omega)^*}{-\rho_\omega(\lam)}&=
   \begin{bmatrix}-s^\#(\lam)&\ I_q\end{bmatrix}
       {\displaystyle\frac{j_{pq}}{\rho_\omega(\lam)}}\begin{bmatrix}I_p\\
      -s(\omega)^*
    \end{bmatrix}\\
&=\Delta_-(\lam){\mathsf K}^W_\omega(\lam)\Delta_+(\omega)^*+R(\lam)
\frac{\varepsilon^\#(\lam)-\varepsilon(\omega)^*}{-\rho_\omega(\lam)}
R(\om)^*
\end{split}
\end{equation}
for $\lam\in \Omega_W\cap \gh_{\varepsilon^\#}^-\cap{\mathfrak
h}_{s^{\#}}^-$ and $\omega\in\Omega_W\cap{\mathfrak
h}_\varepsilon^+\cap{\mathfrak h}_s^+$.

%%%%%%%%%%%%%%%%%%%%%%%%%%%%%%%%%%%%%%

In much the same way, (\ref{eq:2.17}) implies that
\begin{eqnarray}\label{eq:2.23}
{\mathsf L}^s_\omega(\lam)&=& \begin{bmatrix}-s^\#(\lam)&\
I_q\end{bmatrix} \frac{j_{pq}}{\rho_\omega(\lam)}\left[
\begin{array}{c}
-s^\#(\omega)^*\\
I_q
\end{array}\right]\nonumber\\
&=& \Delta_-(\lam)\ {\mathsf
K}^W_\omega(\lam)\Delta_-(\omega)^*+R(\lam){\mathsf
L}^\varepsilon_\omega(\lam)R(\omega)^*
\end{eqnarray}
for $\lam,\omega \in
 \Omega_W\cap \gh_{\varepsilon^\#}^-\cap{\mathfrak h}_{s^{\#}}^-$.
Thus, it follows from~\eqref{DOL+}, \eqref{DOL-},
\eqref{eq:2.15}--\eqref{eq:2.23} that
\begin{equation}\label{eq:2.21}
    {\mathsf D}^s_\omega(\lam)=\Delta(\lam)
    {\mathsf K}^W_\omega(\lam)\Delta(\omega)^*+
    R(\lam){\mathsf D}^\varepsilon_\omega(\lam)R(\omega)^*
\end{equation}
for all $\lam,\omega\in
\Omega_W\cap(\gh_\varepsilon^+\cup\gh_{\varepsilon^\#}^-)\cap({\mathfrak
H}_s^+\cup {\mathfrak h}_{s^{\#}}^-)$.

Applying Theorems \ref{SUMK12} and \ref{FACTK} to the identity
(\ref{eq:2.21}) and using the fact that the number
$\textup{sq}_-{\mathsf D}^s_\omega(\lambda)$ coincides with the sum
of the numbers of negative squares of the kernels on the right hand
side of \eqref{eq:2.21} one obtains the following decomposition of
the space $\wh\sD(s)$
\begin{equation}\label{eq:2.27}
   \wh\sD(s)=\Delta(\lam){\cK}(W)+R(\lam)\wh\sD(\varepsilon)
\end{equation}
where the mapping $f\mapsto \Delta  f$ from ${\cK}(W)$ into
$\wh\sD(s)$ is a contraction. The necessity part of the theorem now
follows from Definition~\ref{WHDS} and Theorem~\ref{DswhDs}, since
the mapping $f\mapsto \Delta_s  f$ from ${\cK}(W)$ into $\sD(s)$ is
the composition of the mapping $ \Delta:\,{\cK}(W)\to\wh\sD(s)$ and
the unitary mapping $U^{-1}:\,\wh\sD(s)\to\sD(s)$
from~\eqref{eq:1.300}.

{\it Sufficiency.} If conditions (1)--(3) of the theorem hold, then
\begin{equation}
    \label{IneqDK}
f\in\cK(W)\Longrightarrow\Delta_sf\in\sD(s)\quad\text{and}\quad
\langle\Delta_sf,\Delta_sf\rangle_{\sD(s)}\le \langle
f,f\rangle_{\cK(W)}.
\end{equation}
Moreover, in view of Lemma~\ref{HandD}, $\begin{bmatrix} I_p &
-s\end{bmatrix} f\in\cK(s)$ and
\begin{equation}
    \label{IneqKsW}
\langle \begin{bmatrix} I_p & -s
\end{bmatrix}f,\begin{bmatrix}
I_p & -s
\end{bmatrix}f\rangle_{\cK(s)}\le \langle\Delta_sf,\Delta_sf\rangle_{\sD(s)}\le
\langle f,f\rangle_{\cK(W)}.
\end{equation}
Therefore, the operator $T:\, f\mapsto \begin{bmatrix} I_p & -s
\end{bmatrix}f$ maps ${\cK}(W)$ into $\cK(s)$ contractively.
To find the adjoint operator
\[
T^*:\, {\cK}(s)\to\cK(W),
\]
let $\alpha\in\gh_W\cap\gh_s^+$, $u\in\dC^p$ and $f\in\cK(W)$. Then
\[
\begin{split}
\langle f,T^*{\mathsf \Lambda}^s_\alpha u\rangle_{\cK(W)} &=\langle
Tf,{\mathsf \Lambda}^s_\alpha u\rangle_{\cK(s)} =\langle
\begin{bmatrix} I_p & -s
\end{bmatrix}f,{\mathsf \Lambda}^s_\alpha u\rangle_{\cK(s)}\\
&=u^*\begin{bmatrix} I_p & -s(\alpha)
\end{bmatrix}f(\alpha)=\left\langle f,{\mathsf K}_{\alpha}^W\begin{bmatrix}
  I_p \\
  -s(\alpha)^*
\end{bmatrix} u\right\rangle_{\cK(W)}.
\end{split}
\]
Therefore,
\begin{equation}\label{eq:2.270}
 T^*{\mathsf \Lambda}^s_\alpha u={\mathsf K}_{\alpha}^W\begin{bmatrix}
  I_p \\
  -s(\alpha)^*
\end{bmatrix} u\quad(\alpha\in\gh_W\cap \gh_s^+,\,u\in\dC^p).
\end{equation}

Since $T:\, \cK(W)\to\cK(s)$ is a contraction,  Theorem 1.3.4 of
\cite{ADRS} implies that
\[
\nu_-(I-TT^*)=\textup{ind}_-\cK(s)-\textup{ind}_-\cK(W)=\kappa_2,
\]
%%%%%%%%%%%%%%%%%%%%%%%%%%%%%%%%%%%%%%%%%%%%%%%%%%%%%%%%%%%%%%%%%
and hence that for every choice of
$\alpha_j\in\gh_W\cap\gh_s^+,\,u_j\in\dC^p$ and $\xi_j\in\dC$ ($1\le
j\le n$) the form
\begin{equation}\label{eq:2.31}
  \sum_{j,k=1}^n\left\{
\left\langle{\mathsf \Lambda}_{\alpha_j}^su_j,{\mathsf
\Lambda}_{\alpha_k}^su_k\right\rangle_{\cH(s)} -\left\langle{\mathsf
K}_{\alpha_j}^W\begin{bmatrix}
  I_p \\
  -s(\alpha_j)^*
\end{bmatrix} u_j,{\mathsf K}_{\alpha_k}^W\begin{bmatrix}
  I_p \\
  -s(\alpha_k)^*
\end{bmatrix} u_k\right\rangle_{\cK(W)}\right\}\xi_j\bar\xi_k
\end{equation}
has at most $\kappa_2$ (and for some choice of
$\alpha_j\in\gh_W\cap\gh_s^+,\,u_j\in\dC^p$ and $\xi_j\in\dC$
exactly $\kappa_2$) negative squares. Since \eqref{eq:2.31} can be
rewritten in the form
\[
 \begin{split}
 \sum_{j,k=1}^n&\left\{
u_k^*{\mathsf \Lambda}_{\alpha_j}^s(\alpha_k)u_j
-u_k^*\begin{bmatrix} I_p & -s(\alpha_k)
\end{bmatrix}{\mathsf K}_{\alpha_j}^W(\alpha_k)\begin{bmatrix}
  I_p \\
  -s(\alpha_j)^*
\end{bmatrix} u_j\right\}\xi_j\bar\xi_k\\
&=\sum_{j,k=1}^n\left\{ u_k^*\begin{bmatrix} I_p & -s(\alpha_k)
\end{bmatrix}
\frac{W(\alpha_k)j_{pq}W(\alpha_j)^*}{\rho_{\alpha_j}(\alpha_k)}\begin{bmatrix}
  I_p \\
  -s(\alpha_j)^*
\end{bmatrix} u_j\right\}\xi_j\bar\xi_k,
\end{split}
\]
it follows that the mvf $ \begin{bmatrix} I_p & -s(\lambda)
\end{bmatrix}W (\lambda)$ is
$(\Omega_+,j_{pq})_{\kappa_2}$--admissible. Theorem~\ref{thm:3.3}
serves to complete the proof.
\end{proof}

%%%%%%%%%%%%%%%%%%%%%%%%%%%%%%%%%%%%%%%%%%%%%%%%%%%%%
\section{The resolvent matrix and the class $\cU_\kappa^\circ(j_{pq})$}
\label{Sec4}
%%%%%%%%%%%%%%%%%%%%%%%%%%%%%%%%%%%%%%%%%%%%%%%%%%%%%
\subsection{The class $H_{\kappa,\infty}^{\ptq}$.}%$Pole and zero multiplicities.}
Let  $G(\lambda)$ be a $p\times q$ mvf that is meromorphic on $\Omega_+$ with
a Laurent expansion
\begin{equation}\label{Gexp}
    G(\lambda)=(\lambda-\lambda_0)^{-k}G_{-k}+\cdots
+(\lambda-\lambda_0)^{-1}G_{-1}+G_0+\cdots
\end{equation}
 in a neighborhood of a pole $\lambda_0\in\Omega_+$. The pole
multiplicity $M_{\pi}(G, \lambda_0)$ is defined by (see~\cite{KL})
\begin{equation}\label{MTexp}
M_{\pi}(G, \lambda_0)=\rank L(G,\lam_0), \quad
L(G,\lam_0)=\begin{bmatrix}
  G_{-k} &  & {\bf 0} \\
  \vdots & \ddots &  \\
  G_{-1} &\hdots & G_{-k}
\end{bmatrix}.
\end{equation}
The pole multiplicity of $G$ over $\Omega_+$ is given by
\begin{equation}\label{eq:7.11p}
  M_{\pi}(G, \Omega_+)=\sum_{\lambda\in\Omega_+}M_{\pi}(G, \lambda).
\end{equation}
This definition of pole multiplicity coincides with that based on the
Smith-McMillan representation of $G$ (see \cite{BGR}).
\begin{remark}\label{rem:3.1}
For a Blaschke-Potapov product $b$
of the form (\ref{BPprod}) the following statements are equivalent:
\begin{enumerate}
  \item[(1)] the degree of the Blaschke-Potapov product $b$ is equal $\kappa$;
  \vskip 6pt
  \item[(2)] $M_{\pi}(b^{-1}, \Omega_+)=\kappa$;\vskip 6pt
  \item[(3)] the kernel ${\mathsf \Lambda}_\om^{b^{-1}}(\lambda)$ has $\kappa$ negative squares in $\Omega_+$.
\end{enumerate}
\end{remark}

The zero multiplicity of a square mvf $F\in H_\infty^{\ptp}$ over $\Omega_+$ is defined in~\cite{KL81}
as the degree of the maximal left Blaschke-Potapov factor $b$ of $F$.
Remark~\ref{rem:3.1} implies that
the zero multiplicity of a square mvf $F\in H_\infty^{\ptp}$ with nontrivial determinant is connected with the
pole multiplicity of  $F^{-1}$ via
\[
M_{\zeta}(F, \Omega_+)=M_{\pi}(F^{-1}, \Omega_+).
\]

The class $H_{\kappa,\infty}^{p\times q}(\Omega_+)$ consists of
$p\times q$ mvf's $G$ of the form $G=H+B$, where $B$ is a rational
$p\times q$ mvf of pole multiplicity $M_{\pi}(B, \Omega_+)\le\kappa$
and $H\in H_{\infty}^{p\times q}(\Omega_+)$ (see~\cite{AAK71}).

The next lemma is useful for calculating
pole multiplicities.
\begin{lem}\label{lem:7.7}
Let  $H_1\in H_{\infty}^{r\times p}$, $H_2\in H_{\infty}^{q\times
r}$, and let mvf's $G_1\in H_{\kappa_1,\infty}^{\ptq}$, $G_2\in
H_{\kappa_2,\infty}^{\ptq}$ have the following Laurent expansions
\begin{equation}\label{Gexp1}
    G_j(\lambda)=\frac{G_{-k}^{(j)}}{(\lambda-\lambda_0)^{k}}+\cdots
+\frac{G_{-1}^{(j)}}{\lambda-\lambda_0}+G_0^{(j)}+\cdots\quad
(j=1,2)
\end{equation}
% in a neighborhood of a pole
at $\lambda_0\in\Omega_+$ and let
\begin{equation}\label{M1}
   \textup{rng }L(G_1,\lambda_0)\subseteq \textup{rng }L(G_2,\lambda_0)\quad\forall \  \lambda_0\in\Omega_+.
\end{equation}
Then
\begin{equation}\label{M2}
    M_{\pi} (H_1G_1, \Omega_+)\le M_{\pi} (H_1G_2, \Omega_+).
\end{equation}
If
\begin{equation}\label{M3}
   \textup{ker }L(G_1,\lambda_0)\supseteq \textup{ker }L(G_2,\lambda_0)\quad\forall \  \lambda_0\in\Omega_+,
\end{equation}
then
\begin{equation}\label{M4}
    M_{\pi} (G_1H_2, \Omega_+)\le M_{\pi} (G_2H_2, \Omega_+).
\end{equation}
 \end{lem}
\begin{proof}
Let the mvf's $H_j$ $(j=1,2)$ have the following expansions at
$\lambda_0$
\[
   H_j(\lam)=H_0^{(j)}+H_1^{(j)}(\lambda-\lambda_0)+\cdots
+H_{n-1}^{(j)}(\lambda-\lambda_0)^{k-1}+\cdots
   \quad
\]
and let
\[
T({H_j},\lam_0)=\begin{bmatrix}
  H_0^{(j)} &  & {\bf 0} \\
  \vdots & \ddots &  \\
  H_{k-1}^{(j)} &\hdots & H_0^{(j)}
\end{bmatrix} \quad \text{for}\ j=1,2.
\]
 Then
$$
    L(H_1G_j,\lam_0)=T(H_1,\lambda_0)L(G_j,\lam_0)\quad\text{and}\quad
    L(G_jH_2,\lam_0)=L(G_j,\lam_0)T(H_2,\lam_0).
$$
Therefore,  in view of~\eqref{M1},  one obtains for every
$\lambda_0\in\Omega_+$
\[
\begin{split}
M_{\pi} (H_1G_1, \lambda_0)&=\rank T(H_1,\lambda_0)L(G_1,\lam_0)\\
&\le  \rank T(H_1,\lambda_0)L(G_2,\lam_0)= M_{\pi} (H_1G_2,
\lam_0)\quad\text{for}\ \lambda_0\in\Omega_+,
\end{split}
\]
which implies~\eqref{M2}. Next~\eqref{M3} yields the inequality
\[
\begin{split}
M_{\pi} (G_1H_2, \lambda_0)&=\rank L(G_1,\lam_0)T(H_2,\lam_0)\\
&\le  \rank L(G_2,\lam_0)T(H_2,\lam_0)=M_{\pi} (G_2H_2, \lam_0)
\quad\text{for}\ \lambda_0\in\Omega_+,
\end{split}
\]
which proves~\eqref{M4}.
\end{proof}
As a corollary we obtain the following generalization of a noncancellation
lemma from~\cite{Der03}.
\begin{lem}\label{lem:5.4}
Let $G\in H_{\kappa,\infty}^{\ptq}$, $H_1\in
H_{\infty}^{p\times p}$ and $H_2\in H_{\infty}^{q\times q}$. Then
\begin{equation}\label{b1G}
    M_\pi(H_1G,\Omega_+)=M_\pi(G,\Omega_+)\Longrightarrow
M_\pi(H_1GH_2,\Omega_+)=M_\pi(GH_2,\Omega_+),
\end{equation}
whereas
\begin{equation}\label{Gb2}
    M_\pi(GH_2,\Omega_+)=M_\pi(G,\Omega_+)\Longrightarrow
M_\pi(H_1GH_2,\Omega_+)=M_\pi(H_1G,\Omega_+).
\end{equation}
\end{lem}

\begin{proof}
Since
$$
\textup{ker }L(H_1G, \lambda)\supseteq \textup{ker }L(G,\lambda)
$$
for every point $\lambda\in\Omega_+$, it follows that
\begin{equation}
\label{eq:jul9a8} M_\pi(H_1G,\lambda)=\textup{rank
}L(H_1G,\lambda)\le \textup{rank }L(G,\lambda)=M_\pi(G,\lambda)
\end{equation}
for every point $\lambda\in\Omega_+$. Therefore, the first equality in
(\ref{b1G})
implies that
%%%%%%%%%%%%%%%%%%%%%%%%%%%%%%%%%%%%%%%%%%%%%%%%%%%%%%%%%%%%%%%%%%%%%%%%%%%%
equality prevails in (\ref{eq:jul9a8}) and so too in (\ref{M1}) and
(\ref{M3}) with $G_1=H_1G$ and $G_2=G$. The
second equality in  (\ref{b1G}) then follows easily from Lemma
\ref{lem:7.7}. The verification of the
second assertion is similar.
\end{proof}

It follows from the results of \cite{KL}, \cite{Der03} (see also
\cite[Theorem 5.2]{Ki}) that every  mvf $G\in
H_{\kappa,\infty}^{p\times q}(\Omega_+)$ admits %coprime
factorizations
\begin{equation}\label{eq:7.110}
  G(\lambda)=G_\ell(\lambda)^{-1}H_\ell(\lambda)=H_r(\lambda)G_r(\lambda)^{-1},
\end{equation}
where $G_\ell\in {\mathcal S}_{in}^{p\times p}$, $G_r\in {\cS}_{in}^{q\times q}$ are
Blaschke-Potapov factors of degree $M_{\pi} (G, \Omega_+)(\le\kappa)$ and
$H_\ell,\,H_r\in H_{\infty}^{p\times q}$ such that
\begin{equation}\label{KLcanon2a}
\rank \left[\begin{array}{cc}
 G_\ell(\lam) & H_\ell(\lam)
\end{array}\right]
=p\quad\text{for}\ \lam\in\Omega_+,
\end{equation}
\begin{equation}\label{KRcanon2a}
\rank \left[\begin{array}{cc}
 G_r^\#(\lam) & H^\#_r(\lam)
\end{array}\right]
=q\quad \text{for} \ \lam\in\Omega_-.
\end{equation}

It is easily shown  that Lemma~\ref{Corona} can be extended to this more general
situation. (The details are left to the reader.)

\begin{lem}\label{CoronaA}
 Let $H_\ell, H_r\in H_{\infty}^{p\times q}$ and let
 $G_\ell\in H_\infty^{\ptp}$, $G_r\in H_{\infty}^{q\times q}$ be a pair
 of mvf's such that $
G_\ell^{-1}\in H_{\kappa,\infty}^{p\times p}$ and
$G_r^{-1}\in H_{\kappa,\infty}$ for some
$\kappa\in\dN\cup\{0\}$. Then:
\begin{enumerate}
\item[(i)]  $H_\ell$ and $G_\ell$ meet the rank condition
(\ref{KLcanon2a}), if and only if there exists a pair of mvf's
$C_\ell\in H^{p\times p}_\infty$ and $D_\ell\in H^{q\times
p}_\infty$ such that
\begin{equation}\label{CorFormulaA}
G_\ell(\lam)C_\ell(\lam)+H_\ell(\lam)D_\ell(\lam)=I_p\ \text{ for }\
\lam\in \Omega_+;
\end{equation}
\item[(ii)]  $H_r$ and $G_r$ meet the rank condition
(\ref{KRcanon2a}), if and only if there exists a pair of mvf's
$C_r\in H^{p\times p}_\infty$ and $D_r\in H^{q\times p}_\infty$
such that
\begin{equation}\label{CorFormula1A}
C_r(\lam)G_r(\lam)+D_r(\lam)H_r(\lam)=I_q\ \text{ for }\ \lam\in
\Omega_+.
\end{equation}
\end{enumerate}
\end{lem}

Pairs of mvf's $H_\ell$, $G_\ell$ and $H_r$, $G_r$ which satisfy the
assumptions of Lemma~\ref{CoronaA} and the
conditions~\eqref{CorFormulaA} and~\eqref{CorFormula1A} are called
left coprime and right coprime, respectively, over $\Omega_+$. Left
and right coprime factorizations may be characterized in terms of
the pole
 multiplicities of their factors.
\begin{proposition}\label{Prop:5.1}  Let $H_\ell, H_r\in
H_{\infty}^{p\times q}$ and let  $G_\ell\in H_{\infty}^{p\times p}$ and
$G_r\in H_{\infty}^{q\times q}$ be a pair
 of mvf's such that $
G_\ell^{-1}\in H_{\kappa,\infty}^{p\times p}$ and
$G_r^{-1}\in H_{\kappa,\infty}$ for some
$\kappa\in\dN\cup\{0\}$. Then:
\begin{enumerate}
\item[(i)] The pair $G_\ell$, $H_\ell$ is left coprime over $\Omega_+$
$\Longleftrightarrow$ $M_{\pi} (G_\ell^{-1}H_\ell, \Omega_+)=M_{\pi}
(G_\ell^{-1}, \Omega_+)$.
\vskip 6pt
\item[(ii)] The pair $G_r$, $H_r$ is right
coprime over $\Omega_+$ $\Longleftrightarrow$ $M_{\pi} (H_rG_r^{-1},
\Omega_+)=M_{\pi} (G_r^{-1}, \Omega_+)$.
\end{enumerate}
\end{proposition}
\begin{proof}
Suppose first  that $H_\ell$ and $ G_\ell $ are left coprime over
$\Omega_+$. Then it follows from~\eqref{CorFormulaA} that
\[
G_\ell^{-1}= G_\ell^{-1} H_\ell D_\ell+ C_\ell
\]
and hence, that
\[
M_{\pi} (G_\ell^{-1}, \Omega_+)=M_\pi(G_\ell^{-1}H_\ell D_\ell, \Omega_+)\le M_{\pi}
(G_\ell^{-1}H_\ell, \Omega_+).
\]
The converse inequality is obvious. This proves the implication
$\Longrightarrow$ in (i).

Next, assume that
\[
M_{\pi} (G_\ell^{-1}H_\ell, \Omega_+)=M_{\pi} (G_\ell^{-1},
\Omega_+)=\kappa'
\]
for some finite nonnegative integer $\kappa^\prime$ and
let $ G_\ell^{-1}$ and $ G_\ell^{-1} H_\ell$ have the
following left coprime factorizations
\begin{equation}\label{eq:7.110a}
 G_\ell (\lambda)^{-1}=b_\ell(\lambda)^{-1}\varphi_\ell(\lambda),
\end{equation}
\begin{equation}\label{eq:7.110b}
 G_\ell (\lambda)^{-1}H_\ell(\lambda) =\wt b_\ell(\lambda)^{-1}\wt\varphi_\ell(\lambda),
\end{equation}
where $b_\ell,\,\wt b_\ell \in {\mathcal S}_{in}^{p\times p}$ are
Blaschke-Potapov factors of degree $\kappa'$ and $\varphi_\ell,\,
\wt\varphi_\ell \in H_{\infty}^{p\times q}$. Then $G_\ell$ admits
the left coprime factorization
\begin{equation}\label{eq:7.110c}
 G_\ell (\lambda)= \varphi_\ell (\lambda)^{-1}b_\ell(\lambda) ,
\end{equation}
and it follows from the first part of the proof that
\[
M_{\pi} (\varphi_\ell^{-1}, \Omega_+)=M_{\pi} (G_\ell,
\Omega_+)=0.
\]
Then as  $\Vert\varphi_\ell(\mu)^{-1}\Vert=\Vert G_\ell(\mu)\Vert$ a.e. on $\Omega_0$, the maximum principle  implies that $\varphi_\ell^{-1}\in H_\infty^{\ptp}$. By Lemma~\ref{lem:7.7}, \[
M_{\pi} (\wt b_\ell b_\ell^{-1}\varphi_\ell, \Omega_+)=M_{\pi}
(\wt b_\ell \wt b_\ell^{-1}\wt\varphi_\ell, \Omega_+)=M_{\pi}
(\wt\varphi_\ell, \Omega_+)=0.
\]
Since $\varphi_\ell^{-1}\in H_\infty^{\ptp}$, this implies $\wt b_\ell b_\ell^{-1}\in H_\infty^{\ptp}$ which shows that $b_\ell$ and $\widetilde{b}_\ell$ coincide up to a constant right unitary factor.

Now it follows from (\ref{eq:7.110b}) and (\ref{eq:7.110c}) that
$H_\ell= \varphi_\ell^{-1}\wt \varphi_\ell$. Since the
factorization~\eqref{eq:7.110b} is also left coprime over
$\Omega_+$,
\[
\begin{split}
\rank \left[\begin{array}{cc}
 G_\ell(\lam) & H_\ell(\lam)
\end{array}\right]
&=\rank \left[\begin{array}{cc}
 \varphi_\ell(\lam)^{-1}b_\ell (\lam)& \varphi_\ell(\lam)^{-1}\wt \varphi_\ell (\lam)
\end{array}\right]\\
&=\rank \left[\begin{array}{cc}
b_\ell (\lam)& \wt \varphi_\ell (\lam)
\end{array}\right]=p
 \quad \text{for} \ \lam\in\Omega_+.
 \end{split}
\]
This proves the implication $\Leftarrow$ in (i) and completes the proof of       (i). Assertion (ii) follows from (i) by passing to adjoints.
\end{proof}
In the rational case this statement can be found in \cite[Theorem 11.1.4]{BGR}.
\subsection{A class of generalized Schur mvf's}
In this section we study the
class of mvf's
\begin{equation}
\label{M5}
S=\begin{bmatrix}s_{11}&s_{12}\\s_{21}&s_{22}\end{bmatrix}
\in\cS^{\mtm}_\kappa \quad\text{for which}\quad s_{21}\in{\mathcal S}^{\qtp}_\kappa,
\end{equation}
where the indicated block decomposition of $S$ is conformal with $j_{pq}$.

\begin{thm}
\label{thm:11.2} If
$S\in\cS^{\mtm}_\kappa$ meets the constraint (\ref{M5})
 and the  Kre\u{\i}n-Langer factorizations of $s_{21}$ are
\begin{equation}\label{M6}
 s_{21}={\gb}_{\ell}^{-1}{\gs}_{\ell}={\gs}_r{\gb}_r^{-1},
 \end{equation}
then ${\gb}_{\ell}s_{22}\in \cS^{q\times q}$ and $s_{11}{\gb}_r\in
\cS^{p\times p}$.
\end{thm}

\begin{proof} Let
$$
G=[s_{21}\quad s_{22}] \quad\text{and}\quad H=\begin{bmatrix}s_{11}\\s_{21}
\end{bmatrix}.
$$
Then the kernels
$$
{\mathsf
\Lambda}^G_\omega(\lambda)=\frac{I_q-G(\lam)G(\omega)^*}{\rho_\omega(\lambda)}=
[0\quad I_q] {\mathsf
\Lambda}^S_\omega(\lambda)\begin{bmatrix}0\\I_q\end{bmatrix}
$$
and
$$
{\mathsf
\Delta}^H_\omega(\lam)=\frac{I_p-H(\omega)^*H(\lambda)}{\rho_\omega(\lambda)}=
[I_p\quad 0] {\mathsf
\Delta}^S_\omega(\lambda)\begin{bmatrix}I_p\\0\end{bmatrix}
$$
have at most $\kappa$ negative squares in $\gh^+_S\subset\Omega_+$.
On the other hand, since $s_{21}\in{\mathcal S}^{\qtp}_\kappa$, the
formulas
$$
{\mathsf  \Lambda}^G_\omega(\lambda)={\mathsf
\Lambda}^{s_{21}}_\omega(\lambda)-
\frac{s_{22}(\lam)s_{22}(\omega)^*}{\rho_\omega(\lambda)}
$$
and
$$
{\mathsf  \Delta}^H_\omega(\lam)={\mathsf
\Delta}^{s_{21}}_\omega(\lam)-\frac{s_{11}(\omega)^*s_{11}(\lambda)}
{\rho_\omega(\lambda)}
$$
imply that ${\mathsf  \Lambda}^G_\omega(\lambda)$ and ${\mathsf
\Delta}^H_\omega(\lambda)$ have at least $\kappa$ negative squares.
Therefore, $G\in\cS^{q\times m}_\kappa$, $H\in\cS_\kappa^{m\times
p}$ and hence
\begin{equation}
\label{M8}
M_{\pi}(G,\Omega_+)=M_{\pi}(s_{21},\Omega_+)=M_{\pi}(H,\Omega_+)=\kappa.
\end{equation}
Consequently, in view of Lemma \ref{lem:7.7} and the factorization
(\ref{M6}),
$$
M_\pi({\gb}_\ell G, \Omega_+)=M_\pi({\gb}_\ell s_{21}, \Omega_+)=
M_\pi({\gs}_\ell, \Omega_+)=0
$$
and hence
$$
{\gb}_\ell G=[{\gb}_\ell s_{21}\quad {\gb}_\ell s_{22}]\in
H_\infty^{m\times q}.
$$
Similarly, by Lemma \ref{lem:7.7},
$$
M_\pi(H{\gb}_r,\Omega_+)=M_\pi(s_{21}{\gb}_r,\Omega_+)
=M_\pi({\gs}_r,\Omega_+)=0.
$$
Therefore, $s_{11}{\gb}_r\in H_\infty^{\ptp}$. By the maximum principle,
${\gb}_\ell s_{22}\in \cS^{\qtq}$ and
$s_{11}{\gb}_r\in \cS^{\ptp}$.
\end{proof}

\begin{corollary}
\label{cor:11.3a}
If $S\in\cS^{\mtm}_\kappa$ meets the constraint (\ref{M5}) and  its Kre\u{i}n-Langer factorization
is\begin{equation}\label{BlKLfact}
    S=B_\ell^{-1}S_\ell=S_rB_r^{-1},
\end{equation}
then
$$
[0\quad \gb_\ell]B_\ell^{-1}\in \cS^{q\times m}\quad\text{and}\quad
B_r^{-1}\begin{bmatrix}\gb_r\\0\end{bmatrix}\in \cS^{m\times p}.
$$
\end{corollary}

\begin{proof}
Since $\textup{rng }L(B_\ell^{-1},\lam)= \textup{rng }L(S,\lam)$ for all $\lambda\in\Omega_+$,
Lemma \ref{lem:7.7} implies that
$$
M_\pi([0\quad \gb_\ell]B_\ell^{-1},\Omega_+)= M_\pi([0\quad
\gb_\ell]S,\Omega_+) =M_\pi([\gs_\ell\quad
\varphi_2b_2],\Omega_+)=0.
$$
Thus, $[0\quad \gb_\ell]B_\ell^{-1}\in H_\infty^{q\times m}$, and hence is in $\cS^{q\times m}$ by the maximum principle.

Similarly, since $\textup{ker }L(B_r^{-1},\lambda)=\textup{ker
}L(S,\lambda)$ for all $\lambda\in\Omega_+$, Lemma \ref{lem:7.7} implies that
$$
M_\pi(B_r^{-1}\begin{bmatrix}\gb_r\\0\end{bmatrix},\Omega_+)=
M_\pi(S\begin{bmatrix}\gb_r\\0\end{bmatrix},\Omega_+)=
M_\pi(\begin{bmatrix}s_{11}\gb_r\\ \gs_r\end{bmatrix},\Omega_+)=0,
$$
which yields the second assertion.
\end{proof}

\begin{lem}\label{lem:11.4}
If $S\in {\mathcal S}_{\kappa}^{m\times m}$ meets the constraint~\eqref{M5}, admits Kre\u{\i}n-Langer factorization~\eqref{BlKLfact}
and $\left[
       \begin{array}{cc}
         0 & I_q \\
       \end{array}
     \right]B_\ell^{-1} h\in H_2^q
$ for some $h\in H_2^m$ then $B_\ell^{-1}h\in H_2^m$.
 \end{lem}

\begin{proof}
Assume that $\lambda_0\in\Omega_+$ is a pole of $f=B_\ell^{-1}h$ and that
$$
f(\lambda) =\frac{f_{-k}}{(\lam-\lam_0)^k}+\cdots
+\frac{f_{-1}}{\lam-\lam_0}+\cdots \quad\text{in a deleted
neighborhood of}\ \lam_0
$$
and
$$
h(\lam)=h_0+\cdots+h_{k-1}(\lam-\lam_0)^{k-1}+\cdots \quad\text{in a
neighborhood of}\ \lam_0.
$$
Then
$$
\mathbf{f}=\begin{bmatrix}f_{-k}\\ \vdots\\f_{-1}\end{bmatrix}=
L(B_\ell^{-1},\lam_0)\mathbf{h} \quad
\text{where}\quad \mathbf{h}=\begin{bmatrix}h_0\\ \vdots\\h_{k-1}\end{bmatrix}.
$$

Since $[0\quad I_q]f\in H_2^q$, it follows that
$\mathbf{h}\in\textup{ker }L(\left[
       \begin{array}{cc}
         0 & I_q \\
       \end{array}
     \right]B_\ell^{-1}, \lam_0)$.
     Clearly,
\begin{equation}\label{kerLB}
 \ker L(\left[
       \begin{array}{cc}
         0 & I_q \\
       \end{array}
     \right]B_\ell^{-1}, \lam_0)\supseteq    \ker L(B_\ell^{-1},
     \lam_0).
\end{equation}
The equality~\eqref{M8} and Remark~\ref{rem:3.1} yield
\[
M_\pi(\left[
       \begin{array}{cc}
         s_{21} & s_{22} \\
       \end{array}
     \right],\Omega_+)=M_\pi(B_\ell^{-1},\Omega_+)=\kappa
\]
and, since
\begin{equation}\label{eq:kerLB}
\rank  L(\left[
       \begin{array}{cc}
         s_{21} & s_{22} \\
       \end{array}
     \right], \lam_0)\le \rank L(\left[
       \begin{array}{cc}
         0 & I_q \\
       \end{array}
     \right]B_\ell^{-1}, \lam_0)\le  \rank L(B_\ell^{-1}, \lam_0),
\end{equation}
equality holds in~\eqref{eq:kerLB} and~\eqref{kerLB}. Thus
$\mathbf{f}=L(B_\ell^{-1}, \lam_0)\mathbf{h}=0$, and hence
$B_\ell^{-1}h\in H_2^m$.
\end{proof}

\begin{thm}
\label{thm:jun1a8}
If $\gs\in{\mathcal S}^{\qtp}_\kappa$ with Kre\u{i}n-Langer factorizations
$\gs=\gb_\ell^{-1}\gs_\ell=\gs_r\gb_r^{-1}$, then:
\begin{enumerate}
\item[\rm(1)] There exists a set of mvf's
$c_\ell^\circ\in H_\infty^{\qtq}$, $d_\ell^\circ\in H_\infty^{\ptq}$,
$c_r^\circ\in H_\infty^{\ptp}$ and $d_r^\circ\in H_\infty^{\ptq}$ such that
\begin{equation}
\label{eq:jun9a9}
\gb_\ell c_\ell^\circ+\gs_\ell d_\ell^\circ=I_q\quad\text{and}\quad
c_r^\circ \gb_r+d_r^\circ \gs_r=I_p.
\end{equation}
\item[\rm(2)] The mvf's $c_\ell\in H_\infty^{\qtq}$,
$d_\ell\in H_\infty^{\ptq}$,
$c_r\in H_\infty^{\ptp}$ and $d_r\in H_\infty^{\ptq}$ are solutions to the
equations in (\ref{eq:jun9a9}) if and only if
\begin{equation}
\label{eq:jun9b9}
c_\ell =c_\ell^\circ+\gs_r\psi\quad\text{and}\quad
d_\ell =d_\ell^\circ-\gb_r\psi\quad\textrm{for some}\ \psi\in H_\infty^{\ptq}
\end{equation}
and
\begin{equation}
\label{eq:jun9c9}
c_r=c_r^\circ +\phi\gs_\ell\quad\text{and}\quad d_r=d_r^\circ -\phi \gb_\ell
\quad\textrm{for some}\ \phi\in H_\infty^{\ptq}
\end{equation}
\item[\rm(3)] There exists a set of mvf's
$c_\ell\in H_\infty^{\qtq}$, $d_\ell\in H_\infty^{\ptq}$,
$c_r\in H_\infty^{\ptp}$ and $d_r\in H_\infty^{\ptq}$ such that
\begin{equation}
\label{eq:jun1a8}
\begin{bmatrix}c_r&d_r\\-\gs_\ell&\gb_\ell\end{bmatrix}
\begin{bmatrix}\gb_r&-d_\ell\\ \gs_r&c_\ell\end{bmatrix}=\begin{bmatrix}I_p&0\\
0&I_q\end{bmatrix}.
\end{equation}
Moreover, this set of mvf's satisfies the auxiliary equalities
\begin{eqnarray}
\label{eq:jun1b8}
\gb_rc_r+d_\ell \gs_\ell=I_p,\quad \gb_rd_r-d_\ell b_\ell =0,\\
\gs_rc_r-c_\ell \gs_\ell=0\quad\text{and}\quad \gs_rd_r+c_\ell \gb_\ell=I_q .
\label{eq:jun1c8}
\end{eqnarray}
\end{enumerate}
\end{thm}

\begin{proof} The first assertion is immediate from
Lemmas \ref{Corona} and \ref{Corona1}. If also
$\gb_\ell c_\ell+\gs_\ell d_\ell=I_q$, then
$$
\gb_\ell (c_\ell-c_\ell^\circ)+\gs_\ell(d_\ell-d_\ell^\circ)=0.
$$
Thus, if $\psi=-\gb_r^{-1}(d_\ell-d_\ell^\circ)$, it is readily seen that
$c_\ell-c_\ell^\circ=\gs_r\psi$ and hence that
$$
-c_r^\circ(d_\ell-d_\ell^\circ)+d_r^\circ(c_\ell-c_\ell^\circ)=
c_r^\circ \gb_r\psi+d_r^\circ \gs_r\psi=\psi.
$$
Therefore, $\psi\in H_\infty^{\ptq}$. Similar considerations serve to justify
the necessity of  (\ref{eq:jun9c9}). The sufficiency is self-evident.

Next, in view of (\ref{eq:jun9b9}) and (\ref{eq:jun9c9}) the matrix product
on the left of (\ref{eq:jun1a8}) can be expressed as
$$
\begin{bmatrix}c_r&d_r\\-\gs_\ell&\gb_\ell\end{bmatrix}
\begin{bmatrix}\gb_r&-d_\ell\\ \gs_r&c_\ell\end{bmatrix}=
\begin{bmatrix}I_p&-\phi\\0&I_q\end{bmatrix}
\begin{bmatrix}c_r^\circ&d_r^\circ\\-\gs_\ell&\gb_\ell\end{bmatrix}
\begin{bmatrix}\gb_r&-d_\ell^\circ\\ \gs_r&c_\ell^\circ\end{bmatrix}
\begin{bmatrix}I_p&\psi\\
0&I_q\end{bmatrix}=\begin{bmatrix}I_p&Z\\
0&I_q\end{bmatrix},
$$
where $Z=-\phi-c_r^\circ
d_\ell^\circ+d_r^\circ c_\ell^\circ+\psi$. Formula
(\ref{eq:jun1a8}) is obtained by choosing  $\phi$ and $\psi$ to make $Z=0$.
The final set of
formulas
(\ref{eq:jun1b8}) and (\ref{eq:jun1c8}) follow by reversing the order of
multiplication on  the left hand side of formula (\ref{eq:jun1a8}).
\end{proof}

The formulation (\ref{eq:jun1a8}) was motivated by the discussion of the
Bezout identity, (3) in Section 4.1 of \cite{fran}.

\subsection{Associated pairs.}
%%%%%%%%%%%%%%%%%%%%%%%%%%%%%%%%%%%%%%%%%%%%%%%%%%%%%%%%%%%%%%%%%%%%%%%
%%%%%%%%%%%%%%%%%%%%%%%%%%%%%%%%%%%%%%%%%%%%%%%%%%%%%%%%%%%%
If $W$ belongs to the class $\cU_\kappa^\circ(j_{pq})$ which is defined
by~\eqref{eq:11.10}, then the Potapov-Ginzburg transform $S=PG(W)$ meets the
constraints  imposed in the preceding subsection. In this case, the class of
associated pairs of $W$ are defined as the inner factors in the factorizations
in~\eqref{eq:0.5}.
\begin{thm}\label{thm:11.5}
Let $W\in \cU_\kappa^\circ(j_{pq})$%satisfy assumption~\eqref{eq:11.10}
, and let $\{b_1,
b_2\}\in ap(W)$. Then $W$ can be expressed in terms of the factors in
(\ref{eq:0.5}) as
\begin{equation}
\label{eq:may30a8}
W=\begin{bmatrix}w_{11}& w_{12}\\w_{21}& w_{22}\end{bmatrix}=
\begin{bmatrix}b_1&0\\0&b_2^{-1}\end{bmatrix}
\begin{bmatrix}\varphi_1^{-*}&0\\0&\varphi_2^{-1}\end{bmatrix}
\begin{bmatrix}\gb_r^* &-\gs_r^*\\-\gs_\ell&\gb_\ell\end{bmatrix}\quad
\text{a.e. in}\ \Omega_0.
\end{equation}
Moreover,
\begin{equation}\label{eq:sr}
    I_p-\gs_r^*\gs_r=\varphi_1^*\varphi_1 \quad \text{and}\quad
    I_q-\gs_\ell\gs_\ell^*=\varphi_2\varphi_2^*\quad \text{a.e. in }\ \Omega_0.
\end{equation}
\end{thm}

\begin{proof}
The asserted identities follow easily from the formulas
\begin{equation}\label{eq:w11star}
    \begin{bmatrix}
w_{11}^\#\\w_{12}^\#\end{bmatrix}=\begin{bmatrix}\gb_r\\
-\gs_r\end{bmatrix}\varphi_1^{-1}b_1^{-1}
\quad\text{and}\quad
\begin{bmatrix}w_{21}&w_{22}\end{bmatrix}
=b_2^{-1}\varphi_2^{-1}\begin{bmatrix}-\gs_\ell &
\gb_\ell\end{bmatrix}
\end{equation}
and the fact that $S^\#S=SS^\#=I_m$ in $\gh_S\cap\gh_{S^{\#}}$.
\end{proof}

\begin{thm}\label{thm:11.7}
Let $W\in \cU_\kappa^\circ(j_{pq})$, % satisfy the assumption~\eqref{eq:11.10}
 let $\{b_1,
b_2\}\in ap(W)$, and let $c_r$, $d_r$, $c_\ell$ and $d_\ell$ be as in
Theorem \ref{thm:jun1a8} (3) for $\gs=s_{21}$ and let
\begin{equation}
\label{eq:jun10a9}
K=(-w_{11}d_\ell+w_{12}c_\ell)(-w_{21}d_\ell+w_{22}c_\ell)^{-1}.
\end{equation}
Then $K$ belongs to $H_\infty^{\ptq}$ and admits the representation
\begin{equation}\label{eq:11.28}
    K=(-w_{11}d_\ell+w_{12}c_\ell)\varphi_2b_2,
\end{equation}
and the dual representation
\begin{equation}\label{eq:11.29}
    K=b_1\varphi_1(c_rw_{21}^\#-d_rw_{22}^\#).
\end{equation}
Moreover, if $c_\ell^\circ$ and $d_\ell^\circ$ is a
second set of mvf's
such that (\ref{eq:jun9a9})
holds and correspondingly
\begin{equation}\label{eq:Kcirc}
    K^\circ=(-w_{11}d_\ell^\circ +w_{12}c_\ell^\circ)
(-w_{21}d_\ell^\circ +w_{22}c_\ell^\circ)^{-1},
\end{equation}
then
\begin{equation}
\label{eq:jun9d9}
(b_1\varphi_1)^{-1}(K-K^\circ)(\varphi_2b_2)^{-1}\in H_\infty^{\ptq}.
\end{equation}
There is a choice of mvf's $c_\ell$ and $d_\ell$,  such that the
corresponding mvf $K$ admits a
pseudocontinuation  of bounded type in $\Omega_-$.
\end{thm}

\begin{proof}
The formula $S=PG(W)$ implies that
$$
\begin{bmatrix}w_{11}&w_{12}\end{bmatrix}= \begin{bmatrix}s_{11}&s_{12}
\end{bmatrix}\begin{bmatrix}I_p&0\\w_{21}&w_{22}\end{bmatrix}.
$$
Therefore, since
$$
(-w_{21}d_\ell+w_{22}c_\ell)\varphi_2b_2=I_q
$$
by (\ref{eq:w11star}), it is readily seen that
\begin{eqnarray*}
K&=&(-w_{11}d_\ell+w_{12}c_\ell)\varphi_2b_2\\
&=&\begin{bmatrix}s_{11}&s_{12}
\end{bmatrix}\begin{bmatrix}I_p&0\\w_{21}&w_{22}\end{bmatrix}\begin{bmatrix}
-d_\ell\\c_\ell\end{bmatrix}\varphi_2b_2
=\begin{bmatrix}s_{11}&s_{12}
\end{bmatrix}\begin{bmatrix}-d_\ell\varphi_2b_2\\I_q\end{bmatrix}
\end{eqnarray*}
and hence that
\begin{equation}
\label{eq:jun11a9}
S\begin{bmatrix}-d_\ell\varphi_2b_2\\I_q\end{bmatrix}=\begin{bmatrix}K\\
c_\ell\varphi_2b_2\end{bmatrix}.
\end{equation}
Consequently,
$$
[0\quad I_q]S\begin{bmatrix}-d_\ell\varphi_2b_2\\I_q\end{bmatrix}h=
c_\ell\varphi_2b_2h
$$
belongs to $H_2^q$ for every $h\in H_2^q$. Thus, Lemma \ref{lem:11.4}
implies that
$$
[I_p\quad 0]Sh=Kh
$$
belongs to $H_2^p$ for every $h\in H_2^q$. Therefore, $K\in H_\infty^{\ptq}$.

The identity
\[
\begin{bmatrix}
  c_r & -d_r
\end{bmatrix}
W^\#j_{pq}W
\begin{bmatrix}
  -d_\ell\\ c_\ell
\end{bmatrix}
=\begin{bmatrix}
  c_r & -d_r
\end{bmatrix}
j_{pq}
\begin{bmatrix}
  -d_\ell\\ c_\ell
\end{bmatrix}=0
\]
implies that
\[
(c_r w_{11}^\#-d_rw_{12}^\#)(-w_{11}d_\ell+w_{12}c_\ell) =
(c_r w_{21}^\#- d_rw_{22}^\#)(-w_{21}d_\ell+w_{22}c_\ell),
\]
and hence that $K$ admits the dual representation
\begin{equation}\label{eq:11.29a}
    K=(c_rw_{11}^\#-d_rw_{12}^\#)^{-1}(c_r
w_{21}^\#-d_rw_{22}^\#)=b_1\varphi_1(c_rw_{21}^\#-d_rw_{22}^\#)
\end{equation}
which coincides with~\eqref{eq:11.29}.

It follows from~\eqref{eq:11.28}, \eqref{eq:Kcirc}, \eqref{eq:w11star},
\eqref{eq:jun9b9} and \eqref{eq:sr} that
\[
\begin{split}
K-K^\circ&=\begin{bmatrix}w_{11}&w_{12}\end{bmatrix}
\begin{bmatrix}-(d_\ell-d_\ell^\circ)\\ c_\ell-c_\ell^\circ\end{bmatrix}
\varphi_2b_2
=b_1\varphi_1^{-\#}
\begin{bmatrix} b_r^\#& -s_r^\#\end{bmatrix}
\begin{bmatrix} b_r\\ s_r\end{bmatrix}\psi\varphi_2b_2\\
&=b_1\varphi_1^{-\#}
 (b_r^\#b_r -s_r^\#s_r)\psi
\varphi_2b_2=b_1\varphi_1\psi \varphi_2b_2\quad\textrm{on}\
\Omega_0.
\end{split}
\]
This justifies (\ref{eq:jun9d9}).

Finally, to verify the last statement, first note that, in view
of~\eqref{eq:0.5},
formula~\eqref{eq:11.28} can be rewritten as
\[
K=(-w_{11}d_\ell+w_{12}c_\ell)\gb_\ell s_{22}.
\]
Therefore, since the
mvf's $w_{11}$, $w_{12}$, $\gb_\ell$, $ s_{22}$  admit
pseudocontinuations of bounded type in $\Omega_-$, it remains only to show
that there is a choice of mvf's $c_\ell$ and $d_\ell$ of the
form~\eqref{eq:jun9b9} that admit
pseudocontinuations to $\Omega_-$. If $c_\ell^\circ$ and
$d_\ell^\circ$ is a a fixed pair of mvf's that
satisfy~\eqref{eq:jun9a9} and $\Omega_+=\dD$, then
\[
\gb_r^{-1}d_\ell^\circ=\psi_++\psi_-,
\]
with $\psi_+\in
H_2^{\ptq}$, and $\psi_-\in L_\infty^{\ptq}\cap(H_2^{\ptq})^\perp$, since
$b_r$ is a finite Blaschke-Potapov product.
Therefore,
\[
\psi_+\in L_\infty^{\ptq}\cap H_2^{\ptq}\subset H_\infty^{\ptq},
\]
by the maximum principle and hence the particular choice
\[
d_\ell=d_\ell^\circ-\gb_r\psi_+=\gb_r\psi_-
\]
admits a pseudocontinuation to $\Omega_-$
by a matrix version of the Douglas-Shapiro-Shields condition \cite{DSS},
due to Fuhrmann~\cite[Theorem 1]{Fuhr74}.
\end{proof}

%%%%%%%%%%%%%%%%%%%%%%%%%%%%%%%%%%%%%%%%%%%%%%%%%%%%%%%%%%%%%%%%
\subsection{Factorization of the resolvent matrix}
%%%%%%%%%%%%%%%%%%%%%%%%%%%%%%%%%%%%%%%%%%%%%%%%%%%
\begin{thm}
\label{thm:11.8} Let $W\in\cU_\kappa^\circ(j_{pq})$, %$s_{21}
%\in{\mathcal S}^{\qtp}_\kappa$,
$\{b_1, b_2\}\in ap(W)$  let $K$ be defined as in Theorem
\ref{thm:11.7}, and let $c_r$, $d_r$, $c_\ell$ and $d_\ell$ be as in
Theorem \ref{thm:jun1a8} (3). Then $W$ admits the factorizations
\begin{equation}
\label{eq:11.34}
W=\Theta\,\Phi\quad in\  \Omega_+\quad\text{and}\quad W=\wt{\Theta}\,\wt{\Phi}
\quad in\ \Omega_-,
\end{equation}
where
\begin{equation}
\label{eq:11.33}
\Theta=\begin{bmatrix}b_1&Kb_2^{-1}\\0&b_2^{-1}\end{bmatrix}\ in \ \Omega_+,
\quad
\wt{\Theta}=\begin{bmatrix}b_1&0\\K^\# b_1&b_2^{-1}\end{bmatrix}\ in \ \Omega_-
\end{equation}
\begin{equation}
\label{eq:11.35}
\Phi= \begin{bmatrix}\varphi_{11}& \varphi_{12}\\
\varphi_{21}&\varphi_{22}\end{bmatrix}=
\begin{bmatrix}\varphi_1&0\\0&\varphi_2^{-1}\end{bmatrix}
\begin{bmatrix}c_r&d_r\\-\gs_\ell&\gb_\ell\end{bmatrix}\ in \ \Omega_+,
\end{equation}
\begin{equation}
\label{eq:11.36}
\wt{\Phi}= \begin{bmatrix}\wt\varphi_{11}& \wt\varphi_{12}\\
\wt\varphi_{21}&\wt\varphi_{22}\end{bmatrix}=
\begin{bmatrix}\varphi_1^{-\#}&0\\0&\varphi_2^\#\end{bmatrix}
\begin{bmatrix}\gb_r^\#&-\gs_r^\#\\ d_\ell^\#&c_\ell^\#\end{bmatrix} \ in
\ \Omega_-
\end{equation}
and
\begin{equation}
\label{eq:may13d9}
\widetilde{\Theta}^\#j_{pq}\Theta=\widetilde{\Phi}^\#j_{pq}\Phi=j_{pq}\ in\ \
\Omega_+.
\end{equation}
Moreover,
\begin{equation}
\label{eq:jun1d8}
\Phi^{-1}=\begin{bmatrix}\gb_r&-d_\ell\\ \gs_r&c_\ell\end{bmatrix}
\begin{bmatrix}\varphi_1^{-1}&0\\0&\varphi_2\end{bmatrix}\ in \ \Omega_+,\quad
\wt{\Phi}^{-1}=\begin{bmatrix}c_r^\#&\gs_\ell^\#\\ -d_r^\#&\gb_\ell^\#
\end{bmatrix}
\begin{bmatrix}\varphi_1^{\#}&0\\0&\varphi_2^{-\#}\end{bmatrix}\ in\
\Omega_-,
\end{equation}
\begin{equation}\label{eq:11.30}
W^{-1}\begin{bmatrix}
  K \\ I_q
\end{bmatrix}
b_2^{-1}=\begin{bmatrix}-d_\ell\\ c_\ell\end{bmatrix}\varphi_2
\in H_\infty^{m\times q},
\quad
W^{-1}\begin{bmatrix}
  I_p \\ K^\#
\end{bmatrix}b_1=\begin{bmatrix}c_r^\#\\ -d_r^\#\end{bmatrix}\varphi_1^\#
\in H_\infty^{m\times p}(\Omega_-)
\end{equation}
and
$S=PG(W)$ admits the representations
\begin{equation}
\label{eq:11.41}
S=\begin{bmatrix}s_{11}& s_{11}d_\ell\varphi_2b_2+K\\ s_{21}&s_{21}
d_\ell\varphi_2b_2+c_\ell\varphi_2b_2
\end{bmatrix}=\begin{bmatrix}b_1\varphi_1c_r+b_1\varphi_1d_rs_{21}
&K+b_1\varphi_1d_rs_{22}\\s_{21}& s_{22}\end{bmatrix}.
\end{equation}
\end{thm}

\begin{proof}
The evaluations
\[
\begin{split}
w_{11}\gb_r+w_{12}\gs_r&=s_{11}\gb_r=b_1\varphi_1,\\
w_{21}\gb_r+w_{22}\gs_r&=(w_{21}+w_{22}s_{21})\gb_r=0,\\
-w_{11}d_\ell+w_{12}c_\ell&=
Kb_2^{-1}\varphi_2^{-1},\\
-w_{21}d_\ell+w_{22}c_\ell &=b_2^{-1}\varphi_2^{-1}
\end{split}
\]
lead easily to the formula
$$
W\begin{bmatrix}\gb_r&-d_\ell\\ \gs_r&c_\ell\end{bmatrix}=\begin{bmatrix}
b_1\varphi_1&Kb_2^{-1}\varphi_2^{-1}\\0&b_2^{-1}\varphi_2^{-1}\end{bmatrix};\quad\text{i.e.,}\quad
W=\Theta\begin{bmatrix}\varphi_1&0\\0&\varphi_2^{-1}\end{bmatrix}
\begin{bmatrix}\gb_r&-d_\ell\\ \gs_r&c_\ell\end{bmatrix}^{-1}.
$$
Formula~\eqref{eq:11.35} for $\Phi$ in \eqref{eq:11.34} is easily verified
with the help of Theorem~\ref{thm:jun1a8}. The second factorization
formula follows from the first and the observation that
$$
W=j_{pq}(W^\#)^{-1}j_{pq},\quad \wt{\Theta}=j_{pq}(\Theta^\#)^{-1}j_{pq}\quad
\text{and}\quad \wt{\Phi}=j_{pq}(\Phi^\#)^{-1}j_{pq}.
$$

Moreover, the first formula in (\ref{eq:11.41}) is equivalent
to formula (\ref{eq:jun11a9}). The second follows by much the same sort of
manipulations:
$$
s_{11}=b_1\varphi_1(c_rw_{11}^\#-d_rw_{12}^\#)(w_{11}^\#)^{-1}=
b_1\varphi_1c_r+b_1\varphi_1d_rs_{21}
$$
and
\begin{eqnarray*}
s_{12}&=&b_1\varphi_1(c_rw_{11}^\#-d_rw_{12}^\#)(w_{11}^\#)^{-1}w_{21}^\#
\\
&=&
b_1\varphi_1(c_rw_{21}^\#-d_rw_{22}^\#+d_rs_{22})=K+b_1\varphi_1d_rs_{22}.
\end{eqnarray*}
Finally, formulas (\ref{eq:11.34}), (\ref{eq:11.33}) and (\ref{eq:jun1d8})
imply that
$$
W^{-1}\begin{bmatrix}
  K \\ I_q
\end{bmatrix}
b_2^{-1}=W^{-1}\Theta\begin{bmatrix}0\\I_q\end{bmatrix}=
\Phi^{-1}\begin{bmatrix}0\\I_q\end{bmatrix}=\begin{bmatrix}-d_\ell\\c_\ell
\end{bmatrix}\varphi_2$$
and
$$
W^{-1}\begin{bmatrix}I_p\\ K^\#\end{bmatrix}b_1=
W^{-1}\wt{\Theta}\begin{bmatrix}I_p\\0\end{bmatrix}=\wt{\Phi}^{-1}
\begin{bmatrix}I_p\\0\end{bmatrix}=
\begin{bmatrix}c_r^\#\\-d_r^\#\end{bmatrix}\varphi_1^\#
$$
which serves to justify (\ref{eq:11.30})
\end{proof}

\begin{corollary}
\label{cor:jun2a8}
In the setting of Theorem \ref{thm:11.8},
\begin{equation}
\label{eq:jun2a8}
WH_2^m=\Theta \begin{bmatrix}\varphi_1 & 0\\0&\varphi_2^{-1}\end{bmatrix}H_2^m.
\end{equation}
\end{corollary}

\begin{proof}
 Theorem \ref{thm:jun1a8} and Theorem \ref{thm:11.8} imply that
$$
W=\Theta \,\Phi^\circ E\quad\text{where}\ \Phi^\circ=
\begin{bmatrix}\varphi_1 & 0\\0&\varphi_2^{-1}\end{bmatrix}\quad\text{and}\
E^{\pm 1}\in H_\infty^{\mtm}.
$$
Therefore,
$$
WH_2^m=\Theta \Phi^\circ EH_2^m
\subseteq \Theta \Phi^\circ H_2^m
=\Theta \Phi^\circ EE^{-1}H_2^m\subseteq
\Theta \Phi^\circ EH_2^m.
$$
Thus, equality must prevail throughout.
\end{proof}
\begin{corollary}
\label{cor:11.10}
Let $W\in  \cU_\kappa^\circ(j_{pq})$ and the Kre\u{\i}n-Langer factorizations
of $S=PG(W)$ are
\[
S(\lam)=B_{\ell}^{-1}(\lam)S_{\ell}(\lam)=S_r(\lam)B_r^{-1}(\lam),
\]
then, in the setting of Theorem~\ref{thm:11.7},
\begin{equation}
\label{eq:11.45}
 B_{\ell}\left[\begin{array}{cc}
  K-s_{12} \\
   -s_{22}
\end{array}  \right]b_2^{-1}\in H_{\infty}^{m\times q},\quad
B_r^\#\left[\begin{array}{c}
  s_{11}^\#  \\
  s_{12}^\#-K^\#
\end{array}  \right]b_1\in H_{\infty}^{m\times p}(\Omega_-).
\end{equation}
\end{corollary}
\begin{proof}
The relations in~\eqref{eq:11.45} are implied by~\eqref{eq:11.41}, since
\[
 B_{\ell}\left[\begin{array}{cc}
  K-s_{12} \\
   -s_{22}
\end{array}  \right]b_2^{-1}= -B_{\ell}\left[\begin{array}{cc}
  s_{11} \\
   s_{21}
\end{array}  \right]d_\ell\varphi_2 - B_{\ell}\left[\begin{array}{cc}
 0 \\
 c_\ell\varphi_2
\end{array}  \right]\in H_{\infty}^{m\times q},
\]
\[
B_r^\#\left[\begin{array}{c}
  s_{11}^\#  \\
  s_{12}^\#-K^\#
\end{array}  \right]b_1=B_r^\#\left[\begin{array}{cc}
  -s_{21}^\#  \\
  -s_{22}^\#
\end{array}  \right]d_r^\#\varphi_1^\#+
B_r^\#\left[\begin{array}{cc}
   c_r^\#\varphi_1^\# \\
  0
\end{array}  \right]\in H_{\infty}^{m\times p}(\Omega_-).
\]
\end{proof}
\begin{corollary}
\label{cor:jun2b8} If in the setting of Theorem \ref{thm:11.8},
$\wt\varphi_{11}\in H_\infty^{p\times p}(\Omega_-)$,\,
  $\wt\varphi_{12}\in H_\infty^{p\times q}(\Omega_-)$,
$\varphi_{21}\in H_\infty^{q\times p}(\Omega_+)$,\, $\varphi_{22}\in
H_\infty^{q\times q}(\Omega_+)$, then:
\begin{enumerate}
\item[(1)]
the mvf $\begin{bmatrix}w_{11}& w_{12}\end{bmatrix}$ admits a right
coprime factorization over $\Omega_-$
\begin{equation}\label{eq:7.43}
 \begin{bmatrix}w_{11} & w_{12}\end{bmatrix}
 =b_1\begin{bmatrix}\wt\varphi_{11}& \wt\varphi_{12}\end{bmatrix};
\end{equation}
\item[(2)]
the mvf
 $\begin{bmatrix}w_{21} & w_{22}\end{bmatrix}$ admits a left coprime
factorization over $\Omega_+$
\begin{equation}\label{eq:7.41}
 \begin{bmatrix}w_{21}& w_{22}\end{bmatrix}=
b_2^{-1}\begin{bmatrix}\varphi_{21}&\varphi_{22}\end{bmatrix}.
\end{equation}
\end{enumerate}
\end{corollary}
%%%%%%%%%%%%%%%%%%%%%%%%%%%%%%%%%%%%%%%%%%%%%%%%%%%%%%%%%%%%%%%%%%%%%%%
\subsection{Characterization of $\cK(W)$ spaces.}
The next theorem characterizes $\cK(W)$ spaces in
terms of the Kre\u{\i}n-Langer factorizations of  $S=PG(W)$.
\begin{thm}\label{cHW}
Let $W\in  \cU_\kappa(j_{pq})$, let the Kre\u{\i}n-Langer factorizations
of $S=PG(W)$ be
\[
S(\lam)=B_{\ell}(\lam)^{-1}S_{\ell}(\lam)=S_r(\lam)B_r(\lam)^{-1}
\]
and let $h_1$, $h_2$ be a pair of measurable vvf's on $\Omega_0$ of
height $p$ and $q$, respectively. Then $h=\textup{col}(h_1,h_2)\in\cK(W)$
if and only if:
$$
B_{\ell}\left[\begin{array}{cc}
  I_p & -s_{12} \\
  0 & -s_{22}
\end{array}  \right]h\in H_2^m\quad\text{and}\quad
B_r^*\left[\begin{array}{cc}
  s_{11}^* & 0 \\
  s_{12}^* & -I_q
\end{array}  \right]h\in (H_2^m)^\perp.
$$
Moreover, in this case
\begin{equation}\label{cHWinnerPr}
\left\langle h,h\right\rangle_{\cK(W)}=\|f\|_{st}^2
-2\Re\left\langle f,\wt\Gamma_{\ell}(S^*f)\right\rangle_{st},
\end{equation}
where $f\in\cK(B_r)$, $S^*f$ and $\wt\Gamma_{\ell}(S^*f)$ are
defined by the formulas (cf.~\eqref{GammaS2})
\begin{equation}\label{fS*f}
    f:=\left[\begin{array}{cc}
  I_p & -s_{12} \\
  0 & -s_{22}
\end{array}  \right]
h,\quad
S^* f=\left[\begin{array}{cc}
  s_{11}^* & 0 \\
  s_{12}^* & -I_q
\end{array}  \right]h,
\end{equation}
$$ \wt \Gamma_{\ell} (S^*f):=\widetilde{X}_{\ell}^{-1}P_+(S^*f)\quad
\textrm{and} \quad \widetilde{X}_{\ell}:g\in\cH_*(B_\ell)\longrightarrow
P_+S^*g\in\cH(B_r).
$$
\end{thm}
\begin{proof}
The formula for the inverse Potapov-Ginzburg transform
\[
 W(\lambda)=\left[\begin{array}{cc}
     I_p & -s_{12}(\lambda) \\
      0 &  -s_{22}(\lam)
    \end{array}      \right]^{-1}
    \left[\begin{array}{cc}
            s_{11}(\lambda) &       0\\
            s_{21}(\lambda) &       -I_q
    \end{array}      \right]
\]
leads to the following representation of the kernel
$K_\omega^W(\lam)$
\[
K_\omega^W(\lam)=\left[\begin{array}{cc}
     I_p & -s_{12}(\lambda) \\
      0 &  -s_{22}(\lam)
    \end{array}      \right]^{-1}
    \frac{I_m-S(\lam)S(\om)^*}{\rho_\om(\lam)}
    \left[\begin{array}{cc}
     I_p & -s_{12}(\om) \\
      0 &  -s_{22}(\om)
    \end{array}      \right]^{-*}.
\]
This identity implies that the mapping
\[
   h \mapsto f=\left[\begin{array}{cc}
     I_p & -s_{12}(\lambda) \\
      0 &  -s_{22}(\lam)
    \end{array}      \right]h
\]
is an isometry from $\cK(W)$ onto $\cK(S)$ (cf.~\cite{ArovD97}).

Since $S_{\ell}$ is a square inner mvf,
Corollary~\ref{Hsin*in} guarantees that the inclusion $f\in\cK(S)$ is
equivalent to
the inclusion $S^*f\in\cK_*(S)$. Now the first statement of the
theorem is implied by Theorem~\ref{Hsinner}.

To verify formula~\eqref{cHWinnerPr}, consider the orthogonal
decomposition of the vector $f\in\cK(S)$ corresponding to the
fundamental decomposition of $\cK(S)$ (see~\eqref{Hsdecom})
\begin{equation}\label{Decom5}
    f=B_{\ell}^{-1}y+x,\quad y\in\cH(S_{\ell}),\,\,x\in
\cH_*(B_{\ell}),
\end{equation}
where $x\in\cH_*(B_{\ell})$ is the unique solution of the equation
\begin{equation}\label{eq:Xl}
    P_+S_{\ell}^*x=P_+S^*f\,(=P_{\cH(B_r)}S^*f).
\end{equation}
In the notation of Definition~\ref{GammaS},
\begin{equation}\label{x6}
    x=\wt\Gamma_{\ell}(S^*f)=\wt{X}_\ell^{-1}P_+S^*f.
\end{equation}
Now ~\eqref{Decom5} yields the formula
\[
\left\langle f,f\right\rangle_{\cK(S)}=\| f-x\|^2_{st}-\|x\|^2_{st}=\|f\|^2_{st}-\left\langle
f,x\right\rangle_{st}-\left\langle
x,f\right\rangle_{st},
\]
which is equivalent to~\eqref{cHWinnerPr}.
\end{proof}

\begin{remark}
If $\kappa=0$, then $B_{\ell}=B_r=I_m$, $\wt\Gamma_{\ell}=0$ and the
statement of Theorem~\ref{cHW} reduces to the characterization of
${\cH}(W)$ spaces given in Theorem 2.4 of \cite{D89} and the next
theorem reduces to Theorem 2.7 of \cite{D89} (see
also~\cite{Arva} for another proof of the latter).
\end{remark}
\begin{thm}\label{thm:11.11}
    Let $W\in \cU_\kappa^\circ(j_{pq})$,
$S=[s_{ij}]_{i,j=1}^2$ be the   Potapov-Ginzburg transform
of $W$,  let $\{b_1, b_2\}$ be an associated pair of $W$, and let
the mvf $K\in H_\infty^{p\times q}$ be defined as in
Theorem~\ref{thm:11.7}. Then $h=\textup{col }(h_1,h_2)\in\cK(W)$ if
and only if:
\begin{enumerate}
\item[\rm(1)]
$\gb_\ell s_{22}h_2\in H_2^q;$ \vskip 6pt
\item[\rm(2)]
$\gb_r^*s_{11}^*h_1\in (H_2^p)^\perp;$ \vskip 6pt
\item[\rm(3)]
$\begin{bmatrix}I_p & -K \end{bmatrix}h\in H_2^p$; \vskip 6pt
\item[\rm(4)] $\begin{bmatrix} -K^*&I_q\end{bmatrix}h\in (H_2^q)^\perp$.
\end{enumerate}
\end{thm}
\begin{proof}
{\it Necessity.}  Let $h=\textup{col}(h_1,h_2)\in\cK(W)$.
Then by Theorem~\ref{cHW}
\begin{equation}\label{eq:11.51}
    B_{\ell}\left[\begin{array}{cc}
  I_p & -s_{12} \\
  0 & -s_{22}
\end{array}  \right]h\in H_2^m,\quad
B_r^*\left[\begin{array}{cc}
  s_{11}^* & 0 \\
  s_{12}^* & -I_q
\end{array}  \right]h\in (H_2^m)^\perp
\end{equation}
and hence, by Corollary~\ref{cor:11.3a} and the formulas in \eqref{eq:0.5},
\begin{equation}
\label{eq:jun4a8}
\left[
  \begin{array}{cc}
    0 & \gb_\ell \\
  \end{array}
\right]
\left[\begin{array}{cc}
  I_p & -s_{12} \\
  0 & -s_{22}
\end{array}  \right]h=-\gb_\ell s_{22}h_2=-\varphi_2b_2h_2\in H_2^q,
\end{equation}
and
\begin{equation}
\label{eq:jun4b8}
\left[
  \begin{array}{cc}
    \gb_r^* & 0 \\
  \end{array}
\right]
\left[\begin{array}{cc}
  s_{11}^* & 0 \\
  s_{12}^* & -I_q
\end{array}  \right]h=\gb_r^*s_{11}^* h_1=\varphi_1^* b_1^* h_1
\in (H_2^p)^\perp.
\end{equation}
The first condition in~\eqref{eq:11.51} can be rewritten as
\[
     B_{\ell}\left[\begin{array}{cc}
  h_1-Kh_2 \\
  0
\end{array}  \right]+  B_{\ell}\left[\begin{array}{cc}
  K-s_{12} \\
   -s_{22}
\end{array}  \right]h_2\in H_2^m,
\]
where due to \eqref{eq:11.41}
\begin{equation}
\label{eq:jun4e8} B_{\ell}\left[\begin{array}{cc}
  K-s_{12} \\
   -s_{22}
\end{array}  \right]h_2
=B_\ell\begin{bmatrix}
-s_{11}d_\ell\\ -s_{21}d_\ell -c_\ell\end{bmatrix}\varphi_2b_2h_2.
\end{equation}
Thus, as
\[
B_\ell\begin{bmatrix}s_{11}\\
s_{21}\end{bmatrix}\in H_\infty^{m\times q}\quad\mbox{and}\quad\varphi_2b_2h_2
\in H_2^q,
\]
it is readily seen that
\[
 B_{\ell}\left[\begin{array}{cc}
  K-s_{12} \\
   -s_{22}
\end{array}  \right]h_2
\in H_2^m.
\]
Therefore,
\[
B_{\ell}\left[\begin{array}{cc}
  h_1-Kh_2 \\
  0
\end{array}  \right]\in H_2^m,
\]
which in view of Lemma~\ref{lem:11.4} implies (3).

Similarly, since
$$
B_r^*\begin{bmatrix}s_{11}^*&0\\s_{12}^*&-I_q\end{bmatrix}h=
B_r^*\begin{bmatrix}s_{21}^*d_r^*+c_r^*\\s_{22}^*d_r^*\end{bmatrix}
\varphi_1^*b_1^*h_1+B_r^*\begin{bmatrix}0&0\\K^*&-I_q\end{bmatrix}
\begin{bmatrix}h_1\\h_2\end{bmatrix}\in (H_2^m)^\perp
$$
and $\varphi_1^*b_1^*h_1\in (H_2^p)^\perp$, it follows readily that
$$
B_r^*\begin{bmatrix}0&0\\K^*&-I_q\end{bmatrix}
\begin{bmatrix}h_1\\h_2\end{bmatrix}\in (H_2^m)^\perp,
$$
which justifies (4) with the help of a self-evident version of
Lemma \ref{lem:11.4}.

{\it Sufficiency.}
Let $h=\textup{col}(h_1,h_2)$ satisfy assumptions (1)--(4). Then
\[
    B_{\ell}\left[\begin{array}{cc}
  h_1-s_{12}h_2 \\
  -s_{22}h_2
\end{array}  \right]= B_{\ell}\left[\begin{array}{cc}
  h_1-Kh_2 \\
  0
\end{array}  \right]+  B_{\ell}\left[\begin{array}{cc}
  K-s_{12} \\
   -s_{22}
\end{array}  \right]h_2\in H_2^m,
\]
thanks to assumptions (1) and (3) and formula (\ref{eq:jun4e8}), whereas
\[
B_r^*\left[\begin{array}{cc}
  s_{11}^*h_1\\
  s_{12}^*h_1-h_2
\end{array}  \right]=B_r^*\left[\begin{array}{cc}
   0 \\
  K^*h_1-h_2
\end{array}  \right]+B_r^*\left[\begin{array}{c}
  s_{11}^* \\
  s_{12}^*-K^*
\end{array}  \right]h_1\in (H_2^m)^\perp
\]
by much the same sort of arguments. Thus, Theorem~\ref{cHW}
guarantees that $h\in\cK(W)$.
\end{proof}

\subsection{Description of $\cK(W)\cap L_2^m$.} In this subsection we supply a
description of the space $\cK(W)\cap L_2^m$ analogous to the
description of the space $\cH(W)\cap L_2^m$ that is presented in Section
5.14 of \cite{ArovD08}. The main formulas
(\ref{eq:jun3a9})--(\ref{eq:Pick2}) look the same as their
counterparts in the Hilbert space setting and the proofs of all but one of
them are also the same; only the verification (\ref{eq:jun3d9})
requires a different argument.

\begin{thm}
\label{thm:jun2a9}
Let $W\in\cU_\kappa^\circ(j_{pq})$, $\{b_1,b_2\}\in ap(W)$, let $K$ be
defined by (\ref{eq:jun10a9}) and let
$$
\Gamma_{11}:\,f\in H_2^q\longrightarrow P_{\cH(b_1)}Kf,\quad
%\textrm{and}\quad
\Gamma_{22}:\,f\in \cH_*(b_2)\longrightarrow
P_{(H_2^p)^\perp}Kf,
$$
$$
\textrm{and}\quad \Gamma_{12}:\,f\in \cH_*(b_2)\longrightarrow
P_{\cH(b_1)}Kf.
$$
Then:
\begin{equation}
\label{eq:jun3a9}
\cK(W)\cap H_2^m=\left\{\begin{bmatrix}u_1\\ \Gamma_{11}^*u_1\end{bmatrix}:\,
u_1\in\cH(b_1)\right\},
\end{equation}
\begin{equation}
\label{eq:jun3b9}
\cK(W)\cap (H_2^m)^\perp=\left\{\begin{bmatrix}\Gamma_{22}u_2\\ u_2
\end{bmatrix}:\,u_2\in\cH_*(b_2)\right\},
\end{equation}
\begin{equation}
\label{eq:jun3c9}
\cK(W)\cap L_2^m=(\cK(W)\cap H_2^m)\dot{+}(\cK(W)\cap (H_2^m)^\perp)
\end{equation}
and
\begin{equation}
\label{eq:jun3d9}
\langle h,h\rangle_{\cK(W)}=\left\langle \begin{bmatrix}I_p&-K\\-K^*&I_q
\end{bmatrix}h,h\right\rangle_{st}\quad\text{for every}\ h\in\cK(W)\cap L_2^m.
\end{equation}
Moreover, if
\[
h=\left[\begin{array}{cc}
  I   &   \Gamma_{22}\\
  \Gamma_{11}^* &   I
\end{array}  \right]u,\quad with \quad u=\left[\begin{array}{cc}
  u_1\\
  u_2
\end{array}  \right],\quad u_1\in\cH(b_1) \quad and \quad u_2\in\cH_*(b_2),
\]
then
\begin{equation}\label{eq:Pinner}
    \langle h,h\rangle_{\cK(W)}= \langle P u,u\rangle_{st}, \quad
\end{equation}
where
\begin{equation}\label{eq:Pick2}
   P=\left[\begin{array}{cc}
  I-\Gamma_{11}^*\Gamma_{11}   &   -\Gamma_{12}\\
  -\Gamma_{12}^* &   I-\Gamma_{22}^*\Gamma_{22}
\end{array}  \right]\end{equation}
\end{thm}

\begin{proof} Let $h=\textup{col}(h_1,h_2)\in\cK(W)\cap L_2^m$. Then,
in view of (1) and (2) of Theorem \ref{thm:11.11} and the identities
$\gb_\ell s_{22}=\varphi_2b_2$ and $s_{11}{\gb}_r=b_1\varphi_1$, it is
readily checked that
$b_2h_2\in H_2^q$ and $b_1^*h_1\in (H_2^p)^\perp$. Thus, if $h_1=h_1^++h_1^-$
with $h_1^+\in H_2^p$,  $h_1^-\in (H_2^p)^\perp$, and $h_2=h_2^++h_2^-$
with $h_2^+\in H_2^q$ and $h_2^-\in (H_2^q)^\perp$, then the condition
$$
b_1^*h_1\in (H_2^p)^\perp\Longrightarrow b_1^*h_1^+\in (H_2^p)^\perp
\Longrightarrow h_1^+\in\cH(b_1),
$$
whereas the condition
$$
b_2h_2\in H_2^q\Longrightarrow b_2h_2^-\in H_2^q
\Longrightarrow h_2^-\in\cH_*(b_2).
$$
Next, (3) and (4) of Theorem \ref{thm:11.11} imply that
$$
h_1^- - Kh_2^-\in H_2^p\quad\textrm{and}\quad -K^*h_1^+ +h_2^+\in
(H_2^q)^\perp.
$$
Therefore,
$$
h_1^-=P_-Kh_2^-=\Gamma_{22}h_2^-\quad\textrm{and}\quad h_2^+=P_+K^*h_1^+=
\Gamma_{11}^*h_1^+.
$$

%%%%%%%%%%%%%%%%%%%%%%%%%%%%%%%%%%%%%%%%%%%%%%%%%%%%%%%%%%%%%%%%%%%

Since $h\in L_2^m$ and $S^*S=I_m$ a.e. on $\Omega_0$,
\[
\left\| \left[\begin{array}{cc}
  I_p   &   -s_{12}\\
   0 &-s_{22}  \end{array}  \right]h
   \right\|^2_{st}=\left\langle \left[\begin{array}{cc}
  I_p   &   -s_{12}\\
  -s_{12}^* &   I_q
\end{array}  \right]h,h
\right\rangle_{st}
\]
and  formula~\eqref{cHWinnerPr} can be expressed as
\[
\langle h,h \rangle_{\cK(W)}=\left\langle \left[\begin{array}{cc}
  I_p   &   -s_{12}\\
  -s_{12}^* &   I_q
\end{array}  \right]h,h
\right\rangle_{st} -2\Re\left\langle \left[\begin{array}{cc}
  I_p   &   -s_{12}\\
   0 &-s_{22}  \end{array}  \right]h
 ,x
\right\rangle_{st} ,
\]
where $x\in\cH_*(B_\ell)$ is the unique solution of~\eqref{eq:Xl}.
Therefore, since $K=s_{12}+s_{11}d_\ell \varphi_2b_2$ by~\eqref{eq:11.41},
\[
\langle h,h \rangle_{\cK(W)}=\left\langle \left[\begin{array}{cc}
  I_p   &   -K\\
  -K^* &   I_q
\end{array}  \right]h,h
\right\rangle_{st} +2\Re Y,
\]
where
\[
\begin{split}
Y=&\left\langle \left[\begin{array}{cc}
  0   &   -s_{11}d_\ell\varphi_2b_2\\
  0   &   0
\end{array}  \right]h,h
\right\rangle_{st} -\left\langle \left[\begin{array}{cc}
  I_p   &   -s_{12}\\
   0    &   -s_{22}  \end{array}  \right]h
 ,x
\right\rangle_{st} \\
=&-\langle s_{11}d_\ell\varphi_2b_2h_2, h_1 \rangle_{st} -\langle h_1-Kh_2, h_1
\rangle_{st}\\
&+\langle s_{11}d_\ell\varphi_2b_2h_2, x_1 \rangle_{st} +\langle
s_{22}h_2, x_2
\rangle_{st}\\
=&\langle s_{11}d_\ell\varphi_2b_2h_2, h_1- x_1 \rangle_{st} +\langle
s_{22}h_2, x_2 \rangle_{st},
\end{split}
\]
because $h_1-Kh_2\in H_2^p$. Thus, as
$\varphi_2b_2= \gb_\ell
s_{22}$ by~\eqref{eq:0.5}, the last expression for $Y$
can be rewritten as
\[
Y=\langle \gb_\ell s_{22}h_2, d_\ell^* s_{11}^*(x_1-h_1)+\gb_\ell x_2
\rangle_{st}.
\]
Moreover, as~\eqref{eq:Xl} implies that
\[
s_{11}^*(x_1-h_1)+s_{21}^*x_2\in \left(H_2^p\right)^\perp
\]
and  $\gb_\ell s_{22}h_2\in H_2^p$,
\[
Y=\langle \gb_\ell s_{22}h_2, (\gb_\ell-d_\ell^* s_{21}^*) x_2
\rangle_{st}
=\langle (\gb_\ell^*-s_{21}d_\ell)\gb_\ell s_{22}h_2, x_2
\rangle_{st} =0,
\]
since
\[
\gb_\ell^*-s_{21}d_\ell=\gb_\ell^*-\gb_\ell^*\gs_\ell
d_\ell=\gb_\ell^*(I_p-\gs_\ell d_\ell)=c_\ell
\]
by~\eqref{eq:0.4} and \eqref{eq:jun1a8}, and
$c_\ell\gb_\ell s_{22}h_2\in H_2^q$. This completes the proof
of~\eqref{eq:jun3d9}.

Finally, formula~\eqref{eq:Pinner} follows from the evaluations
\begin{equation}
\label{eq:jun5a9}
\begin{split}
    \left\langle \begin{bmatrix}u_1\\ \Gamma_{11}^*u_1\end{bmatrix},
    \begin{bmatrix}u_1\\ \Gamma_{11}^*u_1\end{bmatrix}
    \right\rangle_{\cK(W)}
    &=\left\langle \begin{bmatrix}I_p & -K\\ -K^* & I_q\end{bmatrix}
    \begin{bmatrix}u_1\\ \Gamma_{11}^*u_1\end{bmatrix},
    \begin{bmatrix}u_1\\ \Gamma_{11}^*u_1\end{bmatrix}
    \right\rangle_{\cK(W)}\\
    &=\left\langle (I_p-\Gamma_{11}^*\Gamma_{11})u_1,u_1
    \right\rangle_{st},
    \end{split}
\end{equation}
\begin{equation}
\label{eq:jun5b9}
\begin{split}
    \left\langle \begin{bmatrix}\Gamma_{22}u_2\\ u_2
\end{bmatrix},
    \begin{bmatrix}\Gamma_{22}u_2\\ u_2
\end{bmatrix}
    \right\rangle_{\cK(W)}
    &=\left\langle \begin{bmatrix}I_p & -K\\ -K^* & I_q\end{bmatrix}
    \begin{bmatrix}\Gamma_{22}u_2\\ u_2
\end{bmatrix},
   \begin{bmatrix}\Gamma_{22}u_2\\ u_2
\end{bmatrix}
    \right\rangle_{\cK(W)}\\
    &=\left\langle (I_p-\Gamma_{22}\Gamma_{22}^*)u_2,u_2
    \right\rangle_{st}
    \end{split}
\end{equation}
and
\begin{equation}
\label{eq:jun5c9}
\begin{split}
    \left\langle \begin{bmatrix}\Gamma_{22}u_2\\ u_2
\end{bmatrix},
    \begin{bmatrix}u_1\\ \Gamma_{11}^*u_1\end{bmatrix}
    \right\rangle_{\cK(W)}
    &=\left\langle \begin{bmatrix}I_p & -K\\ -K^* & I_q\end{bmatrix}
    \begin{bmatrix}\Gamma_{22}u_2\\ u_2
\end{bmatrix},
    \begin{bmatrix}u_1\\ \Gamma_{11}^*u_1\end{bmatrix}
    \right\rangle_{\cK(W)}\\
    &=\left\langle -Ku_2,u_1
    \right\rangle_{st}=\left\langle -\Gamma_{12}u_2,u_1
    \right\rangle_{st}.
    \end{split}
\end{equation}
\end{proof}

\begin{corollary}\label{cor:jun7a9}
In the setting of Theorem~\ref{thm:jun2a9}, let
\[
\cH(b_1,b_2)=\begin{array}{c}\cH(b_1)\\ \oplus\\
\cH_*(b_2)\end{array}
\]
and let
the operator $F$ be given by
\begin{equation}\label{eq:F}
    F:\,u\in \cH(b_1,b_2)\mapsto Fu=\left[\begin{array}{cc}
  I   &   \Gamma_{22}\\
  \Gamma_{11}^* &   I
\end{array}  \right]u\in\cK(W).
\end{equation}
Then  $\cK(W)\cap L_2^m=F \cH(b_1,b_2)$ and the following equivalences hold:
\begin{enumerate}
\item[\rm(1)] $\cL(W):=\cK(W)\cap L_2^m$ is dense in $\cK(W)$ if and only if
$\ker P=\{0\}$;
\item[\rm(2)] $\cK(W)\subset L_2^m$  if and only if $P$ is a bounded
invertible operator with a bounded inverse.
\end{enumerate}
Moreover, If $P$ is a bounded invertible operator with a bounded inverse,
then $W\in \widetilde{L}_2^{\mtm}$.
\end{corollary}

\begin{proof}
Since the operator $F$ in~\eqref{eq:F} is injective, (1) follows immediately
from the equality
\begin{equation}\label{eq:FuFv}
    \langle Fu,Fv\rangle_{\cK(W)}=\langle Pu,v\rangle_{st} \quad
\textrm{for}\ u,v\in \cH(b_1,b_2).
\end{equation}

Suppose next that  $0\not\in\sigma(P)$. Then it follows from~\eqref{eq:FuFv}
that
\[
\nu_-(P)=\mbox{ind}_-(\cK(W)\cap L_2^m),
\]
and hence that $\cH(b_1,b_2)$ is a
Pontryagin space with respect to the inner product~\eqref{eq:FuFv} and that
$F$ is an isometry from this  Pontryagin space to a subspace $\cL(W)$ of
$\cK(W)$. Therefore, $\cL(W)$ is closed in $\cK(W)$. Thus,
\[
\cL(W)=\cK(W),
\]
since $\cL(W)$ is also dense in $\cK(W)$ by (1).

Conversely, if $\cL(W)=\cK(W)$, then it follows from~\eqref{eq:FuFv}
that $\cH(b_1,b_2)$ equipped with the inner product~\eqref{eq:FuFv} is
isometrically isomorphic to $\cK(W)$ and, hence, is a Pontryagin space. This
implies that $0\not\in\sigma(P)$.

Finally, the last assertion is immediate from (2) and formula (\ref{kerK}).
\end{proof}

We remark that if $W\in\cU_\kappa(j_{pq})$ with $\kappa=0$, then the class of
mvf's considered in (1) and (2) of the last corollary correspond to the
class of right regular and strongly right regular $j_{pq}$-inner mvf's
in \cite{ArovD08}.

%%%%%%%%%%%%%%%%%%%%%%%%%%%%%%%%%%%%%%%%%%%%%%%%%%%%%%%%%%%%%%%%%%%%%%%%%%%%%%
\subsection{Parametrization of the set $T_W[{\mathcal S}_{\kappa_2}^{\ptq}]
\cap {\mathcal S}_{\kappa_1+\kappa_2}^{\ptq}$.}
In this subsection we characterize the parameters
$\varepsilon\in {\mathcal S}_{\kappa_2}^{\ptq}$ for which
$T_W[\varepsilon]\in {\mathcal S}_{\kappa_1+\kappa_2}$ when
$W\in\cU_{\kappa_1}^\circ(j_{pq})$. The proof
is based on Theorem~\ref{thm:0.1}
augmented by the factorization result of Theorem~\ref{thm:11.8} and a
special case of Kre\u{\i}n-Langer generalization of Rouche's
Theorem, which is formulated below.
\begin{thm}\label{Rouche}
{\rm\cite{KL81}}
    Let $\varphi,\psi\in H_\infty^{q\times q}$, $\det(\varphi+\psi)
\not\equiv 0$ in $\Omega_+$, $M_\zeta(\varphi,\Omega_+)<\infty$ and
\begin{equation}\label{eq:8.3}
    \|\varphi(\mu)^{-1}\psi(\mu)\|\le 1\quad \textup{a.e. on }\Omega_0.
\end{equation}
Then $M_\zeta(\varphi+\psi,\Omega_+)\le M_\zeta(\varphi,\Omega_+)$
with equality if
\begin{equation}\label{eq:8.4}
(\varphi +\psi)^{-1}\varphi|_{\Omega_0}
\in \wtilde{L}_1^{q\times q}.
\end{equation}
\end{thm}
Theorem \ref{Rouche} is used to estimate the zero multiplicity of the
denominator in the linear-fractional transformation $T_W$ associated
with the mvf $W\in \cU_{\kappa_1}^\circ(j_{pq})$.
\begin{lem}\label{lem:5.3}
Let $W\in \cU^\circ_{\kappa_1}(j_{pq})$,  $S=PG(W)$, let $s_{21}$
have the Kre\u{\i}n-Langer factorizations~\eqref{eq:0.4},  and
assume that $\varepsilon\in {\mathcal S}^{\ptq}_{\kappa_2}$
satisfies the assumption
\begin{equation}\label{eq:8.4a}
    (w_{21}\varepsilon+w_{22})^{-1}w_{22}|_{\Omega_0} \in
\wtilde{L}_1^{q\times q}
\end{equation}
and let
\begin{equation}\label{KLepsilon}
\varepsilon=\theta^{-1}_\ell
 \varepsilon_\ell=\varepsilon_r\theta^{-1}_r
\end{equation}
denote its Kre\u{\i}n-Langer factorizations. Then
\begin{equation}\label{eq:8.5}
M_\zeta(\gb_\ell\theta_r-\gs_\ell\varepsilon_r,\Omega_+)=
M_\zeta(\theta_\ell\gb_r-\varepsilon_\ell\gs_r,\Omega_+)=\kappa_1+\kappa_2.
\end{equation}
Moreover, if $W\in\widetilde{L}_2^{\mtm}$, then (\ref{eq:8.4a}) holds for
every $\varepsilon\in\cS_{\kappa_2}^{\ptq}$.

\end{lem}
\begin{proof}
Since $s_{21}\in S_{\kappa_1}^{p\times q}$ one obtains
$M_\zeta(\gb_\ell,\Omega_+)=\kappa_1$. Moreover, as $b_\ell$ and $\theta_r$
are both finite Blaschke-Potapov products,
\[
M_\zeta(\gb_\ell\theta_r,\Omega_+)=M_\zeta(\gb_\ell,\Omega_+)+
M_\zeta(\theta_r,\Omega_+)=\kappa_1+\kappa_2.
\]
Therefore, since $\theta_r$ is unitary a.e. on $\Omega_0$, the identity
\begin{equation}\label{eq:7.52}
   \theta_r(\gb_\ell\theta_r-\gs_\ell\varepsilon_r)^{-1}\gb_\ell
=(w_{21}\varepsilon+w_{22})^{-1}w_{22} \quad \text{a.e. on}\
\Omega_0
\end{equation}
implies that
\[
\|(\gb_\ell\theta_r-\gs_\ell\varepsilon_r)^{-1}\gb_\ell\theta_r\|
=
\|\theta_r(\gb_\ell\theta_r-\gs_\ell\varepsilon_r)^{-1}\gb_\ell\|
= \|(w_{21}\varepsilon+w_{22})^{-1}w_{22}\|,\quad \text{a.e. on}\
\Omega_0,%\quad(\zeta\in\Omega_0).
\]
and hence that
$(\gb_\ell\theta_r-\gs_\ell\varepsilon_r)^{-1}\gb_\ell\theta_r|_{\Omega_0}
\in \wtilde{L}_1^{q\times q}$. Therefore, by Theorem~\ref{Rouche},
\[
M_\zeta(\gb_\ell\theta_r-\gs_\ell\varepsilon_r,\Omega_+)
=M_\zeta(\gb_\ell\theta_r,\Omega_+)=\kappa_1+\kappa_2.
\]
The proof of the second equality in~\eqref{eq:8.5} is similar.

To verify the final assertion, first note that if
$\varepsilon\in\cS_{\kappa_2}^{\ptq}$, then
$$
(w_{21}\varepsilon+w_{22})^{-1}w_{22}=(I_q-s_{21}\varepsilon)^{-1}
$$
and
\begin{eqnarray*}
\Vert (I_q-\varepsilon^* s_{21}^*)u\Vert&\ge& 1-\Vert s_{21}^*u
\Vert \ge \frac12 (1-\Vert s_{21}^*u\Vert^2)\\
&=&\frac12 u^*(I_q-s_{21}s_{21}^*)u=\frac12 u^*s_{22}s_{22}^*u
\end{eqnarray*}
for every $u\in\CC^q$ with $\Vert u\Vert=1$. Thus, the bound
$$
\Vert(I_q-s_{21}\varepsilon)^{-1}\Vert
=\Vert(I_q-\varepsilon^*s_{21}^*)^{-1}\Vert\le 2\Vert w_{22}^*w_{22}\Vert,
$$
implies that
$(I_q-s_{21}\varepsilon)^{-1}\in\widetilde{L}_1^{\qtq}$ when $W\in\widetilde{L}_2^{\mtm}$.
\end{proof}

\noindent{\bf Proof of Theorem \ref{thm:0.2}.} {\it Necessity.} Let
$\kappa=\kappa_1+\kappa_2$, let $s=T_W[\varepsilon]$
belongs to ${\mathcal S}_{\kappa}^{\ptq}$ and
let
\begin{equation}
\label{eq:may13a9}
\begin{bmatrix}G_\ell&-H_\ell\end{bmatrix}=
\begin{bmatrix}\theta_\ell&-\varepsilon_\ell\end{bmatrix}
\begin{bmatrix}\gb_r&-d_\ell\\\gs_r&c_\ell\end{bmatrix},\quad
\begin{bmatrix}H_r\\ G_r\end{bmatrix}=\begin{bmatrix}c_r&d_r\\-\gs_\ell
&\gb_\ell
\end{bmatrix}\begin{bmatrix}\varepsilon_r\\ \theta_r\end{bmatrix}
\end{equation}
and
\begin{equation}\label{G}
    G=(\varphi_{11}\varepsilon_r+\varphi_{12}\theta_r)
     (\varphi_{21}\varepsilon_r+\varphi_{22}\theta_r)^{-1}.
\end{equation}
Then
$$
G=\varphi_1 H_rG_r^{-1}\varphi_2,
$$
$M_\pi(s, \Omega_+)=\kappa$ and it follows
from~\eqref{eq:11.34}-\eqref{eq:11.36} that
\begin{equation}\label{eq:10.3}
s=T_W[\varepsilon]=T_\Theta[G]=b_1Gb_2+K.
\end{equation}
Therefore, since $K$ is holomorphic on $\Omega_+$,
\begin{equation}\label{MbGb}
M_\pi(b_1Gb_2,\Omega_+)=M_\pi(s,\Omega_+)=\kappa,
\end{equation}
 and  hence, by Lemma \ref{lem:5.3},
\begin{equation}\label{eq:10.5}
\kappa=M_\pi(b_1Gb_2,\Omega_+)\le M_\pi(G_r^{-1}\varphi_2b_2,\Omega_+)
\le M_\pi(G_r^{-1},\Omega_+) =\kappa.
\end{equation}
Thus,
\[
M_\pi(G_r^{-1}\varphi_2b_2,\Omega_+)=M_\pi(G_r^{-1},\Omega_+)=\kappa,
\]
and in view of  Proposition~\ref{Prop:5.1}, the factorization
$(\varphi_2b_2)^{-1}G_r$ is
coprime over $\Omega_+$.

Similarly, since
$
G_\ell H_r=H_\ell G_r,
%j_{pq}=W^\#j_{pq}W=\wtilde{\Phi}^\#\wtilde{\Theta}^\#j_{pq}\Theta\Phi=
%\wtilde{\Phi}^\#j_{pq}\Phi,
$
the mvf $G$ can be written as
\begin{equation}\label{eq:10.7}
G=\varphi_1 G_\ell^{-1}H_\ell\varphi_2
\end{equation}
and consequently ~\eqref{MbGb} and Lemma~\ref{lem:5.3} imply that
\[
M_\pi(b_1\varphi_1 G_\ell^{-1},\Omega_+)
=M_\pi(G_\ell^{-1},\Omega_+)=\kappa.
\]
Therefore,  Proposition~\ref{Prop:5.1} implies that the
factorization
$G_\ell(b_1\varphi_1)^{-1}$
is coprime over $\Omega_+$.

%%%%%%%%%%%%%%%%%%%%%%%%%%%%%%%%%%%%%%%%%%%%%%%%%%%%%%%%%%%%%%%%%%%%%%%%%%%%%%%%%%%%%%%
{\it Sufficiency.}
Since the assumptions of Lemma~\ref{lem:5.3} are satisfied,
\begin{equation}\label{eq:8.5a}
M_\pi(G_\ell^{-1},\Omega_+)=
M_\pi(G_r^{-1},\Omega_+)=\kappa_1+\kappa_2.
\end{equation}
Moreover, since  $E^{\pm 1}=\left[\begin{array}{cc}
                                        c_r & d_r \\
                                   -\gs_\ell & \gb_\ell \\
                        \end{array}\right]^{\pm 1}\in H_\infty^{m\times m}$,
the factorization $H_rG_r^{-1}$
is right coprime over $\Omega_+$, and hence, in view of
Proposition~\ref{Prop:5.1}
\[
M_\pi(H_rG_r^{-1},\Omega_+)= M_\pi(G_r^{-1},\Omega_+)=\kappa.
\]
It follows from the equality $ G=\varphi_1 H_rG_r^{-1}\varphi_2$
that
\[
M_\pi(G,\Omega_+)=M_\pi(H_rG_r^{-1},\Omega_+)=
M_\pi(G_r^{-1},\Omega_+)= M_\pi(G_r^{-1}\varphi_2,\Omega_+).
\]
Thus, by Lemma~\ref{lem:5.4}
\[
M_\pi(Gb_2,\Omega_+)= M_\pi(G_r^{-1}\varphi_2b_2,\Omega_+).
\]
Since the factorization in~\eqref{Reg2} is coprime one obtains from
Proposition~\ref{Prop:5.1} and~\eqref{eq:8.5a}
\[
M_\pi(Gb_2,\Omega_+)= M_\pi(G_r^{-1},\Omega_+)=\kappa.
\]

Similarly, Proposition~\ref{Prop:5.1} implies that
\[
M_\pi(G_\ell^{-1}H_\ell,\Omega_+)=M_\pi(G_\ell^{-1},\Omega_+)=\kappa.
%M_\pi((\theta_{\ell}\wt\varphi_{11}^\#+\varepsilon_{\ell}\wt\varphi_{12}^\#)^{-1},\Omega_+)=\kappa,
\]
since the factorization $G_\ell^{-1}H_\ell$ is coprime over $\Omega_+$.
Now the equality $ G=\varphi_1 G_\ell^{-1}H_\ell\varphi_2$,
Lemma~\ref{lem:5.4} and the assumption that the factorization
in~\eqref{Reg1} is coprime yield
\[
M_\pi(b_1G,\Omega_+)=
M_\pi(b_1\varphi_1G_\ell^{-1},\Omega_+)=M_\pi(G_\ell^{-1},\Omega_+)=\kappa.
\]
Therefore,
\[
M_\pi(b_1G,\Omega_+)=M_\pi(G,\Omega_+)=\kappa,
\]
which, with the help of Lemma~\ref{lem:5.4}, implies that
\[
M_\pi(b_1Gb_2,\Omega_+)=M_\pi(Gb_2,\Omega_+)=\kappa.
\]
In view of~\eqref{eq:10.3} $s\in {\mathcal S}_{\kappa}^{p\times q}$.
\hfill \qed

\begin{remark}
\label{rem:jun18a9}
If, in the setting of Theorem \ref{thm:0.2}, it is also assumed that
$W\in L_\infty^{\mtm}$, then:
\begin{enumerate}
\item[\rm(1)] Condition (\ref{eq:0.10}) is met by every mvf
$\varepsilon\in\cS_{\kappa_2}^{\ptq}$ (with $\kappa_2\ge 0$).
\item[\rm(2)] The entries $\varphi_{21}$ and $\varphi_{22}$ in the bottom
block row of $\Phi$ belong to $H_\infty^{\qtp}(\Omega_+)$ and
$H_\infty^{\qtq}(\Omega_+)$, respectively, whereas the entries
$\widetilde{\varphi}_{11}$ and $\widetilde{\varphi}_{12}$ in the
top block row of $\widetilde{\Phi}$ belong to $H_\infty^{\ptp}(\Omega_-)$
and $H_\infty^{\ptq}(\Omega_-)$, respectively. ($\Phi$ and
$\wt{\Phi}$ are defined in
Theorem \ref{thm:11.8}.)
\item[\rm(3)] The factorizations (\ref{Reg1}) and (\ref{Reg2}) can be
rexpressed  as
$$
\theta_\ell w^\#_{11}+\varepsilon_\ell w^\#_{12}
    =(\theta_\ell
\wt\varphi_{11}^\#+\varepsilon_\ell\wt\varphi_{12}^\#)b_1^{-1},
\quad\textrm{and}\quad
w_{21}\varepsilon_r+ w_{22}\theta_r=b_2^{-1}(\varphi_{21}
\varepsilon_r+\varphi_{22}\theta_r).
$$
Thus, $T_W[\varepsilon]\in\cS_{\kappa_1+\kappa_2}^{\ptq}$ if
and only if these two factorizations are coprime over $\Omega_+$.
\item[\rm(4)] If also $\kappa_2=0$, then $s=T_W[\varepsilon]$ belongs to
$\cS_{\kappa_1}^{\ptq}$ if
and only if the two factorizations
$$
    w^\#_{11}+\varepsilon w^\#_{12}
    =( \wt\varphi_{11}^\#+\varepsilon\wt\varphi_{12}^\#)b_1^{-1},
\quad and \quad
w_{21}\varepsilon+ w_{22}=b_2^{-1}(\varphi_{21}\varepsilon+\varphi_{22})
%\end{equation}
$$
are coprime over $\Omega_+$.
\end{enumerate}
\end{remark}


\begin{thebibliography}{ArovD97}
\bibitem{AAK71}
        V. M. Adamyan, D. Z.~Arov,  M. G.~Kre\u{\i}n, Analytic
        properties of the Schmidt pairs of a Hankel operator and the
generalized Schur-Takagi problem,
        Matem. Sb. {\bf 86} (1971), 34-75.
\bibitem{ADRS}
        D.~Alpay, A.~Dijksma, J.~Rovnyak and H. S. V.~de~Snoo,
        \textit{ Schur functions, operator colligations, and reproducing kernel
        Pontryagin spaces}, Oper. Theory: Adv. Appl., {\bf 96},
        Birkh\"auser Verlag, Basel, 1997.
\bibitem{ADRS1}
        D.~Alpay, A.~Dijksma, J.~Rovnyak and H. S. V.~de~Snoo,
        \textit{Realization and factorization in reproducing kernel
Pontryagin spaces}, in Oper. Theory Adv. Appl., {\bf 123} (2001), 43--65.
\bibitem{AD86}
D.~Alpay and H.~Dym, On applications of reproducing kernel spaces to the
Schur algorithm
and rational $J$ unitary factorization.
\textit{ I. Schur methods in operator theory and
signal processing}, 89--159, Oper. Theory Adv. Appl., {\bf 18},
Birkh\"auser, Basel, 1986.
\bibitem{Ar50}
    N. Aronszajn,
    Theory of reproducing kernels, Trans. Amer. Math. Soc., {\bf 68} (1950),
    337-404.
\bibitem{Arov93}
    D. Z.~Arov, The generalized bitangent Caratheodory-Nevanlinna-Pick
    problem, and $(j,J_0)$-inner matrix valued functions, Russian
    Acad. Sci. Izv. Math., {\bf 42} (1994), no.1, 1--26.
\bibitem{ArovD97}
    D. Z.~Arov and H. Dym, $J$-inner matrix functions, interpolation
    and inverse problems for canonical systems, I: Foundations,
Integr. Equat. Oper. Th., {\bf 29} (1997), 373--454.
\bibitem{ArovD08}
    D. Z.~Arov and H. Dym, \textit{ $J$-Contractive Matrix Valued Functions and
Related Topics}, Cambridge University Press, Cambridge 2008.
\bibitem{Arva}
    Z.~Arova,
    The functional model of a  $J$-unitary node with a given $j$-inner characteristic matrix function,
    Integral Equations Operator Theory, {\bf 28} (1997), 1-16.
\bibitem{AI}
        T. Ya.~Azizov and I. S.~Iokhvidov,
        \textit{Foundations of the theory of linear operators
        in spaces with an indefinite metric},
        Nauka, Moscow, 1986 (English translation:
        Wiley, New York, 1989).
\bibitem{BGR}
    J. A. Ball, I. Gohberg, and L. Rodman, \textit{Interpolation of rational
    matrix functions}, OT45, Birkh\"auser Verlag, 1990.
\bibitem{BH83}
    J. A. Ball and J. W. Helton,
    A Beurling--Lax theorem for the Lie group $U(m,n)$
    which contains most classical interpolation theory,
    J. Operator Theory, {\bf 9} (1983), 107--142.
\bibitem{Bo}
        J.~Bognar,
        \textit{Indefinite inner product spaces},
        Ergeb. Math. Grenzgeb., Bd. {\bf 78}, Springer-Verlag,
        New York-Heidelberg, 1974.
\bibitem{dB88}
    L.~de~Branges,
    Complementation in Kre\u{\i}n spaces, Trans. Amer. Math. Soc.,
    {\bf 305} (1988), 277-291.
\bibitem{dBR}
    L.~de~Branges and J.~Rovnyak,
    Canonical models in quantum scattering theory,
    In: \textit{"Perturbation Theory and its Application in Quantum Mechanics"},
    Wiley, New York, 1966, 359--391.
\bibitem{Br}
    M. S.~Brodski\u{\i},
    Unitary operator colligations and their characteristic functions,
    Uspekhi Mat.Nauk, {\bf 33}  (1978),  no.4, 141--168;
    Engl. transl.: Rus. Math. Surveys 33  (1978),  no.4, 159--191.
\bibitem{Der01}
        V. Derkach,
        On indefinite abstract interpolation problem,
        Methods of Funct. Analysis and Topology, {\bf 7} (2001), no.4, 87-100.
\bibitem{Der03}
        V. Derkach,
        On Schur-Nevanlinna-Pick indefinite interpolation problem,
       Ukrainian Math. Zh.,  {\bf 55} (2003), no. 10, 1567--1587.
\bibitem{Der01a}
V. Derkach, On characteristic functions of linear relations and unitary
colligations,
Dopov. Nats. Akad. Nauk Ukr. Mat.  (2001),  no. 11, 28--33.
\bibitem{DrR}
        M. A. Dritschel and J. Rovnyak, Extension theorems for contraction
        operators in Kre\u{\i}n spaces, Oper.Theory: Adv. Appl., {\bf 47},
        Birkh. Verlag, Basel, 1990. -- P. 221--305.
\bibitem{DLS}
        A.~Dijksma, H.~Langer and  H. S. V.~de~Snoo,
        Characteristic functions of unitary colligations in
    $\Pi_\kappa$--spaces, Oper. Theory: Adv. Appl., {\bf 19},
        Birkh\"auser Verlag, Basel, 1986, 125--194.
\bibitem{DSS}
    R. G. Douglas,  H. S. Shapiro and A. L.  Shields, Cyclic vectors and
invariant subspaces for the backward shift operator,
    Ann. Inst. Fourier (Grenoble)  20  (1970), 37--76.

\bibitem{Duren}
 P. L. Duren,    \textit{Theory of $H^p$ spaces},
  New York : Academic Press, 1970.

\bibitem{D89}
    H.~Dym,
    \textit{$J$--contractive matrix functions, reproducing kernel
Hilbert spaces and interpolation}, CBMS Regional Series in Math.,
    vol.71, Providence, RI, 1989.
\bibitem{D8}
H. Dym, Linear fractional transformations, Riccati equations and bitangential
interpolation, revisited. \textit{Reproducing kernel spaces and applications}, 171--212,
Oper. Theory: Adv. Appl., {\bf 143}, Birkh\"auser, Basel, 2003.

\bibitem{fran}
B. Francis,  \textit{ A Course in $H_\infty$ Control Theory}, Lecture Notes in Control
and Information Sciences, {\bf 88}. Springer-Verlag, Berlin, 1987.

\bibitem{Gant}
      L. B. Gantmacher, \textit{The Theory of Matrices}, Chelsea
Publishing Company,
      New-York, 1959.
\bibitem{Fuhr68}
  P. A. Fuhrmann, On the corona theorem and its application to
  spectral problems in Hilbert space, Trans. AMS {\bf 132}, 1968,
  55--66.
\bibitem{Fuhr74}
  P. A. Fuhrmann, On Hankel operator ranges, meromorphic
pseudo-continuations and factorizations
  of operator-valued analytic functions, J. London Math. Soc., 13 (1975),
  323--327.
\bibitem{KKhYu}
    V. E.~Katsnelson, A. Ya.~Kheifets and P. M.~Yuditskii,
    The abstract interpolation problem and extension theory of isometric
    operators, in: \textit{"Operators in Spaces of Functions and Problems in
    Function Theory"}, Kiev, Naukova Dumka, 1987, 83--96 (Russian).
\bibitem{Ki}
        H. Kimura, \textit{Chain scattering approach to $H^\infty$
        control}, Birkh\"auser, Boston, 1997.
\bibitem{KL}
         M. G.~Kre\u{\i}n and H. Langer,
    \"Uber die verallgemeinerten Resolventen und die characteristische
    Function eines isometrischen Operators im Raume $\Pi_\kappa$,
    Hilbert space Operators and Operator Algebras
    (Proc.Intern.Conf.,Tihany, 1970);
    Colloq.Math.Soc.Janos Bolyai,
    {\bf 5}, North--Holland, Amsterdam, 353--399, 1972.
\bibitem{KLSz}
        M. G. Kre\u{\i}n and  H. Langer,
        \"Uber die $Q$--functions eines $\pi$--hermiteschen Operators im
        Raume $\Pi_{\kappa}$, Acta Sci.Math. (Szeged),
        {\bf 34} (1973), 191--230.
\bibitem{KL81}
        M. G. Kre\u{\i}n  and H. Langer,
        Some propositions of analytic matrix functions related to the
        theory of  operators in the space $\Pi_{\kappa}$,
        Acta Sci.Math.Szeged, {\bf 43} (1981), 181--205.
\bibitem{Kuzh}
        A. V. Ku\v zel' , Spectral analysis of unbounded
non-selfadjoint operators in a space with indefinite metric.
(Russian) Dokl. Akad. Nauk SSSR, {\bf 178} (1968), 31--33.
\bibitem {Nud81}
A. A. Nudelman, On a generalization of classical interpolation
problems, Dokl.Akad.Nauk SSSR, {\bf 256} (1981), 790-793.

\bibitem{Pot}
        V. P.~Potapov,
        Multiplicative structure of $J$-nonexpanding matrix functions,
        Trudy Mosk.Matem. Obsch., {\bf 4}, (1955) 125--236.
\bibitem{Sch}
    L. Schwartz,
    Sous espaces hilbertiens d'espaces vectoriels topologiques et
    noyaux associes, J.~Analyse Math. {\bf 13} (1964), 115-256.

\end{thebibliography}
\end{document}